\documentclass[11pt]{article}
\usepackage{amsfonts}
\usepackage{graphics}
\usepackage{indentfirst}
\usepackage{cite}
\usepackage{latexsym}
\usepackage{amsmath}
\usepackage{amssymb}
\usepackage[dvips]{epsfig}
\usepackage{amscd}
\usepackage[bookmarks,colorlinks,citecolor=blue,linkcolor=red]{hyperref}
\clearpage
\def\on{\bar\rho}
\allowdisplaybreaks
\newtheorem{theorem}{Theorem}[section]
\newtheorem{remark}{Remark}[section]

\newtheorem{lemma}[theorem]{Lemma}

\newtheorem{proposition}[theorem]{Proposition}

\newcommand{\n}{\rho}

\newcommand{\lm}{\lambda}

\newcommand{\ltwo}{_{L^2}^2}

\renewcommand{\div}{ {\rm div }  }

\newcommand{\pa}{\partial}
\renewcommand{\r}{\mathbb{R}}

\newcommand{\ia}{\int_0^T}

\newcommand{\bt}{\begin{theorem}}
\newcommand{\bl}{\begin{lemma}}
\newcommand{\el}{\end{lemma}}
\newcommand{\et}{\end{theorem}}
\newcommand{\ga}{\gamma}

\newcommand{\curl}{{\rm curl} }

\newcommand{\de}{\delta}
\newcommand{\ve}{\varepsilon}

\newcommand{\ol}{\overline}

\newcommand{\bn}{\begin{eqnarray}}
\newcommand{\en}{\end{eqnarray}}
\newcommand{\bnn}{\begin{eqnarray*}}
\newcommand{\enn}{\end{eqnarray*}}

\newcommand{\bnnn}{\begin{eqnarray*}}
\newcommand{\ennn}{\end{eqnarray*}}

\newcommand{\ba}{\begin{aligned}}
\newcommand{\ea}{\end{aligned}}
\newcommand{\be}{\begin{equation}}
\newcommand{\ee}{\end{equation}}
\def\O{{\Omega }}

\def\norm[#1]#2{\|#2\|_{#1}}

\newcommand{\si}{\sigma}

\def\na{\nabla}
\def\on{\bar\n}

\def\QEDopen{{\setlength{\fboxsep}{0pt}\setlength{\fboxrule}{0.2pt}\fbox{\rule[0pt]{0pt}{1.3ex}\rule[0pt]{1.3ex}{0pt}}}} 

\def\QED{\QEDopen} 

\def\endproof{\hspace*{\fill}~\QED\par\endtrivlist\unskip}

\makeatletter      
\@addtoreset{equation}{section}
\makeatother       

\title{Global Strong Solutions to the Compressible Magnetohydrodynamic Equations with Slip Boundary Conditions in 3D Bounded Domains}

\author{Yazhou C{\small HEN}, Bin H{\small UANG}, Xiaoding S{\small HI}   \\[3mm]
{\normalsize   College of Mathematics and Physics, }\\ {\normalsize  Beijing University of Chemical Technology, Beijing 100029, P. R. China} }

\date{ }

\begin{document}
\maketitle

\begin{abstract}
We deal with  the barotropic compressible magnetohydrodynamic equations in three-dimensional (3D) bounded domain with slip boundary condition and vacuum. By a series of a priori estimates, especially the boundary estimates, we prove the global well-posedness of classical solution and the exponential decay rate to the initial-boundary-value problem of this system for the regular initial data with small energy but possibly large oscillations. The initial density of such a classical solution is allowed to contain vacuum states. Moreover, it is also shown that the oscillation of the density will grow unboundedly with an exponential rate when the initial state contains vacuum.
\end{abstract}

\textbf{Keywords:} compressible magnetohydrodynamic equations;  global classical existence; priori estimates; slip boundary condition; vacuum.

\textbf{AMS subject classifications:} 35Q55, 35K65, 76N10, 76W05

\section{Introduction}
In this paper, we consider the viscous compressible magnetohydrodynamic (MHD) equations for barotropic flows in a domain $\Omega\subset\r^{3}$, which can be written as
\begin{equation}\label{CMHD}
\begin{cases}
\rho_t+ \mathop{\mathrm{div}}\nolimits(\rho u)=0,\\
(\rho u)_t+\mathop{\mathrm{div}}\nolimits(\rho u\otimes u)+\nabla P
=\mu \Delta u+(\mu+\lambda)\nabla \mathop{\mathrm{div}}\nolimits u +(\nabla\times H)\times H,\\
H_t -\nabla \times (u \times H)=-\nu \nabla \times (\nabla \times H),
\\
\mathop{\mathrm{div}}\nolimits H=0,
\end{cases}
\end{equation}
where $(x,t)\in\Omega\times (0,T]$, $t\geq 0$ is time, and $x=(x_1,x_2,x_3)$ is the spatial coordinate. The unknown functions $\rho, u=(u^1,u^2,u^3), P=P(\rho),$ and $H=(H^1,H^2,H^3)$ denote the fluid density, velocity, pressure and magnetic field, respectively. Here we consider the barotropic flows with $\gamma$-law pressure $P(\rho)=a\rho^{\gamma},$ where $a>0$ and $\gamma >1$ are some physical parameters.
The constants $\mu$ and $\lambda$ are the shear viscosity and bulk coefficients respectively satisfying the following physical restrictions
$\mu>0$ and $2\mu +{3}\lambda\geq 0$.
The constant $\nu >0$ is the resistivity coefficient which is inversely proportional to the electrical conductivity constant and acts as the magnetic diffusivity of magnetic fields.
In addition, the system is solved subject to the given initial data
\begin{equation}\label{initial}
\displaystyle  \rho(x,0)=\rho_0(x), \quad \rho u(x,0)=\rho_0 u_0(x),\quad H(x,0)=H_0(x),\quad x\in \Omega,
\end{equation}
and slip boundary conditions
\begin{align}
& u\cdot n=0,\,\,\,\curl u\times n=0, &\text{on} \,\,\,\partial\Omega, \label{navier-b}\\
& H \cdot n=0,\,\,\,\curl H\times n=0,  &\text{on} \,\,\,\partial\Omega,\label{boundary}
\end{align}
where $n$ is the unit outward normal vector to $\partial \Omega$.

The choice of boundary conditions is very important for hydrodynamics.
For the velocity field, one of the well-accepted choices is the no-slip boundary condition (i.e. Dirichlet boundary condition), which has been successfully applied to many hydrodynamical problems on the macroscopic scale.
However, experimental studies \cite{zg2001} reveal that for flows on the micro-andnanoscale, the empirical non-slip boundary condition may break down, depending on the interfacial roughness and in-terfacial interactions between solids and fluids.
Another choice is the slip boundary condition, which has different behaviors between the macro-scopic scale and microscale. The earliest slip boundary condition is proposed by Navier in \cite{Nclm1}, which indicates that there is a stagnant layer of fluid close to the wall allowing a fluid to slip and the slip velocity is proportional to the shear stress, that is
\begin{align}\label{Navi2}
\displaystyle  u \cdot n = 0, \,\,(D(u) n)_{\tau}+ \vartheta u_{\tau}=0 \,\,\,\text{on}\,\,\, \partial\Omega,
\end{align}
where $D(u) = (\nabla u+(\nabla u)^*)/2$ is the shear stress, $\vartheta$ is a scalar friction function, subscript $\tau$ denotes the tangential component on $\partial\Omega$.
With the development of micro/nano test technologies and molecular-dynamic simulation technology, the Navier-type slip boundary has been more concerned and well studied in numerical studies and analysis for various fluid mechanical problems,
see, for instance \cite{cl2019,cf1988,Ho3,hbdv3,Itt2,z1998,xx2007,xxw2009} and the references therein.
Additional, as shown in \cite{cl2019,xx2007}, the Navier-slip condition \eqref{Navi2} is written to the following generalized one
\begin{align}\label{navi1}
u\cdot n=0,\,\,\,\curl u\times n=-Au \,\,\,&\text{on} \,\,\,\partial\Omega,
\end{align}
where $A$ is a smooth symmetric matrix defined on $\partial \Omega$, especially when $A=0$, it is strongly related to \eqref{Navi2}. Hence, the boundary condition \eqref{navier-b} presented in this paper can be regarded as a Navier-type slip boundary condition.
For the magnetic field, the boundary condition \eqref{boundary} describes that the boundary $\partial \Omega$ is a perfect conductor (see \cite{fy2009,xxw2009}), that means the magnetic field is confined inside and separated from the exterior.  We also observed that \eqref{boundary} is adaptable to the system since it ensured the boundary balance of the quantities on the boundary.
Therefore, it is appropriate to consider the compressible MHD equations with the boundary conditions \eqref{navier-b}-\eqref{boundary}.

The compressible MHD system \eqref{CMHD} has been attracted a lot of attention of physicists and mathematicians due to its physical importance and mathematical challenges, including the strong coupling and interplay interaction between fluid motion and magnetic field, and significant progress has been made in the analysis of the well-posedness and dynamic behavior to the solutions of the system, see, for example, \cite{cw2002,cw2003,djj2013,df2006,fjn2007,fy2008,fy2009,hhpz,hw2008,hw2008-1,hw2009,hw2010,k1984,lxz2013,liu2015,lvh2015,lsx2016,tg2016,vk1972,wdh2003,xh2017,zjx2009,zz2010,zhu2015} and their references. Among them, we briefly review the results related to well-posedness of solutions for the multi-dimensional compressible MHD equation.
The local existence of strong solutions to the compressible MHD equations was obtained by Vol'pert and Hudjaev \cite{vk1972} for the Cauchy problem with large initial data and the initial density being strictly positive. Fan and Yu \cite{fy2009} extended the result to the case that the initial density may contain vacuum for the whole space or a bounded domain with non-slip boundary condition. Lv and Huang \cite{lvh2015} obtained the local existence of strong and classical solutions in $\r^2$ with vacuum as far field density.
The global existence of solutions to the compressible MHD equations has been studied in many works. Kawashima \cite{k1984} obtained the global existence of smooth solutions to the general electro-magneto-fluid equations in two dimensions when the initial data are small perturbations of a given constant state. Hu and Wang \cite{hw2008,hw2010} and Fan and Yu \cite{fy2008} proved the global existence of renormalized solutions to the compressible MHD equations for general large initial data. Recently, Li, Xu and Zhang \cite{lxz2013} established the global existence and uniqueness of classical solutions with constant state as far field in $\r^3$ with large oscillations and vacuum. Hong, Hou, Peng and Zhu \cite{hhpz} generalized the result for large initial data when $\gamma-1$ and $\nu^{-1}$ are suitably small. Lv, Shi and Xu \cite{lsx2016} got the global existence of unqiue classical solutions in two-dimensional space and obtained some better a priori decay with rates.

However, all of the above results only concern with the whole space or with non-slip boundary conditions.  It is rather complicated to investigate the well-posedness and dynamical behaviors of the compressible MHD system with slip boundary condition due to the compatibility issues of the nonlinear terms with the slip boundary conditions.
Tang and Gao \cite{tg2016} consider the local strong solutions to the compressible MHD equations with initial vacuum, in which the velocity field satisfies the Navier-slip condition. Considering the full compressible MHD system, Xi and Hao \cite{xh2017} proved the local existence of the classical solutions to the initial-boundary value problem with slip boundary condition for the full compressible MHD system without thermal conductivity, where the initial data contains vacuum and satisfies some initial layer compatibility condition.
However, to our best knowledge, whether the strong (classical) solution for general bounded smooth domains $\Omega$ with density containing vacuum initially to the MHD system exists globally in time is still open.
Recently, for the barotropic compressible Navier-Stokes equations in a bounded domain $\Omega$ with slip boundary condition, Cai and Li \cite{cl2019} proved that the classical solution of the initial-boundary-value problem exists globally with vacuum and small energy but possibly large oscillations and adopt some new techniques to obtain necessary a priori estimates, especially the boundary estimates, compared with the work in \cite{hlx1,jx01} for the Cauchy problem of the compressible Navier-Stokes equations. 
The main purpose of this paper is to establish the global well-posedness of classical solutions of the compressible MHD system \eqref{CMHD}-\eqref{boundary} in a bounded domain $\Omega\subset\r^3$.
We would like to obtain the time-independent upper bound of the density and the time-dependent higher-norm estimates of $(\rho,u,H)$ and extend the classical solution globally in time, which is motivated by the works of Cai and Li \cite{cl2019} and Li, Xu and Zhang \cite{lxz2013}.

Before formulating our main result, we first explain the notation and conventions used throughout the paper.
For integer $k\geq 1$ and $1\leq q<+\infty$, We denote the standard Sobolev space by $W^{k,q}(\Omega)$ and $H^k(\Omega)\triangleq W^{k,2}(\Omega)$. 
For some $s\in(0,1)$, the fractional Sobolev space $H^s(\Omega)$ is defined by
$$ H^s(\Omega)\triangleq\left\{u\in L^2(\Omega)~\text{:} \int_{\Omega\times\Omega}\frac{|u(x)-u(y)|^2}{|x-y|^{3+2s}}dxdy<+\infty\right\},\,\,\text{with the norm:}$$
$$\| u\|_{H^s(\Omega)}\triangleq \|u\|_{L^2(\Omega)}+\left(\int_{\Omega\times\Omega}\frac{|u(x)-u(y)|^2}{|x-y|^{3+2s}}dxdy\right)^\frac{1}{2}.$$
For simplicity, we denote $L^q(\Omega)$, $W^{k,q}(\Omega)$, $H^k(\Omega)$ and ${H^s(\Omega)}$ by $L^q$, $W^{k,q}$, $H^k$ and ${H^s}$ respectively, and set
$$\int fdx \triangleq \int_\Omega fdx,\quad \int_0^T\int fdxdt\triangleq\int_0^T\int_\Omega fdxdt. $$
For two $3\times 3$  matrices $A=\{a_{ij}\},\,\,B=\{b_{ij}\}$, the symbol $A\colon  B$ represents the trace of $AB^*$, where $B^*$ is the transpose of $B$, that is,
$$ A\colon  B\triangleq \text{tr} (AB^*)=\sum\limits_{i,j=1}^{3}a_{ij}b_{ij}.$$
Finally, for $v=(v^1,v^2,v^3)$, we denote $\nabla_iv\triangleq(\partial_iv^1,\partial_iv^2,\partial_iv^3)$ for $i=1,2,3,$ and the
material derivative of $v$ by  $\dot v\triangleq v_t+u\cdot\nabla v$.

The initial total energy of \eqref{CMHD} is defined as
\begin{align}\label{c0}
\displaystyle  C_0 =\int_{\Omega}\left(\frac{1}{2}\rho_0|u_0|^2 + G(\rho_0)+\frac{1}{2}|H_0|^2 \right)dx.
\end{align}
where
\begin{align}
\displaystyle  G(\rho)\triangleq\rho\int_{\bar{\rho}}^{\rho}\frac{P(s)-\bar{P}}{s^{2}} ds,\quad\bar{\rho}\triangleq\frac{1}{|\Omega|}\int_{\Omega}\rho_0 dx,\quad
\bar{P}\triangleq P(\bar{\rho}).
\end{align}

Now we can state our main result, Theorem \ref{th1}, concerning existence of global classical solutions to the problem  \eqref{CMHD}-\eqref{boundary}.
\begin{theorem}\label{th1}
Let $\Omega$ be a simply connected bounded domain in $\r^3$ and its smooth boundary $\partial\Omega$ has a finite number of 2-dimensional connected components. For $q\in (3,6)$ and some given constants $M_1,M_2>0$, $s\in (\frac{1}{2},1]$, and $\hat{\rho}\geq\bar{\rho}+1$ , and the initial data $(\rho_0,u_0,H_0)$ satisfy the boundary conditions \eqref{navier-b}-\eqref{boundary} and
\begin{align}
\displaystyle & 0\leq\rho_0\leq\hat{\rho},\quad
(\rho_0,P(\rho_0))\in H^2\cap W^{2,q}, \label{dt1}\\
& u_0\in  H^2 ,\quad H_0 \in  H^2 ,\quad \div H_0=0,\label{dt2}\\
& \|u_0\|_{H^s}\leq M_1,\quad \|H_0\|_{H^s}\leq M_2,\label{dt-s}
\end{align}
and the compatibility condition
\begin{align}\label{dt3}
\displaystyle  -\mu\triangle u_0-(\mu+\lambda)\nabla \mathop{\mathrm{div}}\nolimits u_0 + \nabla P(\rho_0)- (\nabla \times H_0) \times H_0 = \rho_0^{1/2}g,
\end{align}
for some  $ g\in L^2.$
Then there exists a positive constant $\ve$ depending only on  $\mu$, $\lambda$, $\nu$, $\ga$, $a$, $\on$, $\hat{\rho}$, $s$, $\Omega$, $M_1$ and $M_2$  such that the system \eqref{CMHD}-\eqref{boundary} has a unique global classical solution $(\rho,u,H)$ in $\Omega\times(0,\infty)$ satisfying
\begin{align}\label{esti-rho}
\displaystyle  0\le \rho(x,t)\le 2\hat{\rho},\quad  (x,t)\in \Omega\times(0,\infty),
\end{align}
\begin{equation}\label{esti-uh}
\begin{cases}
(\rho-\bar{\rho},P-\bar{P})\in C([0,\infty);H^2 \cap W^{2,q} ),\\
\nabla u\in C([0,\infty);H^1 )\cap  L^\infty_{\rm loc}(0,\infty;H^2\cap W^{2,q}),\\
u_t\in L^{\infty}_{\rm loc}(0,\infty; H^2)\cap H^1_{\rm loc}(0,\infty; H^1),\\
H \in C([0,\infty);H^2)\cap  L^\infty_{\rm loc}(0,\infty; H^4),\\
H_t\in C([0,\infty);L^2)\cap H^1_{\rm loc}(0,\infty; H^1)\cap L^\infty_{\rm loc}(0,\infty; H^2),	
\end{cases}
\end{equation}
provided the initial total energy $C_0\leq\ve$.
Moreover,  for any $r\in [1,\infty)$ and $p\in [1,6],$ there exist positive constants $C$ and $\eta_0$ depending only  on $\mu,$  $\lambda,$ $\nu,$  $\gamma,$ $a$, $s$, $\bar{\rho}$, $\hat{\rho}$, $\Omega$,   $M_1, M_2,$  $r$ and $p$  such that for $t>0,$
\begin{align}\label{esti-t}
\displaystyle  \|\rho(\cdot,t)-\bar\rho\|_{L^r}+\|u(\cdot,t)\|_{W^{1,p}} +\|(\sqrt{\rho}\dot{u})(\cdot,t)\|^2_{L^2}+\|H(\cdot,t)\|_{H^2}\leq Ce^{-\eta_0 t}.
\end{align}
\end{theorem}

Then, thanks to the exponential decay rate \eqref{esti-t}, taking the similar procedure as in \cite{cl2019,lx}, we can directly deduce the following large-time behavior of the gradient of the density when the initial density contains vacuum state.
\begin{theorem}\label{th2}
Under the conditions of Theorem \ref{th1}, assume further  that there exists some point $x_0\in \Omega$ such that $\rho_0(x_0)=0.$  Then the unique global classical solution $(\rho,u,H)$ to the problem \eqref{CMHD}-\eqref{boundary} obtained in
Theorem \ref{th1}  satisfies that for any $r_1>3,$   there exist positive constants $\tilde{C}_1$ and $\tilde{C}_2$ depending only  on $\mu$,  $\lambda$, $\nu$, $\gamma$, $a$, $s$, $\bar{\rho}$, $\hat{\rho}$,  $\Omega$, $M_1$, $M_2$ and  $r_1$   such that for any $t>0$,
\begin{align}\label{esti-2}
\displaystyle \|\nabla\rho (\cdot,t)\|_{L^{r_1}}\geq \tilde{C}_1 e^{\tilde{C}_2 t} .
\end{align}
\end{theorem}

\begin{remark}\label{rem:1} From Sobolev's inequality and \eqref{esti-uh}$_1$ with $q>3$, it follows that 
\begin{equation}
\displaystyle \rho, \nabla \rho \in C(\bar\Omega\times [0,T]). 
\end{equation}
Moreover, it also follows from \eqref{esti-uh}$_{2-5}$  that  
\begin{equation}
\displaystyle  u, H, \nabla u, \nabla H, \nabla^2 u,\nabla^2 H, u_t, H_t \in C(\bar\Omega\times [\tau,T]), 
\end{equation}
due to the following simple fact that $$L^2(\tau,T;H^1)\cap H^1(\tau,T;H^{-1})\hookrightarrow C([\tau,T];L^2).$$
Finally, by \eqref{CMHD}$_1,$ we have 
\begin{equation}
\displaystyle \rho_t=-u\cdot \nabla \rho-\rho\div u\in C(\bar\Omega\times [\tau,T]). 
\end{equation}
Hence the solution obtained in Theorem \ref{th1} becomes a classical one away
from the initial time.
\end{remark}

\begin{remark}\label{rem:2}
When we consider the general slip boundary \eqref{navi1} for the velocity field, and assume that the matrix $A$ is smooth and positive semi-definite, and even if the restriction on $A$ is relaxed to $A\in H^3$ and the negative eigenvalues of $A$ (if exist) are small enough, Theorem \ref{th1} will still hold. This can be achieved by a similar way as in \cite{cl2019}.
When $H = 0$, i.e., there is no electromagnetic field effect, the compressible MHD system \eqref{CMHD} turns to be the compressible Navier-Stokes equations, and Theorem \ref{th1} is the same as the result of Cai and Li \cite{cl2019}. Roughly speaking, we generalize the results of \cite{cl2019} to the compressible MHD equations.
\end{remark}

\begin{remark}\label{rem:3} When the initial state contains vacuum,
Theorem \ref{th2} implies that the oscillation of the density will grow unboundedly with an exponential rate, which is somewhat surprisingly compared with the Cauchy problem  \cite{lxz2013} where there is no results concerning the growth rate of the gradient of the density.
\end{remark}

\begin{remark}\label{rem:4}
In our case, $\Omega$ is a bounded domain in $\r^3$, for the small initial energy, we need the boundedness assumptions on the $H^s$-norm, for $s\in (1/2,1]$, of the initial velocity and magnetic field, which is analogous to the $\dot{H}^\beta$-norm in the whole space case \cite{lxz2013,jx01}. Thus, compared with the results in \cite{k1984}, the conditions on the initial velocity may be optimal under the smallness conditions on the initial energy.
\end{remark}

We now sketch the main idea used in the proof of Theorem \ref{th1}. Similar to the argument in \cite{cl2019,lxz2013}, the key issue in our proof is to derive the time-independent upper bound of the density in Proposition \ref{pr1} and the time-dependent higher-norm estimates of $(\rho,u,H)$. It is worth pointing out that the effective viscous flux $F$ and the vorticity $\omega$ (see \eqref{flux} for the definition) play an important role in the proof.
However, unlike \cite{lxz2013}, it can not get the $L^p$-norm ($2\leq p\leq 6$) of $\nabla u$ by the standard elliptic estimate, due to the bounded domain with the slip boundary condition \eqref{navier-b}. To deal with this difficulty, we consider the Petrovsky type Lam\'{e}'s system \eqref{lame1} (see \cite{adn}) and obtain the estimates of $W^{k,q}$-norm to the slip boundary condition (see Lemma \ref{lem-lame}). Thanks to \cite{ar2014,vww}, Lemma \ref{lem-vn} shows that the inequality $\|\nabla u\|_{L^q}\leq C(\|\div u\|_{L^q}+\|\curl u\|_{L^q})\,\,\,\text{for any} \,\,\,q>1$ holds  for $u\in W^{1,q}$ with $u\cdot n=0$ on $\partial\Omega$. These fact allows us to control $\nabla u$ by means of $\div u$ and $\curl u$. For the magnetic field, with the help of the magnetic diffusivity structure, we obtain the estimates of $\curl H$ and $\curl^2 H$ to control $\nabla H$ and $\nabla^2 H$, which can be used to deal with the strong coupling and interplay interaction between the fluid motion and the magnetic field, such as the magnetic force $(\nabla \times H)\times H$ and the convection term $\nabla \times (u\times H)$.
In addition, the slip boundary also makes the time-independent estimates of $A_1(T)$ and $A_2(T)$ more difficult. Our observation is that due to the boundary condition $u\cdot n=0$, which yields $u\cdot\nabla u\cdot n=-u\cdot\nabla n\cdot u$. This equality is the key to estimate the integrals on the boundary $\partial\Omega$ and we obtain the estimate of $\dot{u}$ and $\nabla \dot{u}$ (see \eqref{udot} and \eqref{tdudot}).

The rest of the paper is organized as follows.
In Section 2, we derive the elementary energy estimates for the system \eqref{CMHD}-\eqref{boundary} and some key a priori estimates.
Section 3 and Section 4 are devoted to deriving the necessary time-independent lower-order estimates and time-dependent higher-order estimates, which can guarantee the local classical solution to be a global classical one.
In Section 5, the proof of Theorem \ref{th1} will be completed.
In Appendix \ref{appendix-a}, we list some elementary inequalities and important lemmas that we use intensively in the paper.

\section{Preliminaries}\label{se2}
In this section, we derive the elementary energy estimates for the system \eqref{CMHD}-\eqref{boundary} and some key a priori estimates. Let $T>0$ be a fixed time and $(\rho,u,H)$ be a smooth solution to \eqref{CMHD}-\eqref{boundary} on $\Omega \times (0,T]$.
For $H$ and $u$ sufficiently smooth, there are some formulas based on $\div H =0$:
\begin{equation}\label{divh}
\begin{cases}
(\nabla\times H)\times H=\div(H\otimes H-\frac{1}{2}|H|^2I_{3})=H \cdot \nabla H-\frac{1}{2}\nabla|H|^2,\\		
\nabla\times (u\times H)=(H\cdot\nabla)u-(u\cdot\nabla)H-H\div u.
\end{cases}
\end{equation}
Then we rewrite \eqref{CMHD} in the following form:
\begin{equation}\label{CMHD1}
\begin{cases}
\rho_t+ \mathop{\mathrm{div}}\nolimits(\rho u)=0,\\
\rho u_t+\rho u \cdot \nabla u - (\lambda\! +\! 2\mu)\nabla\div u+\mu\nabla\!\times\!\omega + \nabla(P\!-\!\bar{P})=H \!\cdot\! \nabla H+\nabla \frac{|H|^2}{2},\\
H_t+u \cdot \nabla H- H \cdot \nabla u+ H \div u= -\nu \nabla \times \curl H,
\\
\mathop{\mathrm{div}}\nolimits H=0,
\end{cases}
\end{equation}
where we used the fact $-\Delta u=-\nabla\div u+\nabla\times\omega$ and $\omega \triangleq \nabla\times u, \curl H \triangleq \nabla\times H$.
Multiplying $\eqref{CMHD1}_1 $ by $G'(\rho)$, $\eqref{CMHD1}_2$ by $u$ and $\eqref{CMHD1}_3$ by $H$ respectively, integrating by parts over $\Omega$, summing them up, by \eqref{navier-b} and \eqref{boundary}, we have
\begin{align}\label{m2}
\displaystyle  &\left(\int \Big(G(\rho)+\frac{1}{2}\rho |u|^{2}+\frac{1}{2}|H|^{2}\Big)dx\right)_t + (\lambda + 2\mu)\int(\div u)^{2}dx \nonumber \\
 & + \mu\int|\omega|^{2}dx 
 + \nu\int|\curl H|^{2}dx =0,
\end{align}
which, integrated over $(0,T)$, leads to the following elementary energy estimates.
\begin{lemma}\label{lem-basic}
 Let $(\rho,u,H)$ be a smooth solution of \eqref{CMHD}-\eqref{boundary} on $\O \times (0,T]$. Then 
\begin{align}
\displaystyle  &\sup_{0\le t\le T}
\left(\frac{1}{2}\|\rho^{\frac{1}{2}}u\|_{L^2}^2+\|G(\rho)\|_{L^1}+\frac{1}{2}\|H\|_{L^2}^2\right)
\nonumber \\
 &+ \int_0^{T}(\lambda+2\mu)\|\div u\|_{L^2}^2 +\mu\|\omega\|_{L^2}^2+\nu \|\curl H\|_{L^2}^2)dt \le C_0.\label{basic1}
\end{align}
\end{lemma}

Next, similarly to the compressible Navier-Stokes equations, let us set
\begin{align}\label{flux}
F\triangleq(\lambda+2\mu)\text{div}u-(P-\bar{P})-\frac{1}{2}|H|^2,
\end{align}
where $F$ denotes the effective viscous flux, which plays an important role in our following analysis.
For $F$, $\omega$, $\nabla u$ and $\nabla H$, we give the following conclusion, which is a key to a priori estimates.
\begin{lemma}\label{lem-f-td}
Let $(\rho,u,H)$ be a smooth solution of \eqref{CMHD}-\eqref{boundary} on $\O \times (0,T]$. Then for any $p\in[2,6],\,\,1<q<+\infty,$ there exists a positive constant $C$ depending only on $p$, $q$, $\mu$, $\lambda$ and $\Omega$ such that
\begin{align}
\displaystyle & \|\nabla u\|_{L^q}\leq C(\|\div u\|_{L^q}+\|\omega\|_{L^q}),\label{tdu1}\\
& \|\nabla H\|_{L^q}\leq C\|\curl H\|_{L^q},\label{tdh1}\\
& \|\nabla F\|_{L^p}\leq C(\|\rho\dot{u}\|_{L^p}+\|H\! \cdot \nabla H\|_{L^p}),\label{tdf1}\\
& \|\nabla\omega\|_{L^p}\leq  C(\|\rho\dot{u}\|_{L^p}+\|H \cdot \nabla H\|_{L^p}+\|\nabla u\|_{L^2}),\label{tdxd-u1}
\end{align}
\begin{align}
\|F\|_{L^p}\leq & C(\|\rho\dot{u}\|_{L^2}+\|H \!\cdot \!\nabla H\|_{L^2})^{(3p-6)/(2p)}(\|\nabla u\|_{L^2}+\|P\!-\!\bar{P}\|_{L^2}+\|H\|^2_{L^4})^{(6-p)/(2p)} \nonumber \\
& +C(\|\nabla u\|_{L^2}+\|P\!-\!\bar{P}\|_{L^2}+\|H\|^2_{L^4}),\label{f-lp}\\
\|\omega\|_{L^p} \leq & C(\|\rho\dot{u}\|_{L^2}+\|H \cdot \nabla H\|_{L^2})^{(3p-6)/(2p)}\|\nabla u\|_{L^2}^{(6-p)/(2p)}+C\|\nabla u\|_{L^2},\label{xdu1}
\end{align}
Moreover,
\begin{align}
\|\nabla u\|_{L^p}\leq & C(\|\rho\dot{u}\|_{L^2}+\|H \!\cdot \!\nabla H\|_{L^2})^{(3p-6)/(2p)}(\|\nabla u\|_{L^2}+\|P\!-\!\bar{P}\|_{L^2}+\|H\|^2_{L^4})^{(6-p)/(2p)}\nonumber \\
&+C(\|\nabla u\|_{L^2}+\|P\!-\!\bar{P}\|_{L^p}+\||H|^2\|_{L^p}).\label{tdu2}
\end{align}
\end{lemma}
\begin{proof}
The inequality \eqref{tdu1}-\eqref{tdh1} is a direct result of Lemma \ref{lem-vn}, since $u\cdot n=0, H \cdot n= 0$ on $\partial\Omega$.
Moreover, noticing that \eqref{CMHD}$_3$ and $H \cdot n=0, \curl H \times n=0$ on $\partial \Omega$, by Lemma \ref{lem-vn}-\ref{lem-curl}, for any integer $k\geq 1$,  we obtain
\begin{align}\label{tdhk}
\|H\|_{W^{k+1,q}}\leq C \|\curl H\|_{W^{k,q}} \leq C(\|\curl^2 H\|_{W^{k-1,q}}+\|\curl H \|_{L^p}),
\end{align}
where $\curl^2 H \triangleq \curl\curl  H$ and we have used the fact $\div\curl H=0$.

By \eqref{CMHD}$_2$, \eqref{divh} and the slip boundary condition \eqref{navier-b}, one can find that the viscous flux $F$ satisfies
\begin{equation}
\begin{cases}
\Delta F=\div(\rho\dot{u}-H \cdot \nabla H)~~ &in\,\,\Omega,\\ \frac{\partial F}{\partial n}=(\rho\dot{u}-H \cdot \nabla H 
)\cdot n\,\, &on\,\, \partial\Omega.
\end{cases}
\end{equation}
It follows from Lemma 4.27 in \cite{ns2004} that
\begin{align}\label{tdf2}
\displaystyle  \|\nabla F\|_{L^q}\leq C(\|\rho\dot{u}\|_{L^q}
+\|H \cdot \nabla H\|_{L^q}
),
\end{align}
which gives  \eqref{tdf1}.
Moreover, for any integer $k\geq0$,
\begin{equation}\label{tdfk}
\|\nabla\! F\|_{W^{k+1,q}}\!\leq\! C(\|\rho\dot{u}\|_{L^q}\!\!+\!\|H\!\! \cdot\! \nabla\! H\|_{L^q}\!\!+\!\|\nabla\!(\rho\dot{u})\|_{W^{k,q}}\!\!+\!\|\nabla\!(H\!\! \cdot\! \nabla\! H)\|_{W^{k,q}}
).
\end{equation}

On the other hand, one can rewrite $\eqref{CMHD}_2 $ as
\begin{equation}\label{2m2}
\displaystyle  \mu\nabla\times\omega=\nabla F-\rho\dot{u}+H \cdot \nabla H.
\end{equation}
Noticing that $\omega\times n=0$ on $\partial\Omega$ and $\mathop{\mathrm{div}}\nolimits \omega=0$, by Lemma \ref{lem-curl}, we get
\begin{align}
\displaystyle  \|\nabla\omega\|_{L^q}\leq C(\|\nabla\times\omega\|_{L^q}+\|\omega\|_{L^q}) \leq C(\|\rho\dot{u}\|_{L^q}+\|H \cdot \nabla H\|_{L^q}+\|\omega\|_{L^q}),\label{tdxd-u2}
\end{align}
and for any integer $k\geq0$,
\begin{align}\label{tdxdk}
&\quad\|\nabla\omega\|_{W^{k+1,q}}\leq C(\|\nabla\times\omega\|_{W^{k+1,q}}+\|\omega\|_{L^q})\nonumber \\
&\leq C(\|\rho\dot{u}\|_{L^q}+\|H \cdot \nabla H\|_{L^q}+\|\nabla(\rho\dot{u})\|_{W^{k,q}}+\|\nabla(H \cdot \nabla H)\|_{W^{k,q}}+\|\omega\|_{L^q}),
\end{align}
where we have taken advantage of \eqref{tdf2} and \eqref{tdfk}.
By Sobolev's inequality and \eqref{tdxd-u2}, for $p\in[2,6]$,
\begin{align}
\|\nabla\omega\|_{L^p}&\le C(\|\rho\dot{u}\|_{L^p}+\|H \cdot \nabla H\|_{L^p}+\|\omega\|_{L^p} ) \nonumber \\
&\le C(\|\rho\dot{u}\|_{L^p}+\|H \cdot \nabla H\|_{L^p}+\|\rho\dot{u}\|_{L^2}+\|H \cdot \nabla H\|_{L^2}+\|\omega\|_{L^2}) \nonumber \\
&\le C(\|\rho\dot{u}\|_{L^p}+\|H \cdot \nabla H\|_{L^p}+\|\nabla u\|_{L^2}),
\end{align}
which implies \eqref{tdxd-u1}.

Furthermore, one can deduce from \eqref{g1} and \eqref{tdf1} that for $p\in[2,6]$,
\begin{align}\label{f-lp1}
\|F\|_{L^p}\leq& C\|F\|_{L^2}^{(6-p)/(2p)}\|\nabla F\|_{L^2}^{(3p-6)/(2p)}+C\|F\|_{L^2}\nonumber \\
\leq& C(\|\rho\dot{u}\|_{L^2}+\|H\! \cdot\! \nabla H\|_{L^2})^{(3p-6)/(2p)}(\|\nabla u\|_{L^2}+\|P\!-\!\bar{P}\|_{L^2}+\|H\|^2_{L^4})^{(6-p)/(2p)}\nonumber\\
&+C(\|\nabla u\|_{L^2}+\|P\!-\!\bar{P}\|_{L^2}+\|H\|^2_{L^4}),
\end{align}
similarly, by \eqref{g1} and \eqref{tdxd-u2},
\begin{align}\label{xdu-lp1}
\|\omega\|_{L^p}\leq& C\|\omega\|_{L^2}^{(6-p)/(2p)}\|\nabla \omega\|_{L^2}^{(3p-6)/(2p)}+C\|\omega\|_{L^2}\nonumber \\
\leq& C(\|\rho\dot{u}\|_{L^2}+\|H \cdot \nabla H\|_{L^2}+\|\nabla u\|_{L^2})^{(3p-6)/(2p)}\|\nabla u\|_{L^2}^{(6-p)/(2p)}+C\|\nabla u\|_{L^2}, \nonumber \\
\leq& C(\|\rho\dot{u}\|_{L^2}+\|H \cdot \nabla H\|_{L^2})^{(3p-6)/(2p)}\|\nabla u\|_{L^2}^{(6-p)/(2p)}+C\|\nabla u\|_{L^2},
\end{align}
and so \eqref{f-lp}-\eqref{xdu1} are established.
By virtue of \eqref{tdu1}, \eqref{f-lp}, \eqref{xdu1}, \eqref{tdf1} and \eqref{tdxd-u2}, it implies that \eqref{tdu2} holds. This completes the proof.
\end{proof}

\begin{remark}\label{rem:energy}
It is easy to check that there exists a positive constant $C$ depending only on $a, \gamma, \hat{\rho}, \bar{\rho}$ such that
\begin{align}\label{grho}
\displaystyle  C^{-1}(\rho-\bar{\rho})^{2}\leq G(\rho)\leq C(\rho-\bar{\rho})^{2},
\end{align}
which together with \eqref{basic1}, \eqref{tdu1} and  \eqref{tdh1} gives
\begin{align}\label{basic2}
\displaystyle \sup_{0\le t\le T}\|\rho-\bar{\rho}\|_{L^2}^2+\int_{0}^{T}(\|\nabla u\|_{L^{2}}^{2}+\|\nabla H\|_{L^{2}}^{2})dt\leq CC_{0}.
\end{align}
\end{remark}

\begin{remark}
From \eqref{tdhk}, we can get the estimate of $\|\nabla^{2} H\|_{L^p}$ and $\|\nabla^{3} H\|_{L^p}$ for $p\in [2,6]$,
\begin{align}\label{2tdh}
\|\nabla^2 H\|_{L^p}\leq C \|\curl H\|_{W^{1,p}} \leq C(\|\curl^2  H\|_{L^p}+\|\curl H \|_{L^p}),
\end{align}
and
\begin{align}\label{3tdh}
\|\nabla^3 H\|_{L^p}\leq C \|\curl H\|_{W^{2,p}} \leq C(\|\curl^2  H\|_{W^{1,p}}+\|\curl H \|_{L^p}).
\end{align}
On the other hand, we can get the estimates of $\|\nabla^{2}u\|_{L^p}$ and $\|\nabla^{3}u\|_{L^p}$ for $p\in[2,6]$ by Lemma \ref{lem-vn}, which will be devoted to giving higher order estimates in Section \ref{se4}. In fact, by Lemma \ref{lem-vn}, \eqref{tdf2} and \eqref{tdfk}, for $p\in[2,6]$,
\begin{align}\label{2tdu}
&\quad\|\nabla^{2}u\|_{L^p}\leq C(\|\div u\|_{W^{1,p}}+\|\omega\|_{W^{1,p}})\nonumber \\
&\leq C(\|\rho\dot{u}\|_{L^p}+\|H \cdot \nabla H\|_{L^p}+\|\nabla P\|_{L^p}+\|P-\bar{P}\|_{L^p} \nonumber \\
&\quad +\||H|^2\|_{L^p}+\|\nabla H \cdot H\|_{L^p}+\|\nabla u\|_{L^2}),
\end{align}
and
\begin{align}\label{3tdu}
&\quad\|\nabla^{3}u\|_{L^p}\leq C(\|\div u\|_{W^{2,p}}+\|\omega\|_{W^{2,p}})\nonumber \\
&\leq C(\|\nabla(\rho\dot{u})\|_{L^p}+\|\nabla(H \cdot \nabla H)\|_{L^p}+\|\nabla^2 P\|_{L^p}+\|\nabla(\nabla H \cdot H)\|_{L^p}+\|H \cdot \nabla H\|_{L^p}\nonumber \\
&\quad +\|\rho\dot{u}\|_{L^p}+\|\nabla P\|_{L^p}+\|P-\bar{P}\|_{L^p} +\||H|^2\|_{L^p}+\|\nabla H \cdot H\|_{L^p}+\|\nabla u\|_{L^2}).
\end{align}
\end{remark}

The lemma below gives an a priori estimate on the $L^2(\Omega\times(0,T))$-norm of  $ \rho-\bar{\rho} .$ 
\begin{lemma} \label{lem-rho2}
Let $(\rho,u,H)$ be a smooth solution of \eqref{CMHD}-\eqref{boundary} on $\O \times (0,T]$. Then there exists a positive constant $C$ depending only on $p$, $q$, $\mu$, $\lambda$, $\nu$ and $\Omega$ such that
\begin{align}\label{rho2-int}
\displaystyle  \int_0^T\int(\rho-\bar{\rho})^{2}dxdt\leq CC_0.
\end{align}
\end{lemma}
\begin{proof}
From Lemma \ref{lem-divf}, multiplying $\eqref{CMHD}_2$ by $\mathcal{B}[\n-\bar\n]$ and integrating over $\Omega,$ one has
\begin{align}\label{p-time}
&\int(P-\bar{P})(\n-\bar\n) dx \nonumber \\
= &\left(\int\rho u\cdot\mathcal{B}[\n-\bar\n] dx\right)_t-\int\rho u\cdot\nabla\mathcal{B}[\n-\bar\n]\cdot udx - \int\rho u\cdot\mathcal{B}[\n_t]  dx  \nonumber \\
& +\mu\int\nabla u\cdot\nabla\mathcal{B}[\n-\bar\n] dx + (\lambda+\mu)\int(\rho-\bar{\rho})\div udx \nonumber \\
& -\int H\cdot\nabla\mathcal{B}[\n-\bar\n]\cdot H dx - \int(\rho-\bar{\rho}) |H|^2/2 dx \nonumber \\
\leq &  \left(\int\rho u\cdot\mathcal{B}[\n-\bar\n] dx\right)_t+C\|\rho^{\frac{1}{2}}u\|_{L^{4}}^{2}\|\n-\bar\n\|_{L^2} +C\|\rho u\|_{L^2}^2 \nonumber \\
& +C\|\rho-\bar{\rho}\|_{L^2}\|\nabla u\|_{L^2}+C\|\rho-\bar{\rho}\|_{L^2}\|H\|_{L^4}^2 \nonumber \\
\leq & \left(\int\rho u\cdot\mathcal{B}[\n-\bar\n] dx\right)_t+\de \|\n-\bar\n\|_{L^2}^2 +C(\de)(\|\na u\|_{L^2}^2+\|\na H\|_{L^2}^2),
\end{align}
integrating it over $(0,T]$ together with \eqref{basic2} gives
\begin{align}
\displaystyle  \int_0^T\int(\rho-\bar{\rho})^{2}dxdt 
& \leq C\Big(\sup_{   0\le t\le T  }(\|\rho^{\frac{1}{2}}u\|_{L^{2}}^{2}+\|\rho-\bar{\rho}\|_{L^{2}}^{2})+\int_0^{T}(\|\nabla u\|_{L^{2}}^{2}+\|\nabla H\|_{L^{2}}^{2})dt\Big) \nonumber\\
& \leq CC_0,
\end{align}
and we finish the proof.
\end{proof}

Next we need the estimates on the material derivative of $u$. Since $u\cdot n=0$ on $\partial\Omega$, it follows that
\begin{align}\label{bdd2}
\displaystyle u\cdot\nabla u\cdot n=-u\cdot\nabla n\cdot u,
\end{align}
which implies
\begin{align}\label{bdd3}
\displaystyle (\dot{u}+(u\cdot\nabla n)\times u^{\perp})\cdot n=0 \mbox{ on } \partial \Omega  ,
\end{align}
where $u^{\perp}\triangleq -u\times n$ on $\partial\Omega$. We review the following Poincare-type inequality of $\dot u$, which depends on this observation (see \cite{cl2019}, Lemma 3.2).
\begin{lemma}\label{lem-ud}
If $(\rho,u,H)$ is a smooth solution of \eqref{CMHD} with slip condition \eqref{navier-b}-\eqref{boundary}, then there exists a positive constant $C$ depending only on $\Omega$ such that
\begin{align}
&\|\dot{u}\|_{L^6}\le C(\|\nabla\dot{u}\|_{L^2}+\|\nabla u\|_{L^2}^2),\label{udot}\\
&\|\nabla\dot{u}\|_{L^2}\le C(\|\div \dot{u}\|_{L^2}+\|\curl \dot{u}\|_{L^2}+\|\nabla u\|_{L^4}^2).\label{tdudot}
\end{align}
\end{lemma}


To this end, we recall the following local existence theorem of classical solution of \eqref{CMHD}-\eqref{boundary}, which can be proved in a similar manner as that in \cite{tg2016,xh2017}, base on the standard contraction mapping principle.
\begin{lemma}\label{lem-local}
Assume that the initial date $(\rho_0,u_0,H_0)$ satisfy the conditions \eqref{dt1}, \eqref{dt2} and \eqref{dt3}. Then there exist a positive time $T_0>0$ and a unique classical solution $(\rho,u,H)$ of the system \eqref{CMHD}-\eqref{boundary} in $\Omega \times (0,T_0]$, satisfying that $\rho \geq 0$, and that for $\tau \in (0,T_0)$,
\begin{equation}\label{esti-uh-local}
\begin{cases}
(\rho-\bar{\rho},P-\bar{P})\in C([0,T_0);H^2 \cap W^{2,q} ),\\
\nabla u\in C([0,T_0);H^1 )\cap  L^\infty(\tau,T_0;H^2\cap W^{2,q}),\\
u_t\in L^\infty(\tau,T_0; D^1 \cap H^2)\cap H^1(\tau,T_0; H^1),\\
H \in C([0,T_0);H^2)\cap  L^\infty(\tau,T_0; H^4),\\
H_t\in C([0,T_0);L^2)\cap H^1(\tau,T_0; H^1)\cap  L^\infty(\tau,T_0; H^2).	
\end{cases}
\end{equation}
\end{lemma}

\section{\label{se3} A priori estimates(I): lower order estimates}

In this section, we will establish the time-independent a priori bounds of the solutions of the problem \eqref{CMHD}-\eqref{boundary}. Let $T>0$ be a fixed time and $(\rho,u,H)$ be a smooth solution to \eqref{CMHD}-\eqref{boundary} on $\Omega \times (0,T]$  with smooth initial data $(\rho_0,u_0,H_0)$ satisfying $u_0\in H^s, H_0\in H^s$ for some $s\in(\frac{1}{2},1]$ and $0\leq\rho_0\leq 2\hat{\rho}$.

Set $\si=\si(t)\triangleq\min\{1,t \},$ we define
\begin{align}
&  A_1(T) \triangleq \sup_{   0\le t\le T  }\sigma\left(\|\nabla u\|_{L^2}^2+\|\nabla H\|_{L^2}^2\right) \nonumber \\
&\qquad\qquad + \int_0^{T} \sigma\int
 \left(\rho|\dot{u} |^2+|\curl^2 H|^2+|H_t|^2\right) dxdt,\label{As1}\\
& A_2(T)  \triangleq \sup_{  0\le t\le T   }\sigma^2\int
 \left(\rho|\dot{u} |^2+|\curl^2 H|^2+|H_t|^2\right) dx \nonumber \\
&\qquad\qquad+ \int_0^{T}\int \sigma^2\left(|\nabla\dot{u}|^2+|\nabla H_t|^2\right)dxdt,\label{As2}\\
& A_3(T) \triangleq\sup_{  0\le t\le T   }\int|H|^3dx ,\label{As3}\\
& A_4(T) \triangleq \sup_{   0\le t\le T  }\sigma^{\frac{3-2s}{4}}\left(\|\nabla u\|_{L^2}^2+\|\nabla H\|_{L^2}^2\right) \nonumber \\
&\qquad\qquad+\int_0^{T} \sigma^{\frac{3-2s}{4}}\int
 \left(\rho|\dot{u} |^2+|\curl^2 H|^2+|H_t|^2\right) dxdt,\label{As4} \\
& A_5(T) \triangleq\sup_{  0\le t\le T   }\int\rho|u|^3dx,\label{As5}
\end{align}
where $s\in(\frac{1}{2},1]$ and $\dot{v}=v_t+u \cdot \nabla v$ is the material derivative.

Now we will give the following key a priori estimates in this section, which guarantees the existence of a global classical solution of \eqref{CMHD}--\eqref{boundary}.
\begin{proposition}\label{pr1}
Under the conditions of Theorem \ref{th1}, for $\delta_0\triangleq\frac{2s-1}{9s}\in(0,\frac{1}{9}]$, there exists a  positive constant  $\ve$ depending on $\mu$, $\lambda$, $\nu$, $a$, $\ga$, $\on$, $\hat{\rho}$, $s$, $\Omega$, $M_1$ and $M_2$ such that if $(\rho,u,H)$  is a smooth solution of \eqref{CMHD}-\eqref{boundary}  on $\Omega\times (0,T] $ satisfying
\begin{equation}\label{key1}
\begin{cases}
\sup\limits_{\Omega\times [0,T]}\rho\le 2\hat{\rho},\quad
A_1(T) + A_2(T) \le 2C_0^{1/2},\\
A_3(T)\leq 2C_0^{\delta_0},\quad A_4(\sigma(T))+A_5(\sigma(T))\leq 2C_0^{\delta_0},
\end{cases}
\end{equation}
 then the following estimates hold
 \begin{equation}\label{key2}
\begin{cases}
\sup\limits_{\Omega\times [0,T]}\rho\le 7\hat{\rho}/4,\quad
A_1(T) + A_2(T) \le C_0^{1/2},\\
A_3(T)\leq C_0^{\delta_0},\quad A_4(\sigma(T))+A_5(\sigma(T))\leq C_0^{\delta_0},
\end{cases}
\end{equation}
provided $C_0\le \ve.$
\end{proposition}
\begin{proof} Proposition \ref{pr1} is a consequence of the following Lemmas \ref{lem-h3}, \ref{lem-a3}-\ref{lem-brho} below.
\end{proof}

In the following, we will use the convention that $C$ denotes a generic positive constant depending on $\mu , \lambda , \nu,  \ga ,  a , \bar{\rho}, \hat{\rho},  s, \Omega$,  $M_1$  and $M_2$ and use $C(\alpha)$ to emphasize that $C$ depends on $\alpha$. We begin with the following standard energy estimate for $(\nabla u,\nabla H)$.
\begin{lemma}\label{lem-basic2}
 Let $(\rho,u,H)$ be a smooth solution of \eqref{CMHD}-\eqref{boundary} satisfying \eqref{key1}. Then there is a positive constant $C$ such that
\begin{align}\label{basic3}
\displaystyle \int_0^T(\|\nabla u\|_{L^{2}}^{4}+\|\nabla H\|_{L^{2}}^{4})dt\leq CC_{0}^{2 \delta_0},
\end{align}
provided $C_0\leq 1$.
\end{lemma}
\begin{proof}
By \eqref{key1}, together with \eqref{basic1}, \eqref{basic2}, we have
\begin{align}
&\int_0^T(\|\nabla u\|_{L^{2}}^{4}+\|\nabla H\|_{L^{2}}^{4})dt \nonumber \\
\leq & C \int_0^{\sigma(T)}(\|\nabla u\|_{L^{2}}^{4}+\|\nabla H\|_{L^{2}}^{4})dt
+\int_{\sigma(T)}^T(\|\nabla u\|_{L^{2}}^{4}+\|\nabla H\|_{L^{2}}^{4})dt \nonumber\\
\leq & C \sup_{   0\le t\le \sigma(T)}\left(\sigma^{\frac{3-2s}{4}}(\|\nabla u\|_{L^2}^2+\|\nabla H\|_{L^2}^2)\right)^2\int_0^{\sigma(T)} \sigma^{\frac{2s-3}{2}}dt \nonumber \\
& +C\sup_{\sigma(T)\le t\le T  }\sigma\left(\|\nabla u\|_{L^2}^2+\|\nabla H\|_{L^2}^2\right)\int_{\sigma(T)}^T(\|\nabla u\|_{L^{2}}^{2}+\|\nabla H\|_{L^{2}}^{2})dt \nonumber \\
\leq & CC_0^{2 \delta_0}+CC_0^{3/2} ~\leq  CC_0^{2 \delta_0},
\end{align}
since $s\in (1/2, 1]$ and $\delta_0 \in (0, 1/9]$. The proof of Lemma \ref{lem-basic} is completed.
\end{proof}

Now, we give the estimate of $A_3(T)$.
\begin{lemma}\label{lem-h3}
 Let $(\rho,u,H)$ be a smooth solution of
 \eqref{CMHD}-\eqref{boundary} satisfying \eqref{key1}.
  Then there is a positive constant
  $\ve_1>0 $, depending on $\mu,$ $\lambda,$ $\nu,$ $a$, $\gamma$, $\bar{\rho}$, $\hat{\rho}$ and $\Omega$  such that
  \begin{align}\label{h-l3}
  \displaystyle A_3(T)\leq C_0^{\delta_0},
  \end{align}
provided $C_0\leq \ve_1$.
\end{lemma}
\begin{proof}
Multiplying \eqref{CMHD1}$_3$ by $3|H|H$ and integrating by parts over $\Omega$, we have
\begin{align}\label{h3-1}
\frac{d}{dt}\|H\|^3_{L^3}=&-3\nu \int \curl H \cdot \curl (|H|H) dx
+3 \int |H|H \cdot \nabla u \cdot H dx -2 \int |H|^3 \div u dx \nonumber \\
\leq & C\|H\|_{L^\infty} \|\nabla H\|_{L^2}^2+ C \|\nabla u\|_{L^2}\|H\|_{L^6}^3 \nonumber \\
\leq & C \|\nabla H\|^{5/2}_{L^2}\|\curl^2  H\|_{L^2}^{1/2}+C \|\nabla H\|^{2}_{L^2}+ C \|\nabla H\|^{4}_{L^2}+C \|\nabla u\|_{L^2}^4,
 \end{align}
which together with \eqref{key1} and \eqref{basic3} indicates that
\begin{align}\label{h3-2}
&\displaystyle \sup_{0\le t\le T}\|H\|_{L^{3}}^{3} \nonumber \\
\leq & \|H_0\|_{L^{3}}^{3}+C \int_0^T \|\nabla H\|^{5/2}_{L^2}\|\curl^2  H\|_{L^2}^{1/2}dt +C C_0+ C C_0^{2 \delta_0} \nonumber \\
\leq & \|H_0\|_{L^{3}}^{3}+C \int_0^{\sigma(T)}
\left(\sigma^{\frac{3-2s}{4}}\|\nabla H\|_{L^2}^2\right)^{5/4}\left(\sigma^{\frac{3-2s}{4}}\|\curl^2  H\|_{L^2}^2\right)^{1/4}\sigma^{\frac{3(2s-3)}{8}} dt \nonumber \\
&+C\|\nabla H\|_{L^2} \int_{\sigma(T)}^T \left(\|\nabla H\|_{L^2}^2\right)^{3/4}\left(\sigma\|\curl^2  H\|_{L^2}^2\right)^{1/4}dt +C C_0+ C C_0^{2 \delta_0} \nonumber \\
\leq & C_1C_0^{3 \delta_0/2}
\end{align}
where in the last inequality we have used the simple fact
\begin{align}\label{h3-3}
\displaystyle  \|H_0\|_{L^{3}}^{3} \leq C \|H_0\|_{L^{2}}^{\frac{3(2s-1)}{2s}}\|H_0\|_{H^{s}}^{\frac{3}{2s}} \leq C(M_2)C_0^{27 \delta_0/4}.
\end{align}
Thus it follows from \eqref{h3-3} that \eqref{h-l3} holds provided
$C_0 \leq \ve_1 \triangleq \min \{1, C_1^{-\frac{2}{\delta_0}}\}.$ 
The proof of \ref{lem-h3} is completed.
\end{proof}

The following lemma shows the preliminary $L^2$ bounds for $\nabla H$.
\begin{lemma}\label{lem-tdh}
 Let $(\rho,u,H)$ be a smooth solution of
 \eqref{CMHD}--\eqref{boundary} satisfying \eqref{key1}.
  Then there is a positive constant  $C$ such that
 \begin{align}\label{tdh}
  \displaystyle \sup_{0\le t\le T  }\left(\sigma\|\nabla H\|_{L^2}^2\right)+\int_0^{T} \sigma
 \left(\|\curl^2 H\|_{L^2}^2+\|H_t\|_{L^2}^2\right)dt \leq CC_0.
  \end{align}
Moreover, for any $\theta \in [0,1]$, one has
\begin{align}\label{tdh-s}
  \displaystyle \sup_{0\le t\le T  }\left(\sigma^{1-\theta}\|\nabla H\|_{L^2}^2\right)+\int_0^{T} \sigma^{1-\theta}
 \left(\|\curl^2 H\|_{L^2}^2+\|H_t\|_{L^2}^2\right)dt \leq C\|H_0\|_{H^\theta}^2.
  \end{align}
\end{lemma}
\begin{proof}
Multiplying \eqref{CMHD1}$_3$ by $H$ and integrating by parts over $\Omega$, by \eqref{boundary}, \eqref{g1} and \eqref{tdh1}, we have
\begin{align}\label{tdh-1}
\displaystyle \left(\frac{1}{2}\|H\|_{L^2}^2\right)_t+\nu \|\curl H\|_{L^2}^2\leq
\|\nabla u\|_{L^2}\|H\|^2_{L^4}\leq \frac{\nu}{2}\|\curl H\|_{L^2}^2+C\|\nabla u\|^4_{L^2}\|H\|^2_{L^2},
\end{align}
which together with \eqref{tdh1}, \eqref{basic3} and Gronwall inequality gives
\begin{align}\label{tdh-2}
\displaystyle \sup_{0\le t\le T}\|H\|_{L^2}^2+\int_0^T\|\nabla H\|_{L^2}^2dt \leq
C\|H_0\|^2_{L^2}.
\end{align}
By Lemma \ref{lem-f-td}, one easily deduces from \eqref{CMHD1}$_3$ and \eqref{boundary} that
\begin{align}\label{tdh-3}
\displaystyle & \left(\frac{\nu}{2}\|\curl H\|_{L^2}^2\right)_t+\nu^2 \|\curl^2  H\|_{L^2}^2+\|H_t\|^2_{L^2}\nonumber \\
\leq & \int |H \cdot \nabla u-u \cdot \nabla H-H \div u|^2 dx \nonumber \\
\leq & C \|\nabla u\|^2_{L^2}\|H\|^2_{L^\infty}+C\|u\|^2_{L^6}\|\nabla H\|^2_{L^3}\nonumber \\
\leq & C \|\nabla u\|^2_{L^2}\|\nabla H\|_{L^2}\|\curl^2  H\|_{L^2}+C\|\nabla u\|^2_{L^2}\|\nabla H\|_{L^2}^2\nonumber \\
\leq & \frac{\nu^2}{2}\|\curl^2  H\|^2_{L^2}+C(\|\nabla u\|^2_{L^2}+\|\nabla u\|^4_{L^2})\|\nabla H\|_{L^2}^2,
\end{align}
using \eqref{tdh1}, \eqref{basic3} and Gronwall inequality, we get
\begin{align}\label{tdh-4}
  \displaystyle \sup_{0\le t\le T  }\|\nabla H\|_{L^2}^2+\int_0^{T}
 \left(\|\curl^2 H\|_{L^2}^2+\|H_t\|_{L^2}^2\right)dt \leq C\|\nabla H_0\|^2_{L^2}.
  \end{align}
On the other hand, multiplying \eqref{tdh-3} by $\sigma$ and integrating it over $(0,T)$, by \eqref{basic3} and \eqref{tdh-2}, we obtain
\begin{align}\label{tdh-5}
\displaystyle \sup_{0\le t\le T  }\left(\sigma\|\nabla H\|_{L^2}^2\right)+\int_0^{T} \sigma
 \left(\|\curl^2 H\|_{L^2}^2+\|H_t\|_{L^2}^2\right)dt \leq C\|H_0\|^2_{L^2} \leq CC_0.
\end{align}
This finishes the proof of \eqref{tdh}.
Note that for fixed $u$ (smooth), the solution operator $H_{0}\mapsto H(\cdot,t)$ is linear, by the standard Stein-Weiss interpolation argument \cite{bl}, one can deduce from \eqref{tdh-4} and \eqref{tdh-5} that \eqref{tdh-s} holds for any $\theta\in [0,1]$. The proof of Lemma \ref{lem-tdh} is completed.
\end{proof}

Then, we give the estimate of $A_1(T)$ and $A_2(T)$.
\begin{lemma}\label{lem-a0}
 Let $(\rho,u,H)$ be a smooth solution of
 \eqref{CMHD}-\eqref{boundary} satisfying \eqref{key1}.
  Then there is a positive constant
  $\ve_2 $ depending only  on $\mu$, $\lambda$, $\nu$, $a$, $\gamma$, $\bar{\rho}$, $\hat{\rho}$ and $\Omega$ such that
 \begin{align}\label{a01}
 \displaystyle  A_1(T) \le  C C_0 + C\int_0^{T}\int\sigma|\nabla u|^3dx dt,
 \end{align}
 \begin{align}\label{a02}
 \displaystyle  A_2(T)
    \le   C C_0^{1/2+2 \delta_0/3} + CA_1(T)  + C\int_0^{T}\int \sigma^2 |\nabla u|^4 dxdt,
 \end{align}
provided $C_0\leq \ve_2$.
\end{lemma}
\begin{proof}
Let $m\ge 0$ be a real number which will be determined later. Multiplying $\eqref{CMHD}_2 $ by
$\sigma^m \dot{u}$ and then integrating the reslting equality over
$\Omega$ lead  to
\begin{align}\label{I0}
\int \sigma^m \rho|\dot{u}|^2dx &
= -\int\sigma^m \dot{u}\cdot\nabla Pdx + (\lambda+2\mu)\int\sigma^m \nabla\div u\cdot\dot{u}dx \nonumber\\
&\quad - \mu\int\sigma^m \nabla\times\omega\cdot\dot{u}dx
+\int\sigma^m (H \cdot \nabla H- \nabla|H|^2/2)\cdot\dot{u}dx \nonumber\\
& \triangleq I_1+I_2+I_3+I_4.
\end{align}
We neglect the calculation details similar to that in \cite{cl2019} and only show the part related to the magnetic field and the boundary terms.
Firstly, By $\eqref{CMHD}_1$ and Lemma \ref{lem-f-td}, a direct calculation gives
\begin{align}\label{I1}
I_1 
= & -\int\sigma^m u_{t}\cdot\nabla(P-\bar{P})dx
- \int\sigma^m u\cdot\nabla u\cdot\nabla Pdx \nonumber\\
= & \left(\int\sigma^m(P-\bar{P})\,\div u\, dx\right)_{t} - m\sigma^{m-1}\sigma'\int(P-\bar{P})\,\div u\,dx \nonumber\\
&+ \int\sigma^{m}P\nabla u:\nabla u dx + (\gamma-1)\int\sigma^{m}P(\div u)^{2}dx - \int_{\partial\Omega}\sigma^{m}Pu\cdot\nabla u\cdot n ds\nonumber\\
\leq &\left(\int\sigma^m(P-\bar{P})\,\div u\, dx\right)_{t} + C \|\nabla u \|_{L^2}^2+ C\|\rho-\bar{\rho}\|_{L^2}^2.
\end{align}

Similarly, by \eqref{bdd2}, it indicates that
\begin{align}\label{I20}
I_2 
& = (\lambda+2\mu)\int_{\partial\Omega}\sigma^m\div u\,(\dot{u}\cdot n)ds - (\lambda+2\mu)\int\sigma^m\div u\,\div \dot{u}dx  \nonumber\\
& = (\lambda+2\mu)\int_{\partial\Omega}\sigma^m\div u\,(u\cdot\nabla u\cdot n)ds - \frac{\lambda+2\mu}{2}\left(\int\sigma^{m}(\div u)^{2}dx\right)_{t} \nonumber\\
&\quad +\frac{\lambda+2\mu}{2}\int\sigma^{m}(\div u)^{3}dx- (\lambda+2\mu)\int\sigma^m\div u\,\nabla u:\nabla udx  \nonumber\\
&\quad + \frac{m(\lambda+2\mu)}{2}\sigma^{m-1}\sigma'\int(\div u)^{2}dx
\end{align}
For the first term on the righthand side of \eqref{I20}, we obtain
\begin{align}\label{I21}
& (\lambda+2\mu)\int_{\partial\Omega}\sigma^m\div u\,(u\cdot\nabla u\cdot n)ds \nonumber \\
= & -\int_{\partial\Omega}\sigma^m Fu\cdot\nabla n\cdot uds-\int_{\partial\Omega}\sigma^m(P-\bar{P})u\cdot\nabla n\cdot uds-\int_{\partial\Omega}\sigma^m \frac{|H|^2}{2}u\cdot\nabla n\cdot uds \nonumber \\
\leq & C\left(\int_{\partial\Omega}\sigma^m|F||u|^{2}ds+\int_{\partial\Omega}\sigma^m|u|^{2}ds+\int_{\partial\Omega}\sigma^m|H|^2|u|^{2}ds\right)\nonumber \\
\leq & C\sigma^m(\|\nabla F\|_{L^{2}}\|u\|_{L^{4}}^{2}+\|F\|_{L^{6}}\|u\|_{L^{3}}\|\nabla u\|_{L^{2}}+\|F\|_{L^{2}}\| u\|^2_{L^{4}})+C\sigma^m\|\nabla u\|_{L^{2}}^2\nonumber\\
    &+C\sigma^m(\|\nabla H\|_{L^{2}}\|H\|_{L^{6}}\|u\|_{L^{6}}^{2}+\|H\|_{L^{6}}^2\|u\|_{L^{6}}\|\nabla u\|_{L^{2}}^{2}+\|H\|_{L^{4}}^{2}\|u\|_{L^{4}}^{2})\nonumber \\
\leq & \frac{1}{2}\sigma^m\|\rho\dot{u}\|_{L^2}^2+C \sigma^m\|\curl^2  H\|_{L^2}^2+C\sigma^m(\|\nabla u\|_{L^{2}}^2+\|\nabla H\|_{L^{2}}^2)(\|\nabla u\|_{L^{2}}^2+1),
\end{align}
where we have used
\begin{align}\label{h-tdh}
\displaystyle \|H \cdot \nabla H\|_{L^2}\leq C \|H\|_{L^3}\|\nabla H\|_{L^6} \leq C C_0^{\delta_0/3}(\|\curl^2  H\|_{L^2}+\|\nabla H\|_{L^2}).
\end{align}

Therefore,
\begin{align}\label{I2}
\displaystyle  I_2 \leq &  - \frac{\lambda+2\mu}{2}\left(\int\sigma^{m}(\div u)^{2}dx\right)_t+C\sigma^{m}\|\nabla u\|_{L^{3}}^{3}+\frac{1}{2}\sigma^m\|\rho\dot{u}\|_{L^2}^2+C \sigma^m\|\curl^2  H\|_{L^2}^2\nonumber \\
& +C\sigma^m(\|\nabla u\|_{L^{2}}^2+\|\nabla H\|_{L^{2}}^2)\|\nabla u\|_{L^{2}}^2+C(\|\nabla u\|_{L^{2}}^{2}+\|\nabla H\|_{L^{2}}^{2}).
\end{align}
Next, by \eqref{navier-b}, a straightforward computation shows that
\begin{align}\label{I3}
\displaystyle  I_3 
& = -\frac{\mu}{2}\left(\int\sigma^{m}|\omega|^{2}dx\right)_t + \frac{\mu m}{2}\sigma^{m-1}\sigma'\int|\omega|^{2}dx   \nonumber\\
& \quad - \mu\int\sigma^{m}(\nabla u^{i}\times\nabla_i u)\cdot\omega dx + \frac{\mu}{2}\int\sigma^{m}|\omega|^{2}\,\div udx\nonumber\\
& \leq -\frac{\mu}{2}\left(\int\sigma^{m}|\omega|^{2}dx\right)_t + Cm\sigma^{m-1}\sigma'\|\nabla u\|_{L^{2}}^{2} + C\sigma^{m}\|\nabla u\|_{L^{3}}^{3}.
\end{align}
Finally, by \eqref{navier-b}, a direct calculation yields
\begin{align}\label{I4}
I_4 
 = & \left(\int\sigma^m (H \cdot \nabla H- \nabla|H|^2/2)\cdot u dx\right)_t-m\sigma^{m-1}\sigma'\int (H \cdot \nabla H- \nabla|H|^2/2)\cdot u dx \nonumber \\
      &+\int\sigma^m \big((H \otimes H)_t:\nabla u- (|H|^2/2)_t\div u\big)dx +\int\sigma^m (H \cdot \nabla H- \nabla|H|^2/2)\cdot u \cdot \nabla u dx \nonumber \\
 \leq & \left(\int\sigma^m (H \cdot \nabla H- \nabla|H|^2/2)\cdot u dx\right)_t+C(\|\nabla H\|_{L^2}^2+\|\nabla u\|_{L^2}^2) \nonumber \\
      &+C\sigma^m(\|H_t\|_{L^2}^{2}+\|\curl^2  H\|_{L^2}^2)+C\sigma^m \|\nabla u\|_{L^3}^3+C\sigma^m\|\nabla H\|_{L^2}^2\|\nabla u\|_{L^2}^2 \nonumber \\
      &+ C\sigma^m\|\nabla H\|_{L^2}^2(\|\nabla H\|_{L^2}^4+\|\nabla u\|_{L^2}^4),
\end{align}
Making use of the results \eqref{I1}, \eqref{I2},\eqref{I3} and \eqref{I4}, it follows from $\eqref{I0}$ that
\begin{align}\label{I01}
&\left((\lambda+2\mu)\int\sigma^{m}(\div u)^{2}dx+\mu\int\sigma^{m}|\omega|^{2}dx\right)_{t}+\int\sigma^{m}\rho|\dot{u}|^{2}dx \nonumber \\
\leq &\left(\int\sigma^{m}(P-\bar{P})\,\div udx\right)_{t}+\left(\int\sigma^m (H \cdot \nabla H- \nabla|H|^2/2)\cdot u dx\right)_t \nonumber \\
     & +C(\|\rho-\bar{\rho}\|_{L^2}^2+\|\nabla H\|_{L^2}^2+\|\nabla u\|_{L^2}^2)+C\sigma^{m}\|\nabla u\|_{L^{3}}^{3} \nonumber \\
     &+C\sigma^m(\|H_t\|_{L^2}^{2}+\|\curl^2  H\|_{L^2}^2)+ C\sigma^m\|\nabla H\|_{L^2}^2(\|\nabla H\|_{L^2}^4+\|\nabla u\|_{L^2}^4) \nonumber \\
     & +C\sigma^m(\|\nabla u\|_{L^{2}}^2+\|\nabla H\|_{L^{2}}^2)(\|\nabla u\|_{L^{2}}^2+1),
\end{align}
integrating over $(0,T]$, by \eqref{tdu1}, \eqref{rho2-int}, Lemma \ref{lem-basic} and Young's inequality, we conclude that for any $m>0$,
\begin{align}\label{I02}
\displaystyle  &\sigma^{m}\|\nabla u\|_{L^{2}}^{2}+\int_0^T\int\sigma^{m}\rho|\dot{u}|^{2}dxdt \nonumber \\
\leq & CC_{0}+C\int_0^T \sigma^m(\|H_t\|_{L^2}^{2}+\|\curl^2  H\|_{L^2}^2) dt\nonumber \\
     &+C\int_0^T\sigma^{m}(\|\nabla H\|_{L^2}^2+\|\nabla u\|_{L^2}^2)(\|\nabla H\|_{L^2}^4+\|\nabla u\|_{L^2}^4)dt \nonumber \\
     &+C\int_0^T\sigma^{m}\|\nabla u\|_{L^{2}}^2(\|\nabla u\|_{L^{2}}^2+\|\nabla H\|_{L^{2}}^2)dt+C\int_0^T\sigma^{m}\|\nabla u\|_{L^{3}}^{3}dt.
\end{align}
Choose $m=1,$ together with \eqref{key1}, \eqref{basic3} and \eqref{tdh}, we obtain $\eqref{a01}$.

Now we will claim \eqref{a02}. Operating $ \sigma^{m}\dot{u}^{j}[\pa/\pa t+\div
(u\cdot)] $ to $ (\ref{2m2})^j,$ summing with respect to $j$, and integrating over $\Omega,$ together with $ \eqref{CMHD}_1 $, we get
\begin{align}\label{J0}
\displaystyle  &\left(\frac{\sigma^{m}}{2}\int\rho|\dot{u}|^{2}dx\right)_t-\frac{m}{2}\sigma^{m-1}\sigma'\int\rho|\dot{u}|^{2}dx \nonumber \\
& = \int\sigma^{m}(\dot{u}\cdot\nabla F_t+\dot{u}^{j}\,\div(u\partial_jF))dx \nonumber \\
&\quad+\mu\int\sigma^{m}(-\dot{u}\cdot\nabla\times\omega_t-\dot{u}^{j}\div((\nabla\times\omega)^j\,u))dx \nonumber \\
&\quad+\int\sigma^{m}(\dot{u}\cdot(\div (H\otimes H))_t+\dot{u}^{j}\div((\div (H\otimes H^j)\,u))dx \nonumber \\
& \triangleq J_1+ J_2+ J_3.
\end{align}
Let us estimate $J_1, J_2$ and $J_3$.
By \eqref{navier-b} and \eqref{CMHD1}$_1$, a direct computation yields
\begin{align}\label{J10}
J_1 
& = \int_{\partial\Omega}\sigma^{m}F_t\dot{u}\cdot nds-\int\sigma^{m}F_t\,\div\dot{u}dx- \int\sigma^{m}u \cdot \nabla\dot{u}^j\partial_jFdx \nonumber \\
& = \int_{\partial\Omega}\sigma^{m}F_t\dot{u}\cdot nds - (\lambda+2\mu)\int\sigma^{m}(\div\dot{u})^{2}dx + (\lambda+2\mu)\int\sigma^{m}\div\dot{u}\,\nabla u:\nabla udx \nonumber\\
& \quad -\gamma\int\sigma^{m} P\div\dot{u}\,\div udx+\int\sigma^{m}\div \dot{u}\, u \cdot \nabla F dx-\int\sigma^{m}u \cdot\nabla\dot{u}^j\partial_jF dx \nonumber \\
 &\quad +\int\sigma^{m}\div\dot{u}\,H \cdot H_tdx+\int\sigma^{m}\div\dot{u}\,u \cdot \nabla H \cdot H dx \nonumber\\
& \leq \int_{\partial\Omega}\sigma^{m}F_t\dot{u}\cdot nds - (\lambda+2\mu)\int\sigma^{m}(\div\dot{u})^{2}dx + \frac{\delta}{12}\sigma^{m}\|\nabla\dot{u}\|\ltwo+C\sigma^m \|\nabla u\|^4_{L^4}\nonumber\\
& \quad+C\sigma^m \big(\|\nabla u\|_{L^2}^2 \|\nabla F\|_{L^3}^2+ C_0^{2 \delta_0/3}\|\nabla H_t\|_{L^2}^2+\|\nabla u\|_{L^2}^2\|\nabla H\|_{L^2}^2\|\nabla H\|_{L^6}^2\big)
\end{align}
where in the second equality we have used
\begin{align*}
\displaystyle F_t
&=(2\mu+\lm)\div\dot  u-(2\mu+\lm)\na u:\na u  - u\cdot\na F+u \cdot \nabla H \cdot H+\ga P\div u-H \cdot H_t.
\end{align*}
For the first term on the righthand side of \eqref{J10}, we have
\begin{align}\label{J11}
&\int_{\partial\Omega}\sigma^{m}F_t\dot{u}\cdot nds=-\int_{\partial\Omega}\sigma^{m}F_t\,(u\cdot\nabla n\cdot u)ds \nonumber\\
= & -\left(\int_{\partial\Omega}\sigma^{m}(u\cdot\nabla n\cdot u)Fds\right)_t+m\sigma^{m-1}\sigma'\int_{\partial\Omega}(u\cdot\nabla n\cdot u)Fds \nonumber\\
&\quad +\int_{\partial\Omega}\sigma^{m}\big(F\dot{u}\cdot\nabla n \cdot u+Fu\cdot\nabla n \cdot\dot{u}\big)ds \nonumber \\
&\quad  -\int_{\partial\Omega}\sigma^{m}\big(F(u \cdot \nabla) u\cdot\nabla n \cdot u+Fu\cdot\nabla n \cdot(u \cdot \nabla) u\big)ds\nonumber\\
\leq& -\left(\int_{\partial\Omega}\sigma^{m}(u\cdot\nabla n\cdot u)Fds\right)_t+Cm\sigma^{m-1}\sigma'\|\nabla u\|_{L^2}^{2}\|F\|_{H^1} \nonumber\\
&\quad+\frac{\delta}{12}\sigma^{m}\|\nabla\dot{u}\|_{L^2}^2+C\sigma^{m}\|\nabla u\|_{L^2}^{4}+C\sigma^{m}\|\nabla u\|_{L^2}^{2}\|F\|_{H^1}^{2} \nonumber\\
&\quad +C\|\na F\|_{L^6} \|\na u\|^3_{L^2}+ C  \|  F\|_{H^1}\|\na u\|_{L^2} \left(\|\na u\|^2_{L^4} +\|\na u\|^2_{L^2}\right).
\end{align}
Together with Lemma \ref{lem-f-td}, \eqref{tdxd-u1},\eqref{udot},\eqref{tdudot},\eqref{h-tdh}, \eqref{J10} and \eqref{J11}, we have
\begin{align}\label{J1}
& J_1 \leq - (\lambda+2\mu)\int\sigma^{m}(\div\dot{u})^{2}dx -\left(\int_{\partial\Omega}\sigma^{m}(u\cdot\nabla n\cdot u)Fds\right)_t\nonumber\\
& \quad + \frac{\delta}{3}\sigma^{m}\|\nabla\dot{u}\|\ltwo+C\sigma^m C_0^{2 \delta_0/3}\|\nabla H_t\|\ltwo+C\sigma^m(\|\nabla H\|_{L^2}^4+\|\nabla u\|_{L^2}^4)\|\curl^2  H\|\ltwo \nonumber\\
& \quad +C\sigma^m \|\nabla u\|^4_{L^4}+C\sigma^{m}(\|\rho^{\frac{1}{2}}\dot{u}\|_{L^2}^{2}+\|\curl^2  H\|\ltwo)\|\nabla u\|_{L^2}^{2}+C\sigma^{m}\|\nabla H\|\ltwo \|\nabla u\|_{L^2}^6\nonumber\\
&\quad +C\sigma^m\|\nabla u\|_{L^2}^2(1+\|\nabla u\|_{L^2}^2+\|\nabla u\|_{L^2}^4+\|\nabla H\|_{L^2}^2+\|\nabla H\|_{L^2}^4+\|\nabla H\|_{L^2}^6)\nonumber\\
&\quad+Cm\sigma^{m-1}\sigma'(\|\rho^{\frac{1}{2}}\dot{u}\|_{L^2}^{2}+\|\curl^2  H\|\ltwo+\|\nabla u\|_{L^2}^{2}+\|\nabla H\|_{L^2}^{2}) \nonumber\\
&\quad+Cm\sigma^{m-1}\sigma'(\|\nabla u\|_{L^2}^{2}+\|\nabla H\|_{L^2}^{2})\|\nabla u\|_{L^2}^{2}.
\end{align}

Next, by $ \omega_t=\curl \dot u-u\cdot \na \omega-\na u^i\times \pa_iu$ and a straightforward calculation leads to
\begin{align}\label{J20}
J_2 
&=
-\mu\int\sigma^{m}|\curl \dot{u}|^{2}dx+\mu\int\sigma^{m}\curl\dot{u}\cdot(\nabla u^i\times\nabla_i u)dx \nonumber\\
&\quad+\mu\int\sigma^{m}u \cdot \nabla\omega \cdot \curl\dot{u}dx+\mu\int\sigma^{m} u \cdot\nabla\dot{u} \cdot(\nabla\times \omega)  dx\nonumber\\
&\leq -\mu\int\sigma^{m}|\curl \dot{u}|^{2}dx+\frac{\delta}{3}\sigma^{m}\|\nabla\dot{u}\|_{L^2}^2+C\sigma^{m}\|\nabla u\|_{L^4}^4 
\end{align}

Finally, a directly computation shows that
\begin{align}\label{J30}
\displaystyle J_3 
& =-\int\sigma^{m}\nabla\dot{u}:(H\otimes H)_t dx-\mu\int\sigma^{m}H \cdot \nabla H^j u \cdot \nabla\dot{u}^{j}dx\nonumber \\
& \leq C\sigma^{m}(\|\nabla\dot{u}\|_{L^2} \|H\|_{L^3} \|H_t\|_{L^6}+\|\nabla\dot{u}\|_{L^2} \|H\|_{L^6} \|\nabla H\|_{L^6}\|u\|_{L^6} ) \nonumber \\
& \leq \frac{\delta}{3}\sigma^{m}\|\nabla\dot{u}\|_{L^2}^2+C\sigma^m(\|\nabla H\|_{L^2}^4+\|\nabla u\|_{L^2}^4)\|\curl^2  H\|\ltwo \nonumber\\
& \quad +C\sigma^m C_0^{2 \delta_0/3}\|\nabla H_t\|\ltwo+C\sigma^m \|\nabla H\|_{L^2}^4\|\nabla u\|_{L^2}^2.
\end{align}

Combining \eqref{J1}, \eqref{J20} with \eqref{J30}, we deduce from \eqref{J0} that
\begin{align}\label{J01}
&\left(\frac{\sigma^{m}}{2}\|\rho^{\frac{1}{2}}\dot{u}\|_{L^2}^2\right)_t+(\lambda+2\mu)\sigma^{m}\|\div\dot{u}\|_{L^2}^2+\mu\sigma^{m}\|\curl\dot{u}\|_{L^2}^2 \nonumber\\
&\leq -\left(\int_{\partial\Omega}\sigma^{m}(u\cdot\nabla n\cdot u)Fds\right)_t+
\delta\sigma^{m}\|\nabla\dot{u}\|\ltwo+C\sigma^m C_0^{2 \delta_0/3}\|\nabla H_t\|\ltwo \nonumber\\
& \quad +C\sigma^m \|\nabla u\|^4_{L^4}+C\sigma^m(\|\nabla H\|_{L^2}^4+\|\nabla u\|_{L^2}^4)\|\curl^2  H\|\ltwo \nonumber\\
& \quad +C\sigma^{m}(\|\rho^{\frac{1}{2}}\dot{u}\|_{L^2}^{2}+\|\curl^2  H\|\ltwo)\|\nabla u\|_{L^2}^{2}+C\sigma^{m}\|\nabla H\|\ltwo \|\nabla u\|_{L^2}^6\nonumber\\
&\quad +C\sigma^m\|\nabla u\|_{L^2}^2(1+\|\nabla u\|_{L^2}^2+\|\nabla u\|_{L^2}^4+\|\nabla H\|_{L^2}^2+\|\nabla H\|_{L^2}^4+\|\nabla H\|_{L^2}^6)\nonumber\\
&\quad+Cm\sigma^{m-1}\sigma'(\|\rho^{\frac{1}{2}}\dot{u}\|_{L^2}^{2}+\|\curl^2  H\|\ltwo+\|\nabla u\|_{L^2}^{2}+\|\nabla H\|_{L^2}^{2}) \nonumber\\
&\quad+Cm\sigma^{m-1}\sigma'(\|\nabla u\|_{L^2}^{2}+\|\nabla H\|_{L^2}^{2})\|\nabla u\|_{L^2}^{2}.
\end{align}

By \eqref{tdudot} and Lemma \ref{lem-basic}, choosing $\delta$ small enough, and integrating \eqref{J01} over $(0,T]$, for $m>0$, we get
\begin{align}\label{J03}
&\quad \sigma^{m}\|\rho^{\frac{1}{2}}\dot{u}\|_{L^2}^2+\int_0^T\sigma^{m}\|\nabla\dot{u}\|_{L^2}^2dt\nonumber\\
&\leq -\int_{\partial\Omega}\sigma^{m}(u\cdot\nabla n\cdot u)Fds+CC_0^{2 \delta_0/3}\int_0^T\sigma^m \|\nabla H_t\|\ltwo dt \nonumber\\
& \quad +C\int_0^T\sigma^m \|\nabla u\|^4_{L^4}dt+CC_0^{2 \delta_0}\sup_{0\leq t\leq T}\sigma^m(\|\curl^2  H\|\ltwo+\|\rho^{\frac{1}{2}}\dot{u}\|_{L^2}^{2}) \nonumber\\
& \quad +CC_0^{2 \delta_0}\sup_{0\leq t\leq T}\sigma^{m}(\|\nabla u\|_{L^2}^{2}\|\nabla H\|\ltwo+\|\nabla u\|_{L^2}^{2})+CC_0\sup_{0\leq t\leq \sigma(T)}\sigma^{m-1}\|\nabla u\|_{L^2}^{2} \nonumber\\
&\quad+C \int_0^{\sigma(T)} m\sigma^{m-1}(\|\rho^{\frac{1}{2}}\dot{u}\|_{L^2}^{2}+\|\curl^2  H\|\ltwo)dt +CC_0.
\end{align}
For the boundary term in the right-hand side of \eqref{J03} , from Lemma \ref{lem-f-td}, we have
\begin{align}\label{J0b1}
&\quad\int_{\partial\Omega}(u\cdot\nabla n\cdot u)Fds\leq C\|\nabla u\|_{L^2}^{2}\|F\|_{H^1}\nonumber\\
&\leq\frac{1}{2}\|\rho^{\frac{1}{2}}\dot{u}\|_{L^2}^2+CC_0^{2 \delta_0/3}\|\curl^2  H\|\ltwo+C(\|\nabla u\|_{L^2}^2+\|\nabla H\|_{L^2}^2+\|\nabla u\|_{L^2}^4).
\end{align}
Therefore,
\begin{align}\label{J04}
&\sigma^{m}\|\rho^{\frac{1}{2}}\dot{u}\|_{L^2}^2+\int_0^T\sigma^{m}\|\nabla\dot{u}\|_{L^2}^2dt-C_2C_0^{2 \delta_0/3}\int_0^T\sigma^m \|\nabla H_t\|\ltwo dt \nonumber\\
&\leq C\int_0^T\sigma^m \|\nabla u\|^4_{L^4}dt+CC_0^{2 \delta_0}\sup_{0\leq t\leq T}\sigma^m(\|\curl^2  H\|\ltwo+\|\rho^{\frac{1}{2}}\dot{u}\|_{L^2}^{2}) \nonumber\\
& \quad +CC_0^{2 \delta_0}\sup_{0\leq t\leq T}\sigma^{m}(\|\nabla u\|_{L^2}^{2}\|\nabla H\|\ltwo+\|\nabla u\|_{L^2}^{2})+CC_0\sup_{0\leq t\leq \sigma(T)}\sigma^{m-1}\|\nabla u\|_{L^2}^{2} \nonumber\\
&\quad+C \int_0^{\sigma(T)} m\sigma^{m-1}(\|\rho^{\frac{1}{2}}\dot{u}\|_{L^2}^{2}+\|\curl^2  H\|\ltwo)dt +CC_0 \nonumber\\
&\quad+CC_0^{2 \delta_0/3}\sigma^{m}\|\curl^2  H\|\ltwo+C\sigma^{m}(\|\nabla u\|_{L^2}^2+\|\nabla H\|_{L^2}^2+\|\nabla u\|_{L^2}^4).
\end{align}

Next, we need to estimate the term $\|\nabla H_t\|_{L^2}$. Noticing that
\begin{equation}\label{ht}
\begin{cases}
 H_{tt}-\nu \nabla \times (\curl H_t)=(H \cdot \nabla u-u \cdot \nabla H-H \div u)_t,&\text{in}\quad \Omega,\\
 H_t \cdot n=0,\quad \curl H_t \times n=0,& \text{on}\quad \partial\Omega,\\
 \end{cases}
\end{equation}
and after directly computations we obtain
\begin{align}\label{ht1}
&\quad \left(\frac{\sigma^{m}}{2}\|H_t\|_{L^2}^2\right)_t+\sigma^{m}\|\curl H_t\|_{L^2}^2-\frac{m}{2}\sigma^{m-1}\sigma' \|H_t\|\ltwo\nonumber\\
&= \int \sigma^{m}(H_t \cdot \nabla u-u \cdot \nabla H_t-H_t \div u)\cdot H_t dx \nonumber \\
&\quad+ \int \sigma^{m}(H \cdot \nabla \dot{u}-\dot{u} \cdot \nabla H-H \div \dot{u})\cdot H_t dx \nonumber \\
&\quad - \int \sigma^{m}(H \cdot \nabla (u \cdot \nabla u)-(u \cdot \nabla u)\cdot \nabla H-H \div(u \cdot \nabla u) )\cdot H_t dx \nonumber \\
& \triangleq K_1+K_2+K_3.
\end{align}
By Lemma \ref{lem-gn} and Lemma \ref{lem-f-td}, a direct calculation leads to
\begin{align}\label{htk1}
K_1 
& \leq C\sigma^{m}(\|H_t\|_{L^3}\|H_t\|_{L^6}\|\nabla u\|_{L^2}
+\|u\|_{L^6}\|H_t\|_{L^3}\|\nabla H_t\|_{L^2})\nonumber\\
&\leq \frac{\delta}{4}\sigma^{m}\|\nabla H_t\|_{L^2}^{2}+C\sigma^{m}\|\nabla u\|_{L^2}^4\|H_t\|_{L^2}^{2}.
\end{align}
Similarly,
\begin{align}\label{htk20}
K_2 
& \leq C\sigma^{m}\|H\|_{L^3}\|H_t\|_{L^6}\|\nabla \dot{u}\|_{L^2}
-\int_{\partial \Omega}\sigma^{m}(\dot{u} \cdot n)( H \cdot H_t) ds \nonumber\\
 & \quad +\int\sigma^{m}\div\dot{u}\, H \cdot H_t dx+\int\sigma^{m} \dot{u}\cdot \nabla H_t \cdot H dx \nonumber\\
&\leq \int_{\partial \Omega}\sigma^{m}(u \!\cdot \!\nabla n \cdot u)( H\! \cdot\! H_t) ds+CC^{\delta_0/3}\sigma^{m}(\|\nabla \dot{u}\|_{L^2}^{2}\!+\!\|\nabla u\|_{L^2}^{4}\!+\!\|\nabla H_t\|_{L^2}^{2}),
\end{align}
where in the lase inequality we have used the fact that $\dot{u} \cdot n=-u \cdot \nabla n \cdot u$ on $\partial \Omega$. By Sobolev trace theorem and Lemma \ref{lem-f-td}, it indicates that
\begin{align}\label{htk21}
& \quad\int_{\partial \Omega}\sigma^{m}(u \cdot \nabla n \cdot u)( H \cdot H_t) ds\nonumber\\
&\leq C\sigma^{m}(\|u\|_{L^6}\|\nabla u\|_{L^2}\|H\|_{L^6}\|H_t\|_{L^6}+\|u\|_{L^6}^2\|\nabla H\|_{L^2}\|H_t\|_{L^6}\nonumber\\
& \quad+\|u\|_{L^6}^2\|\nabla H_t\|_{L^2}\|H\|_{L^6}+\|u\|_{L^4}^{2}\|H\|_{L^3}\|H_t\|_{L^6})\nonumber\\
&\leq \frac{\delta}{4}\sigma^{m}\|\nabla H_t\|_{L^2}^{2}+C\sigma^{m}\|\nabla u\|_{L^2}^4\|\nabla H\|_{L^2}^{2}+CC^{\delta_0/3}\sigma^{m}\|\nabla u\|_{L^2}^{4}.
\end{align}
Combining \eqref{htk20} and \eqref{htk21}, we have
\begin{align}\label{htk2}
K_2 &\leq  \frac{\delta}{4}\sigma^{m}\|\nabla H_t\|_{L^2}^{2}+C\sigma^{m}\|\nabla u\|_{L^2}^4\|\nabla H\|_{L^2}^{2}\nonumber \\
 &\quad +CC^{\delta_0/3}\sigma^{m}(\|\nabla \dot{u}\|_{L^2}^{2}+\|\nabla u\|_{L^2}^{4}+\|\nabla H_t\|_{L^2}^{2}).
\end{align}
By Sobolev trace theorem and Lemma \ref{lem-f-td} again, a direct computation yields
\begin{align}\label{htk3}
K_3 
& =\int \sigma^{m} H \cdot\nabla H_t \cdot( u \cdot \nabla u) dx 
+\int_{\partial \Omega}\sigma^{m} H\cdot H_t\,(u \cdot \nabla u \cdot n) ds
\nonumber \\
& \quad -\int \sigma^{m} u \cdot \nabla u \cdot \nabla H_t \cdot H dx \nonumber \\
&\leq \frac{\delta}{2}\sigma^{m}\|\nabla H_t\|_{L^2}^{2}+C\sigma^{m}(\|\nabla u\|_{L^2}^4+\|\nabla H\|_{L^2}^{4})(\|\rho^{\frac{1}{2}}\dot{u}\|_{L^2}^{2}+\|\curl^2  H\|\ltwo)\nonumber \\
 &\quad+C\sigma^{m}\|\nabla u\|_{L^2}^2\|\nabla H\|_{L^2}^{2}(\|\nabla u\|_{L^2}^2+\|\nabla H\|_{L^2}^{2}+1)\nonumber \\
 &\quad +CC^{\delta_0/3}\sigma^{m}(\|\nabla \dot{u}\|_{L^2}^{2}+\|\nabla u\|_{L^2}^{4}+\|\nabla H_t\|_{L^2}^{2}).
\end{align}
Putting \eqref{htk1}, \eqref{htk2} and \eqref{htk3} into \eqref{ht1}, choosing $\delta$ small enough, we have
\begin{align}\label{ht3}
&\quad \left(\sigma^{m}\|H_t\|_{L^2}^2\right)_t+\sigma^{m}\|\nabla H_t\|_{L^2}^2-CC^{\delta_0/3}\sigma^{m}(\|\nabla \dot{u}\|_{L^2}^{2}+\|\nabla H_t\|_{L^2}^{2})\nonumber\\
&\leq C\sigma^{m}(\|\nabla u\|_{L^2}^4+\|\nabla H\|_{L^2}^{4})(\|\rho^{\frac{1}{2}}\dot{u}\|_{L^2}^{2}+\|\curl^2  H\|\ltwo+\|H_t\|_{L^2}^2)\nonumber \\
 &\quad+C\sigma^{m}\|\nabla u\|_{L^2}^2\|\nabla H\|_{L^2}^{2}(\|\nabla u\|_{L^2}^2+\|\nabla H\|_{L^2}^{2}+1)+Cm\sigma^{m-1}\sigma' \|H_t\|\ltwo.
\end{align}
Integrating over $(0,T]$, then by Lemma \ref{lem-basic}, for $m>0$, we get
\begin{align}\label{ht4}
&\quad \sigma^{m}\|H_t\|_{L^2}^2+\int_0^T\sigma^{m}\|\nabla H_t\|_{L^2}^2dt-C_3C^{\delta_0/3}\int_0^T\sigma^{m}(\|\nabla \dot{u}\|_{L^2}^{2}+\|\nabla H_t\|_{L^2}^{2})dt\nonumber\\
&\leq CC_0^{2 \delta_0}\sup_{0\leq t\leq T}\sigma^{m}(\|\rho^{\frac{1}{2}}\dot{u}\|_{L^2}^{2}+\|\curl^2  H\|\ltwo+\|H_t\|_{L^2}^2)+CC_0\sup_{0\leq t\leq T}\sigma^{m}\|\nabla u\|_{L^2}^2\nonumber \\
 &\quad+CC_0^{2 \delta_0}\sup_{0\leq t\leq T}\sigma^{m}(\|\nabla u\|_{L^2}^2+\|\nabla H\|_{L^2}^{2})+C \int_0^{\sigma(T)} m\sigma^{m-1}\|H_t\|\ltwo dt.
\end{align}
Now take $m=2$ in \eqref{J04} and \eqref{ht4}, we deduce after adding them together that
\begin{align}\label{a20}
&\sigma^{2}(\|\rho^{\frac{1}{2}}\dot{u}\|_{L^2}^2+\|H_t\|_{L^2}^2)+\int_0^T\sigma^{2}(\|\nabla\dot{u}\|_{L^2}^2+\|\nabla H_t\|_{L^2}^2)dt \nonumber \\
& \quad -C_2C_0^{2 \delta_0/3}\int_0^T\sigma^2 \|\nabla H_t\|\ltwo dt -C_3C^{\delta_0/3}\int_0^T\sigma^{2}(\|\nabla \dot{u}\|_{L^2}^{2}+\|\nabla H_t\|_{L^2}^{2})dt\nonumber\\
&\leq C\int_0^T\sigma^2 \|\nabla u\|^4_{L^4}dt+CC_0^{2 \delta_0}\sup_{0\leq t\leq T}\sigma^2(\|\rho^{\frac{1}{2}}\dot{u}\|_{L^2}^{2}+\|\curl^2  H\|\ltwo+\|H_t\|_{L^2}^2) \nonumber\\
& \quad +CC_0^{2 \delta_0}\sup_{0\leq t\leq T}\sigma^{2}(\|\nabla u\|_{L^2}^{2}\|\nabla H\|\ltwo+\|\nabla u\|_{L^2}^{2})+CC_0\sup_{0\leq t\leq T}\sigma\|\nabla u\|_{L^2}^{2} \nonumber\\
&\quad+C \int_0^{\sigma(T)} \sigma(\|\rho^{\frac{1}{2}}\dot{u}\|_{L^2}^{2}+\|\curl^2  H\|\ltwo+\|H_t\|\ltwo)dt +CC_0 \nonumber\\
&\quad+CC_0^{2 \delta_0/3}\sigma^{2}\|\curl^2  H\|\ltwo+C\sigma^{2}(\|\nabla u\|_{L^2}^2+\|\nabla H\|_{L^2}^2+\|\nabla u\|_{L^2}^4)\nonumber\\
&\leq C\int_0^T\sigma^2 \|\nabla u\|^4_{L^4}dt+CC_0^{2 \delta_0+1/2}+CA_1(T)+CC_0+CC_0^{2 \delta_0/3+1/2}.
\end{align}
Thus we have
\begin{align}\label{a21}
&\quad \sup_{0\leq t\leq T}\sigma^{2}(\|\rho^{\frac{1}{2}}\dot{u}\|_{L^2}^2+\|H_t\|_{L^2}^2)+\int_0^T\sigma^{2}(\|\nabla\dot{u}\|_{L^2}^2+\|\nabla H_t\|_{L^2}^2)dt \nonumber \\
&\leq C\int_0^T\sigma^2 \|\nabla u\|^4_{L^4}dt+CA_1(T)+CC_0^{2 \delta_0/3+1/2}.
\end{align}
provided that $C_0$ is chosen to satisfy
\begin{align*}
	C_0\leq \ve_2 \triangleq \min \{\ve_1, (4C_2)^{-3/{2 \delta_0}},(4C_3)^{-3/{\delta_0}}\}.
\end{align*}
Finally, by Lemma \ref{lem-gn} and \eqref{CMHD}$_3$, it holds
\begin{align}\label{h2xd1}
\|\curl^2  H\|_{L^2}
&\leq C(\|H_t\|+\|\curl^2  H\|_{L^2}^{1/2}\|\nabla H\|_{L^2}^{1/2}\|\nabla u\|_{L^2}+\|\nabla H\|_{L^2}\|\nabla u\|_{L^2})\nonumber \\
 &\leq \frac{1}{2}\|\curl^2  H\|_{L^2}+C(\|H_t\|+\|\nabla H\|_{L^2}\|\nabla u\|_{L^2}^2+\|\nabla H\|_{L^2}\|\nabla u\|_{L^2}).
\end{align}
Thus, by \eqref{key1} and \eqref{h2xd1}, we have
\begin{align}\label{h2xd2}
& \quad \sup_{0\leq t\leq T}\sigma^{2} \|\curl^2  H\|_{L^2}^2\nonumber \\
&\leq C\sup_{0\leq t\leq T} \sigma^{2}(\|H_t\|^2+\|\nabla H\|_{L^2}^2\|\nabla u\|_{L^2}^4+\|\nabla H\|_{L^2}^2\|\nabla u\|_{L^2}^2)\nonumber \\
& \leq  C\int_0^T\sigma^2 \|\nabla u\|^4_{L^4}dt+CA_1(T)+CC_0^{2 \delta_0/3+1/2}.
\end{align}
Combining \eqref{a21} and \eqref{h2xd2}, we give \eqref{a02} and complete the proof of Lemma \ref{lem-a0}.
\end{proof}

\begin{lemma}\label{lem-a1} Assume that $(\rho,u,H)$ is a smooth solution of
 \eqref{CMHD}-\eqref{boundary} satisfying \eqref{key1} and the initial data condition \eqref{dt2}, then there exist  positive constants $C$ and $\varepsilon_3$ depending only on $\mu,\,\,\lambda,\,\,\nu,\,\,\gamma,\,\,a,\,\,\bar{\rho},\,\,\hat{\rho},\,\,s,\,\,\Omega$, $M_1$  and $M_2$ such that
\begin{align}
& \sup_{0\le t\le \si(T)}t^{1-s}\|\nabla u\|_{L^2}^2+\int_0^{\si(T)}t^{1-s}\int\rho|\dot{u}|^2 dxdt\le C(\hat{\rho},M_1,M_2),\label{uv1}\\
&\sup_{0\le t\le \si(T)}t^{2-s}\int
 \left(\rho|\dot{u}|^2+| \mathop{\rm curl}^2\nolimits H|^2+|H_t|^2\right) dx \nonumber \\
&\qquad + \int_0^{\si(T)}\int t^{2-s}\left(|\nabla\dot{u}|^2+|\nabla H_t|^2\right)dxdt\leq C(\hat{\rho},M_1,M_2),\label{uv2}
\end{align}
provide that $C_0<\varepsilon_3$.
\end{lemma}

\begin{proof} Suppose $w_1(x,t)$, $w_2(x,t)$ and $w_3(x,t)$ solve problems
\begin{equation}\label{fcz1}
\displaystyle \begin{cases}
  Lw_1=0 \,\, &in \,\,\Omega,\\  w_1(x,0)=w_{10}(x)\,\, &in \,\,\Omega,\\w_1\cdot n=0\,\,\text{and}\,\, \curl w_1\times n=0 \,\,&on \,\,\partial\Omega ,
\end{cases}
\end{equation}
\begin{equation}\label{fcz2}
\displaystyle  \begin{cases}
  Lw_2=-\nabla(P-\bar{P})\,\, &in \,\,\Omega,\\  w_2(x,0)=0\,\, &in \,\,\Omega,\\w_2\cdot n=0\,\,\text{and}\,\, \curl w_2\times n=0 \,\,&on \,\,\partial\Omega ,
\end{cases}
\end{equation}
and
\begin{equation}\label{fcz3}
\displaystyle  \begin{cases}
  Lw_3=(\nabla \times B)\times H\,\, &in \,\,\Omega,\\  w_3(x,0)=0\,\, &in \,\,\Omega,\\w_3\cdot n=0\,\,\text{and}\,\, \curl w_3\times n=0 \,\,&on \,\,\partial\Omega ,
  \end{cases}
\end{equation}
where $Lf\triangleq\rho\dot{f}-\mu\Delta f-(\lambda+\mu)\nabla\div f $, $B$ is the solution of \eqref{CMHD}$_3$ with fixed smooth $u$ and initial data $B_0(x)$. Note that $B$ satisfies \eqref{tdh} and \eqref{tdh-s} of Lemma \ref{lem-tdh}.

Just as we have done in the proof of Lemma \ref{lem-f-td}, by Lemma \ref{lem-lame} and Sobolev's inequality, for any $p\in[2,6],$ we have
\be \ba
\displaystyle  \|\nabla^{2} w_1\|_{L^{2}}\leq C(\|\rho\dot{w}_1\|_{L^{2}}+\|w_1\|_{L^{2}})\leq C(\|\rho\dot{w}_1\|_{L^{2}}+\|\nabla w_1\|_{L^{2}}),
\ea\ee
\be \label{xbh1} \ba
\|\nabla w_1\|_{L^{p}}\leq C\|w_1\|_{W^{2,2}}\leq C(\|\rho\dot{w}_1\|_{L^{2}}+\|\nabla w_1\|_{L^{2}}),
\ea\ee
\be \label{xbh2} \ba
\|\nabla F_{w_2}\|_{L^{p}}\leq C\|\rho\dot{w}_2\|_{L^{p}},
\ea\ee
\be \label{xbh3} \ba
\|F_{w_2}\|_{L^{p}}&\leq C(\|\nabla F_{w_2}\|_{L^{2}}+\|F_{w_2}\|_{L^{2}})\le C (\|\rho\dot{w}_2\|_{L^{2}}+\|\nabla w_2\|_{L^{2}}),
\ea\ee
\be \label{xbh4} \ba
\|\nabla w_2\|_{L^{p}}&\leq C\|\rho^{\frac{1}{2}}\dot{w}_2\|_{L^{2}}^{\frac{3p-6}{2p}}(\|\nabla w_2\|_{L^{2}}+\|P-\bar{P}\|_{L^{2}})^{\frac{6-p}{2p}}\\
& \quad +C(\|\nabla w_2\|_{L^{2}}+\|P-\bar{P}\|_{L^{p}}),
\ea\ee
where $F_{w_2}=(\lambda+2\mu)\div w_2-(P-\bar{P})  .$

A similar way as for the proof of \eqref{basic1} shows that
\begin{align}\label{bw1}
\displaystyle  \sup_{0\le t\le \si(T)}\int\rho|w_1|^2dx+\int_0^{\si(T)}\int|\nabla w_1|^2dxdt\le
C\int|w_{10}|^{2}dx  ,
\end{align}
\begin{align}\label{bw2}
\displaystyle  \sup_{0\le t\le \si(T)}\int\rho|w_2|^2dx+\int_0^{\si(T)}\int|\nabla w_2|^2dxdt\le CC_0 ,
\end{align}
and
\begin{equation}\label{bw3}
\displaystyle  \sup_{0\le t\le \si(T)}\int\!\rho|w_3|^2dx\!+\!\int_0^{\si(T)}\!\int|\nabla w_3|^2dxdt\le C\int_0^{\si(T)}\!\int|\nabla B|^2dxdt \le C \|B_0\|\ltwo.
\end{equation}

Multiplying \eqref{fcz1} by $w_{1t}$ and integrating over $\Omega,$ by \eqref{xbh1}, \eqref{key1},  Sobolev's and Young's inequalities, we obtain
\begin{align}\label{bw11}
\displaystyle  &\quad \left(\frac{\lambda+2\mu}{2}\int(\div w_1)^{2}dx + \frac{\mu}{2}\int|\curl w_1|^{2}dx\right)_t + \int\rho|\dot{w}_1|^{2}dx \nonumber \\
& =\int\rho\dot{w}_1\cdot(u\cdot\nabla w_1)dx 
\le C_4C_0^{\frac{\delta_0}{3}}(\|\rho^{\frac{1}{2}}\dot{w}_1\|_{L^{2}}^{2}+ \|\nabla w_1\|_{L^{2}}^{2}).
\end{align}
Together with \eqref{bw1}, and by Gronwall's inequality and Lemma \ref{lem-vn}, it yields
\begin{align}\label{bw12}
\displaystyle  \sup_{0\le t\le \si(T)}\|\nabla w_1\|_{L^{2}}^{2}+\int_0^{\sigma(T)}\int\rho|\dot{w}_1|^{2}dxdt\leq C\|\nabla w_{10}\|_{L^{2}}^{2},
\end{align}
and
\begin{align}\label{bw13}
\displaystyle  \sup_{0\le t\le \si(T)}t\|\nabla w_1\|_{L^{2}}^{2}+\int_0^{\sigma(T)}t\int\rho|\dot{w}_1|^{2}dxdt\leq C\|w_{10}\|_{L^{2}}^{2},
\end{align}
provided $C_0<\hat{\varepsilon}_1\triangleq(2C_4)^{-\frac{3}{\delta_0}}$.

Since the solution operator $w_{10}\mapsto w_1(\cdot,t)$ is linear, by the   standard Stein-Weiss interpolation argument \cite{bl}, one can deduce from \eqref{bw12} and \eqref{bw13} that for any $\theta\in [s,1],$
\begin{align}\label{bw14}
\displaystyle  \sup_{0\le t\le\si(T)}t^{1-\theta}\|\nabla
w_1\|_{L^2}^2+\int_0^{\si(T)}t^{1-\theta}\int\n|\dot
{w_1}|^2dxdt\leq C\| w_{10}\|_{H^\theta}^2,
\end{align}
with a uniform constant $C$ independent of $\theta.$

Multiplying \eqref{fcz2} by $w_{2t}$ and integrating over $\Omega$ give that
\begin{align}\label{bw21}
&\quad \left(\frac{\lambda+2\mu}{2}\int(\div w_2)^{2}dx+\frac{\mu}{2}\int|\curl w_2|^{2}dx-\int(P-\bar{P})\div w_2dx\right)_t+\int\rho|\dot{w}_2|^{2}dx \nonumber\\
& =\int\rho\dot{w}_2\cdot(u\cdot\nabla w_2)dx - \frac{1}{\lambda+2\mu}\int(P-\bar{P})(F_{w_2}\div u+\nabla F_{w_2}\cdot u ) dx  \nonumber \\
& \quad - \frac{1}{2(\lambda+2\mu)}\int(P-\bar{P})^{2}\div u dx +\gamma\int P\div u\,\div w_2dx \nonumber\\
& \leq C(\|\rho^{\frac{1}{2}}\dot{w}_2\|_{L^2}\|\rho^{\frac{1}{3}}u\|_{L^3}\|\nabla w_2\|_{L^6}+\|\nabla u\|_{L^2}\|F_{w_2}\|_{L^2}+\|\nabla F_{w_2}\|_{L^2}\|u\|_{L^2})\nonumber\\
&\quad+C(\|P-\bar{P}\|_{L^2}\|\nabla u\|_{L^2}+\|\nabla u\|_{L^2}\|\nabla w_2\|_{L^2})\nonumber\\
&\leq C_5C_0^{\frac{\delta_0}{3}}\|\rho^{\frac{1}{2}}\dot{w}_2\|_{L^2}^{2}+\frac{1}{4}\|\rho^{\frac{1}{2}}\dot{w}_2\|_{L^2}^{2}
+C(\|\nabla w_2\|_{L^2}^2+\|\nabla u\|_{L^2}^2+\|P-\bar{P}\|_{L^2}^{2/3}),
\end{align}
where we have utilized \eqref{key1}, \eqref{xbh2}-\eqref{xbh4}, H\"{o}lder's, Poincar\'{e}'s and Young's inequalities.
As a result,
\begin{align}\label{bw22}
\displaystyle  & \quad\left((\lambda+2\mu)\|\div w_2\|_{L^{2}}^{2}+\mu\|\curl w_2\|_{L^{2}}^{2}-2\int(P-\bar{P})\div w_2dx\right)_t
+\int\rho|\dot{w}_2|^{2}dx \nonumber \\
 & \le C\left(\|\nabla w_2\|_{L^{2}}^{2}+\|\nabla u\|_{L^{2}}^{2}+\|P-\bar{P}\|_{L^{2}}^\frac{2}{3}\right) ,
\end{align}
provide that $C_0<\hat{\varepsilon}_2\triangleq(4C_5)^{-\frac{3}{\delta_0}}$.

By Gronwall's inequality, \eqref{bw2} and Lemmas \ref{lem-vn}, \ref{lem-basic}, one has
\begin{align}\label{bw23}
\displaystyle  \sup_{0\le t\le \si(T)}\|\nabla w_2\|_{L^{2}}^{2}+\int_0^{\sigma(T)}\int\rho|\dot{w}_2|^{2}dxdt\leq CC_0^{1/3}.
\end{align}

Similarly, multiplying \eqref{fcz3} by $w_{3t}$ and integrating over $\Omega$ give that
\begin{align}\label{bw31}
&\quad \left(\frac{\lambda+2\mu}{2}\int(\div w_3)^{2}dx+\frac{\mu}{2}\int|\curl w_3|^{2}dx\right)_t+\int\rho|\dot{w}_3|^{2}dx \nonumber\\
& =\left(\int(H \cdot\nabla B- \nabla B \cdot H)\cdot w_3dx\right)_t-\int (H \cdot\nabla B- \nabla B \cdot H)_t \cdot w_3dx \nonumber\\
 &\quad +\int\rho\dot{w}_3\cdot(u\cdot\nabla w_3)dx \nonumber \\
 & \triangleq \frac{d}{dt}L_0+L_1+L_2.
\end{align}
By \eqref{g1} and \eqref{h-l3}, we have
\begin{align}\label{bw3l0}
L_0\leq C\|H\|_{L^3}\|\nabla B\|_{L^2}\|w_3\|_{L^6}\leq \frac{\mu}{4}\|\nabla w_3\|_{L^2}^2+C C_0^{2 \delta_0/3}\|\nabla B\|_{L^2}^2.
\end{align}
Using \eqref{tdh}, a directly computation yields
\begin{align}\label{bw3l1}
L_1 
 &\leq C(\|H_t\|_{L^2}\|\nabla B\|_{L^3}\|\nabla w_3\|_{L^2}+\|B_t\|_{L^2}\|H\|_{L^3}\|\nabla w_3\|_{L^6}+\|B_t\|_{L^2}\|\nabla w_3\|_{L^2}\|\nabla H\|_{L^3}) \nonumber \\
 &\leq CC^{\delta_0/3}(\|\rho^{1/2}\dot{u}\|_{L^2}^2+\|\nabla w_3\|_{L^2}^2)+C(\|B_t\|_{L^2}^2+\|\curl^2  B\|_{L^2}^2+\|\nabla B\|_{L^2}^2)\nonumber \\
 &\quad +C(\|\curl^2  H\|_{L^2}\|\nabla H\|_{L^2}+\|\nabla H\|_{L^2}^2)\|\nabla w_3\|_{L^2}^2+\|H_t\|_{L^2}\|\nabla B\|_{L^2}\|\nabla w_3\|_{L^2}\nonumber \\
 &\quad +\|\curl^2  B\|_{L^2}^{1/2}\|\nabla B\|_{L^2}^{1/2}\|H_t\|_{L^2}\|\nabla w_3\|_{L^2},
\end{align}
where we have used the fact \eqref{xbh4}.
Similarly, using Lemma \ref{lem-gn} yields
\begin{align}\label{bw3l2}
L_2 
 &\leq CC^{\delta_0/3}(\|\rho^{1/2}\dot{w_3}\|_{L^2}^2+\|\curl^2  B\|_{L^2}^{2}+\|\nabla B\|_{L^2}^{2}+\|\nabla w_3\|_{L^2}^2).
\end{align}
Putting \eqref{bw3l0}, \eqref{bw3l1} and \eqref{bw3l2} into \eqref{bw31}, we obtain
\begin{align}\label{bw32}
&\quad (\|\nabla w_3\|\ltwo )_t+\int\rho|\dot{w}_3|^{2}dx-(C C_0^{2 \delta_0/3}\|\nabla B\|_{L^2}^2)_t \nonumber\\
&\leq C_6C^{\delta_0/3}(\|\rho^{1/2}\dot{w_3}\|_{L^2}^2+\|\nabla w_3\|_{L^2}^2)+C(\|B_t\|_{L^2}^2+\|\curl^2  B\|_{L^2}^2+\|\nabla B\|_{L^2}^2)\nonumber \\
 &\quad +C(\|\curl^2  H\|_{L^2}\|\nabla H\|_{L^2}+\|\nabla H\|_{L^2}^2)\|\nabla w_3\|_{L^2}^2+C\|H_t\|_{L^2}\|\nabla B\|_{L^2}\|\nabla w_3\|_{L^2}\nonumber \\
 &\quad +C\|\curl^2  B\|_{L^2}^{1/2}\|\nabla B\|_{L^2}^{1/2}\|H_t\|_{L^2}\|\nabla w_3\|_{L^2}.
\end{align}
Thus, if $C_0$ is chosen to be such that $C_0\leq \hat{\ve}_3 \triangleq (2C_6)^{-3/\delta_0}$, we obtain
\begin{align}\label{bw33}
&\quad \sup_{0\le t\le  \si(T)}\|\nabla w_3\|\ltwo+\int_0^{\sigma(T)}\int\rho|\dot{w}_3|^{2}dxdt \nonumber\\
&\leq C\|\nabla B_0\|_{L^2}^2+CC_0^{\delta_0}\sup_{0\le t\le  \si(T)} \|\nabla w_3\|_{L^2}^2 +CC_0^{\delta_0/2}\sup_{0\le t\le  \si(T)}(\|\nabla B\|_{L^2}^2+\|\nabla w_3\|_{L^2}^2)\nonumber \\
 &\quad +CC_0^{\delta_0/2}\left(\int_0^{\sigma(T)}\|\curl^2  B\|_{L^2}^2dt\right)^{1/4}\sup_{0\le t\le  \si(T)}\|\nabla B\|_{L^2}^{1/2}\|\nabla w_3\|_{L^2}\nonumber \\
 &\leq C\|\nabla B_0\|_{L^2}^2+C_7C_0^{\delta_0/2}\sup_{0\le t\le  \si(T)} \|\nabla w_3\|_{L^2}^2,
\end{align}
where we have used \eqref{tdh} and \eqref{basic3} (with $H, H_0$ being replaced by $B, B_0$, respectively). Thus, it follows from \eqref{bw33} that
\begin{align}\label{bw34}
\quad \sup_{0\le t\le  \si(T)}\|\nabla w_3\|\ltwo+\int_0^{\sigma(T)}\int\rho|\dot{w}_3|^{2}dxdt \leq C\|\nabla B_0\|_{L^2}^2,
\end{align}
provided $C_0\leq \hat{\ve}_4 \triangleq \min\{\hat{\ve}_3, (2C_7)^{-2/\delta_0}\}$.
Similarly, multiplying \eqref{bw32} by $t$, integrating it over $(0, \delta(T))$, for $C_0\leq \hat{\ve}_4$, we obtain that
\begin{align}\label{bw35}
\quad \sup_{0\le t\le  \si(T)}(t\|\nabla w_3\|\ltwo)+\int_0^{\sigma(T)}t\int\rho|\dot{w}_3|^{2}dxdt \leq C\|B_0\|_{L^2}^2,
\end{align}
which we have used \eqref{key1}, \eqref{basic3} and \eqref{tdh}.
Since the solution operators $B_{0}\mapsto B$ and $B\mapsto w_3$ are linear, by the standard Stein-Weiss interpolation argument \cite{bl}, one can deduce from \eqref{bw33} and \eqref{bw34} that for any $\theta\in [s,1],$
\begin{align}\label{bw36}
\displaystyle  \sup_{0\le t\le\si(T)}t^{1-\theta}\|\nabla
w_3\|_{L^2}^2+\int_0^{\si(T)}t^{1-\theta}\int\rho|\dot
{w_3}|^2dxdt\leq C\| B_{0}\|_{H^\theta}^2,
\end{align}

Now let $w_{10}=u_0$ and $B_0=H_0$, so that $w_1+w_2+w_3=u$ and $B=H$,
we derive \eqref{uv1} from \eqref{bw14}, \eqref{bw23} and \eqref{bw36} directly under certain condition $C_0<\bar{\varepsilon}_3\triangleq\min\{\hat{\varepsilon}_1, \hat{\varepsilon}_2, \hat{\varepsilon}_4\}$.

In order to prove \eqref{uv2},  taking $m=2-s$ in \eqref{J01}, \eqref{ht4},  and integrating over $(0,\sigma(T)]$ instead of $(0,T]$, in a similar way as we have gotten \eqref{a20}, we obtain
\begin{align}\label{bu1}
&\sup_{0\le t\le \sigma(T)}\sigma^{2-s}(\|\rho^{\frac{1}{2}}\dot{u}\|_{L^2}^2+\|H_t\|_{L^2}^2)+\int_0^{\sigma(T)}\sigma^{2-s}(\|\nabla\dot{u}\|_{L^2}^2+\|\nabla H_t\|_{L^2}^2)dt \nonumber \\
&\leq C\int_0^{\sigma(T)}\sigma^{2-s} \|\nabla u\|^4_{L^4}dt+CC_0^{2 \delta_0/3}\sigma^{2-s}\|\curl^2  H\|\ltwo+C(\hat{\rho}, M_1, M_2).
\end{align}
where we have taken advantage of \eqref{uv1}. Next, from \eqref{h2xd1}, we have
\begin{align}\label{h2xd3}
& \quad \sup_{0\leq t\leq \sigma(T)}t^{2-s} \|\curl^2  H\|_{L^2}^2\nonumber \\
& \leq  C\sup_{0\leq t\leq \sigma(T)} t^{2-s}\|H_t\|_{L^2}^2 +C\sup_{0\leq t\leq \sigma(T)}(\sigma^{1-s}\|\nabla H\|_{L^2}^2)(\sigma^{(3-2s)/4}\|\nabla u\|_{L^2}^2)^2+CC_0^{1/2}\nonumber \\
& \leq  C\sup_{0\leq t\leq \sigma(T)} t^{2-s}\|H_t\|_{L^2}^2 +C(\hat{\rho}, M_1, M_2).
\end{align}
Combining \eqref{bu1} and \eqref{h2xd3}, we obtain
\begin{align}\label{bu2}
&\sup_{0\le t\le \sigma(T)}\sigma^{2-s}(\|\rho^{\frac{1}{2}}\dot{u}\|_{L^2}^2+\|H_t\|_{L^2}^2)+\int_0^{\sigma(T)}\sigma^{2-s}(\|\nabla\dot{u}\|_{L^2}^2+\|\nabla H_t\|_{L^2}^2)dt \nonumber \\
&\leq C\int_0^{\sigma(T)}\sigma^{2-s} \|\nabla u\|^4_{L^4}dt+C(\hat{\rho}, M_1, M_2),
\end{align}
provided $C_0\leq \ve_2$.
By \eqref{tdu2} and \eqref{uv1}, we have
\begin{align}\label{bu3}
&\quad\int_0^{\sigma(T)}t^{2-s}\|\nabla u\|_{L^{4}}^{4}dt \nonumber \\
& \le C\int_0^{\sigma(T)}t^{2-s}(\|\rho^{\frac{1}{2}}\dot{u}\|_{L^{2}}^{3}+\|\curl^2  H\|_{L^2}^3)(\|\nabla u\|_{L^{2}}+\|P-\bar{P}\|_{L^{2}}+\|\nabla H\|_{L^2})dt \nonumber \\
& \quad +C\int_0^{\sigma(T)}t^{2-s}(\|\nabla u\|_{L^{2}}^{4}+\|P-\bar{P}\|_{L^{2}}^{4}+\|P-\bar{P}\|_{L^{4}}^{4}+\|\nabla H\|_{L^2}^4+\||H|^2\|_{L^{4}}^{4})dt \nonumber \\
& \le CC_0^{\delta_0}\sup_{0\le t\le \sigma(T)}(t^{2-s}(\|\rho^{\frac{1}{2}}\dot{u}\|_{L^{2}}^{2}+\|\curl^2  H\|_{L^{2}}^{2}))+C,
\end{align}
since it follows from \eqref{key1} that for $s\in (1/2,1]$,
\begin{align}\label{bu3-1}
&\quad\int_0^{\sigma(T)}t^{2-s}\|\rho^{\frac{1}{2}}\dot{u}\|_{L^{2}}^{3}(\|\nabla u\|_{L^2}+\|\nabla H\|_{L^2})dt \nonumber \\
& \le C\int_0^{\sigma(T)}t^{\frac{2s-3}{4}}(t^{\frac{3-2s}{4}}\big(\|\nabla u\|_{L^2}^2+\|\nabla H\|_{L^2}^2)\big)^{\frac{1}{2}}(t^{\frac{3-2s}{4}}\|\rho^{\frac{1}{2}}\dot{u}\|_{L^{2}}^{2})^{\frac{1}{2}}(t^{2-s}\|\rho^{\frac{1}{2}}\dot{u}\|_{L^2}^2)dt \nonumber \\
& \le CC_0^{\delta_0}\sup_{0\le t\le \sigma(T)}(t^{2-s}\|\rho^{\frac{1}{2}}\dot{u}\|_{L^2}^2),
\end{align}
and
\begin{align}\label{bu3-2}
\int_0^{\sigma(T)}t^{2-s}\||H|^2\|_{L^{4}}^{4}dt
& \le C\int_0^{\sigma(T)}t^{2-s}\|H\|_{L^{\infty}}^{4}\|H\|_{L^{4}}^{4}dt \nonumber \\
& \le C\int_0^{\sigma(T)}t^{2-s}(\|\nabla H\|_{L^{2}}^{4}\|\curl^2  H\|_{L^{2}}^{2}+\|\nabla H\|_{L^{2}}^{6})dt \nonumber \\
& \le CC_0^{2\delta_0}\sup_{0\le t\le \sigma(T)}(t^{2-s}\|\curl^2  H\|_{L^2}^2)+C.
\end{align}
Then combining \eqref{bu2} and \eqref{bu3}, we have
\begin{align}\label{bu4}
&\sup_{0\le t\le \sigma(T)}\sigma^{2-s}(\|\rho^{\frac{1}{2}}\dot{u}\|_{L^2}^2+\|H_t\|_{L^2}^2)+\int_0^{\sigma(T)}\sigma^{2-s}(\|\nabla\dot{u}\|_{L^2}^2+\|\nabla H_t\|_{L^2}^2)dt \nonumber \\
&\leq C_8C_0^{\delta_0}\sup_{0\le t\le \sigma(T)}\sigma^{2-s}(\|\rho^{\frac{1}{2}}\dot{u}\|_{L^2}^2+\|H_t\|_{L^2}^2)+C(\hat{\rho}, M_1, M_2),
\end{align}
Therefore, if we choose $C_0$ to be such that $C_0\leq \ve_3 \triangleq \min\{\bar{\ve}_3, (2C_8)^{-1/\delta_0}\}$, \eqref{bu4} and \eqref{h2xd3} implies \eqref{uv2}. The proof of Lemma \ref{lem-a1} is completed.
\end{proof}

\begin{lemma}\label{lem-a3} If $(\rho,u,H)$ is a smooth solution of \eqref{CMHD}-\eqref{boundary} satisfying \eqref{key1} and the initial data condition \eqref{dt2}, then there exists a positive constant  $\varepsilon_4$   depending only on $\mu ,  \lambda , \nu,  \ga ,  a ,  \on, \hat{\rho}, s,  \Omega$, $M_1$ and $M_2$ such that
\begin{align}\label{ba3}
\displaystyle  A_4(\sigma(T))+A_5(\sigma(T))\leq C_0^{\delta_0},
\end{align}
provided $C_0<\varepsilon_4$.
\end{lemma}

\begin{proof}
Multiplying $\eqref{CMHD}_2$ by $3|u|u$, and integrating the resulting equation over $ \O$ lead  to
\begin{align}\label{ba33}
&\quad \left(\int\rho|u|^{3}dx\right)_t \nonumber \\
&\leq C\int|u||\nabla u|^{2}dx+C\int|P-\bar{P}||u||\nabla u|dx+C\int|H||\nabla H||u|^2dx \nonumber \\
&\leq C\|\nabla u\|_{L^2}^{\frac{5}{2}}(\|\rho\dot{u}\|_{L^2}^{\frac{1}{2}}+\|\curl^2  H\|_{L^2}^{\frac{1}{2}})+C\|\nabla u\|_{L^2}^{\frac{5}{2}}\|\nabla H\|_{L^2}^{\frac{1}{2}}+CC_0^{\frac{1}{12}}\|\nabla u\|_{L^2}^{\frac{5}{2}} \nonumber \\
&\quad +C\|\nabla u\|_{L^2}^3+CC_0^{\frac{1}{3}}\|\nabla u\|_{L^2}^2+C(\|\nabla H\|_{L^2}^4+\|\nabla u\|_{L^2}^4).
\end{align}
Hence, integrating \eqref{ba33} over $(0,\sigma(T))$ and using \eqref{key1}, \eqref{basic3}, we get
\begin{align}\label{ba34}
&\quad\sup_{0\le t\le  \si(T) }\int\rho|u|^{3}dx \nonumber\\
&\leq C\int_0^{\sigma(T)}(t^{\frac{3-2s}{4}}\|\nabla u\|_{L^2}^2)^{\frac{5}{4}} (t^\frac{3-2s}{4}(\|\rho\dot{u}\|_{L^2}^2+\|\curl^2  H\|_{L^2}^2))^{\frac{1}{4}} t^\frac{3(2s-3)}{8} dt \nonumber \\
&\quad + C\int_0^{\sigma(T)}(t^{\frac{3-2s}{4}}\|\nabla u\|_{L^2}^2)^{\frac{5}{4}} (t^\frac{3-2s}{4}\|\nabla H\|_{L^2}^2)^{\frac{1}{4}} t^\frac{3(2s-3)}{8} dt \nonumber \\
&\quad + CC_0^{\frac{1}{12}}\int_0^{\sigma(T)}(t^{\frac{3-2s}{4}}\|\nabla u\|_{L^2}^2)^{\frac{5}{4}} t^\frac{5(2s-3)}{16} dt+CC_0^{2\delta_0}+\int\rho_0|u_0|^3dx \nonumber \\
&\leq CC_0^{3 \delta_0/2}+\int\rho_0|u_0|^3dx \leq C_9C_0^{3 \delta_0/2},
\end{align}
where we have used the fact $\delta_0\in (0, 1/9]$, $s\in (1/2,1]$ and
\begin{align}\label{ba35}
 \displaystyle \int\rho_0|u_0|^{3}dx\leq C\|\rho_0^{\frac{1}{2}}u_0\|_{L^{2}}^{{3(2s-1)}/{2s}}\|u_0\|_{H^s}^{{3}/{2s}}\leq CC_0^{3 \delta_0/2},
 \end{align}
Finally, set $\varepsilon_4\triangleq\min\{\varepsilon_3,(C_9)^{-\frac{2}{\delta_0}}\}$, we get $A_5({\sigma(T)})\leq C_0^{\delta_0}$.
Next, it remains to estimate $A_4({\sigma(T)})$. Using \eqref{uv1}, we have
\begin{align}\label{ba4}
A_4({\sigma(T)}) \leq & \sup_{0\le t\le\sigma(T)} (t^{1-s}\|\nabla u\|\ltwo)^{\frac{2s+1}{4s}} \sup_{0\le t\le\sigma(T)} (t\|\nabla u\|\ltwo)^{\frac{2s-1}{4s}} \nonumber \\
& +\sup_{0\le t\le\sigma(T)} (t^{1-s}\|\nabla H\|\ltwo)^{\frac{2s+1}{4s}} \sup_{0\le t\le\sigma(T)} (t\|\nabla H\|\ltwo)^{\frac{2s-1}{4s}} \nonumber \\
& +\left(\int_0^{\sigma(T)} t^{1-s}\|\rho^{1/2}\dot{u}\|\ltwo dt \right)^{\frac{2s+1}{4s}} \left(\int_0^{\sigma(T)} t\|\rho^{1/2}\dot{u}\|\ltwo dt \right)^{\frac{2s-1}{4s}} \nonumber \\
& +\left(\int_0^{\sigma(T)} t^{1-s}\|\curl^2  H\|\ltwo dt \right)^{\frac{2s+1}{4s}} \left(\int_0^{\sigma(T)} t\|\curl^2  H\|\ltwo dt \right)^{\frac{2s-1}{4s}} \nonumber \\
\leq & CA_1(T)^{\frac{2s-1}{4s}}\leq CC_0^{\frac{2s-1}{8s}}\leq C_0^{\delta_0},
\end{align}
provided $C_0\leq \ve_4$ and $\delta_0 \triangleq \frac{2s-1}{9s} < \frac{2s-1}{8s} $. The proof of Lemma \ref{lem-a3} is completed.
\end{proof}

\begin{lemma}\label{lem-a1a2} Let $(\rho,u,H)$ be a smooth solution of
 \eqref{CMHD}-\eqref{boundary} on $\O \times (0,T] $ satisfying \eqref{key1}. Then there exists a positive constant $\varepsilon_5$ depending only  on $\mu,$  $\lambda,$ $\nu,$ $\gamma,$ $a$, $s$, $\on$, $\hat{\rho}$, $\Omega$, $M_1$ and $ M_2$  such that
 \begin{align}\label{a1a2}
 \displaystyle  A_1(T)+A_2(T)\le C_0^{\frac{1}{2}},
 \end{align}
provided $C_0\leq\varepsilon_5$.
\end{lemma}

\begin{proof} By \eqref{g1} and \eqref{basic2}, one can check that
\begin{align}\label{al3}
\int_0^{T}\sigma \|\nabla u\|_{L^3}^3 dt
&\le  C \int_0^{T}\sigma(\|\nabla u\|_{L^2}\|\nabla u\|_{L^4}^2+\|\nabla u\|_{L^2}^3)dt
\nonumber\\
&\le  C C_0+C \int_0^{T} \sigma^2\|\nabla u\|_{L^4}^4dt,
\end{align}
which, along with \eqref{a01} and \eqref{a02} gives
\begin{align}\label{a1a2-1}
\displaystyle  A_1(T)+A_2(T)\leq C(C_0^{1/2+\delta_0/2}+\int_0^T\sigma^2\|\nabla u\|_{L^{4}}^{4}dt).
\end{align}
So it reduces to estimate $\int_0^T\sigma^2\|\nabla u\|_{L^{4}}^{4}dt$.
On one hand, by \eqref{tdu2}, \eqref{key1}, \eqref{rho2-int} and Lemma \ref{lem-basic} again,  it indicates that
\begin{align}\label{al4-l}
&\quad\int_0^{\sigma(T)}t^{2}\|\nabla u\|_{L^{4}}^{4}dt \nonumber \\
& \le C\int_0^{\sigma(T)}t^{2}(\|\rho^{\frac{1}{2}}\dot{u}\|_{L^{2}}^{3}+\|\curl^2  H\|_{L^2}^3)(\|\nabla u\|_{L^{2}}+\|P-\bar{P}\|_{L^{2}}+\|\nabla H\|_{L^2})dt \nonumber \\
& \quad +C\int_0^{\sigma(T)}t^{2}(\|\nabla u\|_{L^{2}}^{4}+\|P-\bar{P}\|_{L^{2}}^{4}+\|P-\bar{P}\|_{L^{4}}^{4}+\|\nabla H\|_{L^2}^4+\||H|^2\|_{L^{4}}^{4})dt \nonumber \\
& \le CC_0^{1/2+\delta_0}+CC_0+CC_0^{1/2+2\delta_0}  \le CC_0^{1/2+\delta_0},
\end{align}
since it follows from \eqref{key1} that for $s\in (1/2,1]$,
\begin{align}\label{al4-l1}
&\quad\int_0^{\sigma(T)}t^{2}\|\rho^{\frac{1}{2}}\dot{u}\|_{L^{2}}^{3}(\|\nabla u\|_{L^2}+\|\nabla H\|_{L^2})dt \nonumber \\
& \le C\int_0^{\sigma(T)}t^{\frac{2s-3}{4}}(t^{\frac{3-2s}{4}}\big(\|\nabla u\|_{L^2}^2+\|\nabla H\|_{L^2}^2)\big)^{\frac{1}{2}}(t^{\frac{3-2s}{4}}\|\rho^{\frac{1}{2}}\dot{u}\|_{L^{2}}^{2})^{\frac{1}{2}}(t^{2}\|\rho^{\frac{1}{2}}\dot{u}\|_{L^2}^2)dt \nonumber \\
& \le CC_0^{1/2+\delta_0},
\end{align}
and
\begin{align}\label{al4-l2}
&\quad\int_0^{\sigma(T)}t^{2}\||H|^2\|_{L^{4}}^{4}dt
\le C\int_0^{\sigma(T)}t^{2}\|H\|_{L^{3}}^{2}\|\nabla H\|_{L^{2}}^{4}\|\nabla H\|_{L^{6}}^{2}dt \nonumber \\
& \le C\int_0^{\sigma(T)}t^{2}(\|\nabla H\|_{L^{2}}^{4}\|\curl^2  H\|_{L^{2}}^{2}+\|\nabla H\|_{L^{2}}^{6})dt  \le CC_0^{1/2+2\delta_0}.
\end{align}
On the other hand, by \eqref{key1}, \eqref{basic2} and Lemma \ref{lem-a1}, we have
\begin{align}\label{al4-h}
&\quad\int_{\sigma(T)}^{T}\sigma^{2}\|\nabla u\|_{L^{4}}^{4}dt \nonumber \\
& \le C\int_{\sigma(T)}^{T}\sigma^{2}(\|\rho^{\frac{1}{2}}\dot{u}\|_{L^{2}}^{3}+\|\curl^2  H\|_{L^2}^3)(\|\nabla u\|_{L^{2}}+\|P-\bar{P}\|_{L^{2}}+\|\nabla H\|_{L^2})dt \nonumber \\
& \quad +C\int_{\sigma(T)}^{T}\sigma^{2}(\|\nabla u\|_{L^{2}}^{4}+\|P-\bar{P}\|_{L^{2}}^{4}+\|P-\bar{P}\|_{L^{4}}^{4}+\|\nabla H\|_{L^2}^4+\||H|^2\|_{L^{4}}^{4})dt \nonumber \\
& \le CC_0+CC_0^{1/2+2\delta_0} \le CC_0^{1/2+2\delta_0}.
\end{align}
Combining \eqref{al4-l} and \eqref{al4-h}, it follows from \eqref{a1a2-1} that
\begin{align}\label{a1a2-2}
\displaystyle  A_1(T)+A_2(T)\leq C_{10}C_0^{1/2+\delta_0/2}.
\end{align}
Set $\varepsilon_5\triangleq\min\{\varepsilon_4,(C_{10}^{-2/\delta_0}\}$, from \eqref{a1a2-2}, \eqref{a1a2} holds when $C_0<\varepsilon_5$. The proof of Lemma \ref{lem-a1a2} is completed.
\end{proof}

We now proceed to proof the uniform (in time) upper bound for the
density.

\begin{lemma}\label{lem-brho}
Let $(\rho,u,H)$ be a smooth solution of
 \eqref{CMHD}-\eqref{boundary} on $\O \times (0,T] $ satisfying \eqref{key1}. Then there exists a positive constant $\varepsilon_6$ depending only  on $\mu,$  $\lambda,$ $\nu,$ $\gamma,$ $a$, $s$, $\on$, $\hat{\rho}$, $\Omega$, $M_1$ and $ M_2$  such that
 \begin{align}\label{brho}
 \displaystyle  \sup_{0\le t\le T}\|\n(t)\|_{L^\infty}  \le
\frac{7\hat{\rho} }{4}  ,
 \end{align}
provided $C_0\le \ve_6. $
\end{lemma}

\begin{proof}
First, the equation of  mass conservation $\eqref{CMHD}_1$ can be equivalently rewritten in the form
\begin{align}\label{rho1}
\displaystyle  D_t \n=g(\rho)+b'(t),
\end{align}
where
\begin{align}
 \displaystyle D_t\rho\triangleq\rho_t+u \cdot\nabla \rho ,\,
g(\rho)\triangleq-\frac{\rho(P-\bar{P})}{2\mu+\lambda},\,
 b(t)\triangleq-\frac{1}{2\mu+\lambda} \int_0^t\rho (F+\frac{|H|^2}{2})dt.
 \end{align}
Naturally, we shall prove our conclusion by Lemma \ref{lem-z}. It is sufficient to check that the function $b(t)$ must verify \eqref{a100} with some suitable constants $N_0$, $N_1$.

For $t\in[0,\sigma(T)],$ one deduces from \eqref{g1}, \eqref{g2}, \eqref{tdf1}, \eqref{tdxd-u1}, \eqref{udot}, \eqref{key1} and Lemmas \ref{lem-basic}, \ref{lem-a1} that for $\delta_0$ as in Proposition \ref{pr1} and for all $0\leq t_1\leq t_2\leq\sigma(T)$,
\begin{align}\label{bl1}
&\quad |b(t_2)-b(t_1)| =\frac{1}{\lambda+2\mu}\left|\int_{t_1}^{t_2}\rho (F+\frac{|H|^2}{2})dt\right|\le C\int_0^{\sigma(T)}(\|F\|_{L^{\infty}}+ \|H\|^2_{L^{\infty}})dt \nonumber\\
& \le C\int_0^{\sigma(T)}\|F\|_{L^{6}}^{\frac{1}{2}}\|\nabla F\|_{L^{6}}^{\frac{1}{2}}dt+C\int_0^{\sigma(T)}\|F\|_{L^{2}}dt+C\int_0^{\sigma(T)}\|H\|_{L^{6}}\|\nabla H\|_{L^{6}} dt \nonumber\\
& \leq C\int_0^{\sigma(T)}(\|\rho^{\frac{1}{2}}\dot{u}\|_{L^2}^{\frac{1}{2}}+\|\curl^2  H\|_{L^2}^{\frac{1}{2}})\|\nabla \dot{u}\|_{L^{2}}^{\frac{1}{2}}dt\nonumber\\
& \quad + C\int_0^{\sigma(T)}(\|\rho^{\frac{1}{2}}\dot{u}\|_{L^2}^{\frac{1}{2}}+\|\curl^2  H\|_{L^2}^{\frac{1}{2}})\|\nabla H\|_{L^{2}}^{\frac{1}{4}}\|\curl^2  H\|_{L^{2}}^{\frac{3}{4}}dt\nonumber\\
& \quad + C\int_0^{\sigma(T)}(\|\rho^{\frac{1}{2}}\dot{u}\|_{L^2}^{\frac{1}{2}}+\|\curl^2  H\|_{L^2}^{\frac{1}{2}})(\|\nabla u\|_{L^{2}}+\|\nabla H\|_{L^{2}})dt\nonumber\\
& \quad + C\int_0^{\sigma(T)}(\|\nabla u\|_{L^2}^{\frac{1}{2}}+\|\nabla H\|_{L^2}^{\frac{1}{2}}+\|P-\bar{P}\|_{L^2}^{\frac{1}{2}})\|\nabla \dot{u}\|_{L^{2}}^{\frac{1}{2}}dt\nonumber\\
& \quad + C\int_0^{\sigma(T)}(\|\nabla u\|_{L^2}^{\frac{1}{2}}+\|\nabla H\|_{L^2}^{\frac{1}{2}}+\|P-\bar{P}\|_{L^2}^{\frac{1}{2}})\|\nabla H\|_{L^{2}}^{\frac{1}{4}}\|\curl^2  H\|_{L^{2}}^{\frac{3}{4}}dt\nonumber\\
& \quad + C\int_0^{\sigma(T)}(\|\nabla u\|_{L^2}^{\frac{1}{2}}+\|\nabla H\|_{L^2}^{\frac{1}{2}}+\|P-\bar{P}\|_{L^2}^{\frac{1}{2}})(\|\nabla u\|_{L^{2}}+\|\nabla H\|_{L^{2}})dt\nonumber\\
& \quad + C\int_0^{\sigma(T)}(\|\nabla u\|_{L^2}+\|\nabla H\|_{L^2}+\|P-\bar{P}\|_{L^2})dt\nonumber\\
& \quad + C\int_0^{\sigma(T)}\|\nabla H\|_{L^2}(\|\nabla H\|_{L^2}+\|\curl^2  H\|_{L^2})dt \triangleq \sum_{i=1}^8 B_i.
\end{align}
We have to estimate $B_i, i=1,2,\cdots,8$ one by one. A directly computation gives
\begin{align}\label{bl-b1}
B_1 &\leq C\int_0^{\sigma(T)}\big(t^{\frac{3-2s}{4}}(\|\rho^{\frac{1}{2}}\dot{u}\|_{L^2}^2+\|\curl^2  H\|_{L^2}^2)\big)^{\frac{1}{4}}\big(t^{2-s}\|\nabla \dot{u}\|_{L^{2}}^2\big)^{\frac{1}{4}}t^{\frac{6s-11}{16}}dt\nonumber\\
&\leq C\left(\int_0^{\sigma(T)}t^{\frac{3-2s}{4}}(\|\rho^{\frac{1}{2}}\dot{u}\|_{L^2}^2+\|\curl^2  H\|_{L^2}^2)dt\right)^{\frac{1}{4}} \nonumber \\
 & \qquad \times \left(\int_0^{\sigma(T)}t^{2-s}\|\nabla \dot{u}\|_{L^{2}}^2dt\right)^{\frac{1}{4}}\left(\int_0^{\sigma(T)}t^{\frac{6s-11}{8}}dt\right)^{\frac{1}{2}} \leq CC_0^{\delta_0/4},
\end{align}
similarly,
\begin{align}
B_2 &\leq C\int_0^{\sigma(T)}\big(t^{2-s}(\|\rho^{\frac{1}{2}}\dot{u}\|_{L^2}^2+\|\curl^2  H\|_{L^2}^2)\big)^{\frac{1}{4}}\big(t^{\frac{3-2s}{4}}\|\nabla H\|_{L^{2}}^2\big)^{\frac{1}{8}}\nonumber \\
 & \qquad \times\big(t^{\frac{3-2s}{4}}\|\curl^2  H\|_{L^{2}}^2\big)^{\frac{3}{8}}t^{\frac{4s-7}{8}}dt
 \leq CC_0^{\delta_0/2},\label{bl-b2}\\
B_3 &\leq C\int_0^{\sigma(T)}\big(t^{\frac{3-2s}{4}}(\|\rho^{\frac{1}{2}}\dot{u}\|_{L^2}^2+\|\curl^2  H\|_{L^2}^2)\big)^{\frac{1}{4}}\nonumber \\
 & \qquad \times\big(t^{\frac{3-2s}{4}}(\|\nabla u\|_{L^{2}}^2+\|\nabla H\|_{L^{2}}^2)\big)^{\frac{1}{2}}t^{\frac{3(2s-3)}{8}}dt
 \leq CC_0^{3\delta_0/4},\label{bl-b3}\\
B_4 &\leq C\int_0^{\sigma(T)}\big(\|\nabla u\|_{L^2}^{2}+\|\nabla H\|_{L^2}^{2}+\|P-\bar{P}\|_{L^2}^{2}\big)^{\frac{1}{4}}\big(t^{2-s}\|\nabla \dot{u}\|_{L^{2}}^2\big)^{\frac{1}{4}}t^{\frac{s-2}{4}}dt\nonumber\\
&\leq CC_0^{1/4},\label{bl-b4}\\
B_5 &\leq C\int_0^{\sigma(T)}\big(t^{\frac{3-2s}{4}}(\|\nabla u\|_{L^2}^{2}+\|\nabla H\|_{L^2}^{2})\big)^{\frac{1}{4}}\big(t^{\frac{3-2s}{4}}\|\nabla H\|_{L^{2}}^2\big)^{\frac{1}{8}}\nonumber \\
 & \qquad \times\big(t^{\frac{3-2s}{4}}\|\curl^2  H\|_{L^{2}}^2\big)^{\frac{3}{8}}t^{\frac{3(2s-3)}{16}}dt \nonumber \\
 & \quad +CC_0^{\frac{1}{4}}\int_0^{\sigma(T)}\big(t^{\frac{3-2s}{4}}\|\nabla H\|_{L^{2}}^2\big)^{\frac{1}{8}}\big(t^{\frac{3-2s}{4}}\|\curl^2  H\|_{L^{2}}^2\big)^{\frac{3}{8}}t^{\frac{2s-3}{8}}dt\nonumber\\
&\leq CC_0^{3\delta_0/4}+CC_0^{1/4+\delta_0/2}\leq CC_0^{3\delta_0/4},\label{bl-b5}\\
B_6 &\leq C\int_0^{\sigma(T)}\big(t^{\frac{3-2s}{4}}(\|\nabla u\|_{L^2}^{2}+\|\nabla H\|_{L^2}^{2})\big)^{\frac{3}{4}}t^{\frac{3(2s-3)}{16}}dt \nonumber \\
 & +CC_0^{\frac{1}{4}}\int_0^{\sigma(T)}\big(t^{\frac{3-2s}{4}}(\|\nabla u\|_{L^2}^{2}+\|\nabla H\|_{L^2}^{2})\big)^{\frac{1}{2}}t^{\frac{2s-3}{8}}dt\nonumber\\
&\leq CC_0^{3\delta_0/4}+CC_0^{1/4+\delta_0/2}\leq CC_0^{3\delta_0/4},\label{bl-b6}\\
B_7 &  =\int_0^{\sigma(T)}(\|\nabla u\|_{L^2}+\|\nabla H\|_{L^2}+\|P-\bar{P}\|_{L^2})dt \leq CC_0^{1/2},\label{bl-b7}\\
B_8 &\leq \int_0^{\sigma(T)}\big(t^{\frac{3-2s}{4}}\|\nabla H\|_{L^{2}}^2\big)^{\frac{1}{2}}\big(t^{\frac{3-2s}{4}}\|\curl^2  H\|_{L^{2}}^2\big)^{\frac{1}{2}}t^{\frac{2s-3}{4}}dt+CC_0dt \nonumber \\&
\leq CC_0^{1/2+\delta_0/2}.\label{bl-b8}
\end{align}
Putting \eqref{bl-b1}-\eqref{bl-b8} into \eqref{bl1}, we have
\begin{align}\label{bl2}
|b(t_2)-b(t_1)| \leq C_{11}C_0^{\delta_0/4}.
\end{align}
Combining \eqref{bl2} with \eqref{rho1} and choosing $N_1=0$, $N_0=C_{11}C_0^{\delta_0/4}$, $\bar{\zeta}=\hat{\rho}$ in Lemma \ref{lem-z} give
\begin{align}\label{rho2}
\displaystyle  \sup_{t\in [0,\si(T)]}\|\rho\|_{L^\infty} \le \hat{\rho}
+C_{11}C_0^{\delta_0/4} \le\frac{3\hat{\rho}}{2},
\end{align}
provided $C_0\le \hat{\ve}_6\triangleq\min\{\varepsilon_5, \left(\frac{\hat{\rho}}{2C_{11}}\right)^{\frac{4}{\delta_0}}\}. $

On the other hand, for $t\in[\sigma(T),T],\,\,\sigma(T)\le t_1\le t_2\le T ,$ it follows from \eqref{tdf1}, \eqref{udot}, \eqref{key1}, \eqref{rho2-int} and Lemma \ref{lem-basic} that
\begin{align}\label{br1}
& \quad |b(t_2)-b(t_1)| \le C\int_{t_1}^{t_2}(\|F\|_{L^{\infty}}+\|H\|^2_{L^{\infty}})dt \nonumber\\
&\le \frac{a}{\lambda+2\mu}(t_2-t_1)+C\int_{t_1}^{t_2}\|F\|_{L^{\infty}}^{8/3}dt +C\int_{t_1}^{t_2}\|H\|_{L^{\infty}}^2dt \nonumber\\
& \le \frac{a}{\lambda+2\mu}(t_2-t_1)+CC_0^{\frac{1}{6}}\int_{\sigma(T)}^{T}(\|\nabla \dot{u}\|_{L^2}^{2}+\|\nabla H\|_{L^2}\|\curl^2  H\|_{L^2}^3+\|\nabla H\|_{L^2}^4)dt\nonumber\\
&\quad +CC_0 +C\int_{t_1}^{t_2}(\|\nabla H\|_{L^2}\|\curl^2  H\|_{L^2}+\|\nabla H\|_{L^2}^2)dt \nonumber \\
& \le \frac{a}{\lambda+2\mu}(t_2-t_1)+C_{12}C_0^{2/3}.
\end{align}
Now we choose $N_0=C_{12}C_0^{2/3}$, $N_1=\frac{a}{\lambda+2\mu}$ in \eqref{a100} and set $\bar\zeta= \frac{3\hat{\rho}}{2}$ in \eqref{a101}. Since for all $  \zeta \geq\bar{\zeta}=\frac{3\hat{\rho}}{2}>\bar{\rho}+1$,
$$ g(\zeta)=-\frac{ a\zeta}{2\mu+\lambda}(\zeta^{\gamma}-\bar{\rho}^{\gamma})\le -\frac{a}{\lambda+2\mu}= -N_1. $$
Together with \eqref{rho1} and \eqref{br1}, by Lemma \ref{lem-z}, we have
\begin{align}\label{rho3}
\displaystyle \sup_{t\in
[\si(T),T]}\|\rho\|_{L^\infty}\le \frac{ 3\hat \rho }{2} +C_{12}C_0^{2/3} \le
\frac{7\hat \rho }{4} ,
\end{align}
provided $C_0\le \ve_6 \triangleq\min\{\hat{\ve}_6, (\frac{ \hat \n }{4C_{12}})^{3/2}\}$.
The combination of \eqref{rho2} with \eqref{rho3} completes the
proof of Lemma \ref{lem-brho}.
\end{proof}

\section{\label{se4} A priori estimates (II): higher order estimates }
In this section, we derive the time-dependent higher order estimates, which are necessary for the global existence of classical solutions. Here we adopt the method of the article \cite{cl2019,lxz2013,jx01}, and follow their work with a few modifications. We sketch it here for completeness. Let $(\rho,u,H)$ be a smooth solution of \eqref{CMHD}-\eqref{boundary} satisfying Proposition \ref{pr1} and the initial energy $C_0\leq \ve_6$, and the positive constant $C $ may depend on $T,$ $\mu$, $\lambda$, $\nu$, $a$, $\ga$, $\on,$ $\hat{\rho},$  $s,$  $\Omega$, $M_1, M_2$,  $\|\na u_0\|_{H^1}, \|\na H_0\|_{H^1}, \|\n_0-\bar{\rho}\|_{W^{2,q}},  \| g\|_{L^2}, \|P(\n_0)-\bar{P}\|_{W^{2,q}}$ for $q\in(3,6)$ where $g\in L^2(\Omega)$ is given by compatibility condition \eqref{dt3}.

\begin{lemma}\label{lem-x1}
 There exists a positive constant $C,$ such that
\begin{align}
&\sup_{0\le t\le T}(\|\nabla u\|_{L^2}^2+\|\nabla H\|_{L^2}^2)+\int_0^T(\|\rho^{\frac{1}{2}}\dot{u}\|_{L^2}^2+\| H_t\|_{L^2}^2+\|\nabla^2 H\|_{L^2}^2)dt\leq C,\label{x1b1}\\
&\sup_{0\le t\le T}(\|\rho^{\frac{1}{2}}\dot{u}\|_{L^2}^2+\| H_t\|_{L^2}^2+\|\nabla^2 H\|_{L^2}^2)+\int_0^T(\|\nabla\dot{u}\|_{L^2}^{2}+\|\nabla H_t\|_{L^2}^2)dt\leq C,\label{x1b2}\\
&\sup_{0\le t\le T}(\|\nabla\rho\|_{L^6}+\|u\|_{H^2})+\int_0^T(\|\nabla u\|_{L^\infty}+\|\nabla^{2} u\|_{L^6}^{2})dt\leq C.\label{x2b1}
\end{align}
\end{lemma}
\begin{proof}
First, taking $\theta=1$ in \eqref{tdh-s} and taking $s=1$ in \eqref{uv1} along with \eqref{2tdh} gives \eqref{x1b1}.
Then choosing $m=0$ in \eqref{J01} and \eqref{ht3}, integrating them over $(0,T)$, by \eqref{J0b1}, \eqref{x1b1} and the compatibility condition \eqref{dt3}, we have
\begin{align}\label{x1b4}
\displaystyle  &\quad\sup_{0\le t\le T}(\|\rho^{\frac{1}{2}}\dot{u}\|_{L^2}^2+\| H_t\|_{L^2}^2+\|\nabla^2 H\|_{L^2}^2)+\int_0^T(\|\nabla\dot{u}\|_{L^2}^{2}+\|\nabla H_t\|_{L^2}^2)dt \nonumber\\
& \leq C+C\int_0^T (\|\rho^{\frac{1}{2}}\dot{u}\|_{L^2}^3+\|\nabla^2 H\|_{L^2}^3+\|\nabla H\|_{L^2}^4\|\nabla^2 H\|_{L^2}^2) dt \nonumber \\ 
&\leq C+\frac{1}{2}\sup_{0\le t\le T}(\|\rho^{\frac{1}{2}}\dot{u}\|_{L^2}^2+\|\nabla^2 H\|_{L^2}^2),
\end{align}
where we have also used Lemma \ref{lem-gn}, Lemma \ref{lem-f-td}, \eqref{h-tdh} and \eqref{h2xd1}, then we deduce \eqref{x1b2} from \eqref{x1b4}.
Based on the Beale-Kato-Majda type inequality (see Lemma \ref{lem-bkm}), we can derive \eqref{x2b1}, the estimates on the gradient of density and velocity, in arguments similar to \cite{cl2019}. This finishes the proof.
\end{proof}

\begin{lemma}\label{lem-x3}
There exists a positive constant $C$ such that
\begin{align}
& \sup_{0\le t\le T}\|\rho^{\frac{1}{2}}u_t\|_{L^2}^2 + \ia\int|\nabla u_t|^2dxdt\le C, \label{x3b}\\
& \sup\limits_{0\le t\le T}\left(\|{\rho\!-\!\! \bar{\rho}}\|_{H^2}\! +\!
 \|{P\!-\!\! \bar{P}}\|_{H^2}\!+\!
   \|\rho_t\|_{H^1}\!+\!\|P_t\|_{H^1}\right)
    \!+\!\! \int_0^T\!\!\left(\|\n_{tt}\|_{L^2}^2\!+\!\|P_{tt}\|_{L^2}^2\right)dt
\le C, \label{x4b} \\
& \sup\limits_{0\le t\le T}\sigma (\|\nabla u_t\|_{L^2}^2+\|\nabla H_t\|_{L^2}^2)
    + \int_0^T\sigma(\|\rho^{\frac{1}{2}}u_{tt}\|_{L^2}^2+\|H_{tt}\|_{L^2}^2)dt
\le C.\label{x4bb}
\end{align}
\end{lemma}
\begin{proof} Based on Lemma \ref{lem-x1}, \eqref{x3b}-\eqref{x4b} can be obtained by the same method as that in \cite{cl2019}.
It remains to prove \eqref{x4bb}. Introducing the function
$$K(t)=(\lambda+2\mu)\int(\div u_t)^{2}dx+\mu\int|\omega_t|^{2}dx +\nu \int|\curl H_t|^{2}dx.$$
Since $u_t\cdot n = 0, H_t \cdot n=0$ on $\partial\Omega$, by Lemma \ref{lem-vn}, we have
\begin{align}\label{x4b6}
\displaystyle  \|\nabla u_t\|_{L^2}^2+\|\nabla H_t\|_{L^2}^2\leq C(\Omega)K(t).
\end{align}
Differentiating  $\eqref{CMHD}_{2,3}$  with respect to $t,$
\begin{align}\label{utt}\rho u_{tt}\!-\!(\lambda\!+\!2 \mu)\nabla \div u_t\!+\!\mu \nabla \!\times\! \omega_t\!=-\nabla\! P_t\!-\!\rho_t u_t\!-\!(\rho u\! \cdot\! \nabla\! u)_t\!+\!(H\! \cdot\! \nabla\! H\!-\!\nabla |H|^2/2)_t, \end{align} 
and 
\begin{align}\label{htt} \quad H_{tt}-\nu \nabla \times \curl H_t =(H \cdot \nabla u-u \cdot \nabla H-H\div u)_t, \end{align} 
then multiplying \eqref{utt} by $2u_{tt}$, multiplying \eqref{htt} $2H_{tt}$ respectively, we obtain
\begin{align}\label{x4b7}
\displaystyle &\quad\frac{d}{dt}K(t)+2\int(\rho|u_{tt}|^2+|H_{tt}|^2)dx \nonumber \\
&=\frac{d}{dt}\Big(-\int\rho_t|u_t|^{2}dx-2\int\rho_tu\cdot\nabla u\cdot u_tdx+2\int P_t\div u_tdx \nonumber \\
&\qquad\quad -\int (2(H \otimes H)_t : \nabla u_t-|H|^2_t\div u_t dx)\Big)\nonumber \\
&\quad +\int\rho_{tt}|u_t|^{2}dx + 2\int(\rho_tu\cdot\nabla u)_t\cdot u_tdx-2\int\rho (u\cdot\nabla u)_t\cdot u_{tt}dx \nonumber\\
&\quad - 2\int P_{tt}\div u_tdx+\int (2(H \otimes H)_{tt} : \nabla u_t-|H|^2_{tt}\div u_t) dx\nonumber\\
&\quad +2\int (H \cdot \nabla u-u \cdot \nabla H- H \div u)_t \cdot H_{tt} dx\nonumber\\
&\triangleq\frac{d}{dt}K_0 + \sum\limits_{i=1}^6 K_i .
\end{align}
Let us estimate $K_i$, $i=0,1,\cdots, 6.$
We conclude from $\eqref{CMHD}_1$, \eqref{udot}, \eqref{x1b2}, \eqref{x2b1}, \eqref{x3b}, \eqref{x4b}, \eqref{x4b6} and Sobolev's, Poincar\'{e}'s  inequalities that
\begin{align}\label{x4k0}
K_0 & \le \left|\int{\rm div}(\rho u)\,|u_t|^2dx\right|+C\norm[L^3]{\rho_t}\| u\|_{L^\infty}\|\nabla u\|_{L^2}\norm[L^6]{u_t}+C\|P_t\|_{L^2}\|\nabla u_t\|_{L^2} \nonumber \\
&\quad + C \|H\|_{L^\infty}\|H_t\|_{L^2}\|\nabla u_t\|_{L^2} \nonumber \\
&\le \frac{1}{2}K(t)+C,
\end{align}
\begin{align} \label{x4k1}
K_1 &\leq \left|\int_{ }\rho_{tt}\, |u_t|^2 dx\right|= \left|\int_{ }\div(\rho u)_t\,|u_t|^2 dx\right|= 2\left|\int_{ }(\rho_tu + \rho u_t)\cdot\nabla u_t\cdot u_tdx\right|\nonumber\\
& \le C\|\nabla u_t\|_{L^2}^2 K(t)+C\|\nabla u_t\|_{L^2}^2+C,\\
K_2&+K_3+K_4 
 \le C\norm[L^2]{\rho_{tt}}^2 + C\norm[L^2]{\nabla u_t}^2+\norm[L^2]{\rho^{{1/2}}u_{tt}}^2 +C\norm[L^2]{P_{tt}}^2 +C,\label{x4k2} \\
K_5 
 &\leq \frac{1}{2}\|H_{tt}\|_{L^2}^2+C\|H_{t}\|_{L^2}^2K(t)+C(\|\nabla H_t\|_{L^2}^2+\|\nabla u_t\|_{L^2}^2),\label{x4k5}\\
K_6 
 &\leq \frac{1}{2}\|H_{tt}\|_{L^2}^2+C(\|\nabla H_t\|_{L^2}^2+\|\nabla u_t\|_{L^2}^2).\label{x4k6}
\end{align}
Consequently, multiplying \eqref {x4b7} by $\sigma$, together with \eqref{x4k1}-\eqref{x4k6}, we get
\begin{align}\label{x4b8}
\displaystyle  &\quad\frac{d}{dt}(\sigma K(t)-\sigma K_0)+\sigma\int(\rho|u_{tt}|^{2}+|H_{tt}|^2)dx \nonumber \\
&\le C(1+\|\nabla u_t\|_{L^2}^2)\sigma K(t)+C(1+\|\nabla u_t\|_{L^2}^2+\|\nabla H_t\|_{L^2}^2+\|\rho_{tt}\|_{L^2}^2+\|P_{tt}\|_{L^2}^2),
\end{align}
By Gronwall's inequality, \eqref{x1b2}, \eqref{x3b}, \eqref{x4b} and \eqref{x4k0}, we derive that
\begin{align}\label{x4b9}
\displaystyle \sup_{0\le t\le T}(\sigma K(t))+\int_0^T\sigma(\|\rho^{\frac{1}{2}}u_{tt}\|_{L^2}^2+\|H_{tt}\|_{L^2}^2)dt\le C .
\end{align}
As a result, by \eqref{x4b6}, we get \eqref{x4bb}. This finishes the proof .
\end{proof}

\begin{lemma}\label{lem-x5}
There exists a positive constant $C$ so that for any $q\in(3,6),$
\begin{align}
&  \sup_{t\in[0,T]}\left(\|\rho- \bar{\rho}\|_{W^{2,q}} +\|P-\bar{P}\|_{W^{2,q}}\right)\le C,\label{x5b}\\
& \sup_{t\in[0,T]} \si (\|\nabla u\|_{H^2}^2+\|\nabla H\|_{H^2}^2)\nonumber \\
 & \qquad  +\int_0^T \left(\|\nabla u\|_{H^2}^2+\|\nabla H\|_{H^2}^2 +\|\na^2 u\|^{p_0}_{W^{1,q}}+\si\|\na u_t\|_{H^1}^2\right)dt\le C,\label{x5bb}
\end{align}
where $p_0=\frac{9q-6}{10q-12}\in(1,\frac{7}{6}).$
\end{lemma}
\begin{proof}
Let's start with \eqref{x5bb}.  By Lemma \ref{lem-x1} and Poincar\'{e}'s, Sobolev's inequalities, one can check that
\begin{align}\label{x5b1}
\|\nabla (\n \dot u) \|_{L^2}&\le \||\nabla \n ||  u_t|  \|_{L^2}\!+\! \|\n\nabla   u_t  \|_{L^2}\! +\! \||\nabla \n|| u||\nabla u| \|_{L^2}\! +\! \|\n|\nabla  u|^2\|_{L^2}\!+\! \|  \n |u || \nabla^2 u| \|_{L^2}\nonumber \\
 &\le C+C\| \nabla   u_t  \|_{L^2}.
\end{align}
Consequently, together with \eqref{x4b} and Lemma \ref{lem-x1}, it yields
\begin{align}\label{x5b2}
\|\nabla^2 u\|_{H^1} &\le C (\|\rho \dot u\|_{H^1}+\|H \cdot \nabla H\|_{H^1}+ \| P-\bar{P}\|_{H^2}+ \| |H|^2\|_{H^2}+\|u\|_{L^2})\nonumber \\
 &\le C+C \|\na  u_t\|_{L^2}.
\end{align}
It then follows from \eqref{x5b2}, \eqref{x2b1}, \eqref{x3b} and \eqref{x4bb} that
\begin{align}\label{x5b3}
\displaystyle \sup\limits_{0\le
t\le T}\si\|\nabla  u\|_{H^2}^2+\ia \|\nabla  u\|_{H^2}^2dt \le
 C.
\end{align}
Next, from \eqref{CMHD1}$_3$, \eqref{3tdh}, it follows
\begin{align}\label{x5h1}
\|\nabla^2 H\|_{H^1} 
&\le C (\|H_t\|_{H^1}+\|u \cdot \nabla H\|_{H^1}+ \|H \cdot \nabla u\|_{H^1}+ \|H \div u\|_{H^1}+\|\nabla H\|_{L^2})\nonumber \\
 &\le C+C \|\nabla H_t\|_{L^2}.
\end{align}
Similarly, from \eqref{x5b1}, \eqref{x1b1} and \eqref{x2b1}, we obtain
\begin{align}\label{x5h2}
\displaystyle \sup\limits_{0\le
t\le T}\si\|\nabla  H\|_{H^2}^2+\ia \|\nabla  H\|_{H^2}^2dt \le
 C.
\end{align}
 Next, we deduce from Lemma \ref{lem-x1} and \eqref{x4b} that
\begin{align}\label{x5b4}
\displaystyle  \|\na^2u_t\|_{L^2}
&\le C(\|(\rho\dot{u})_t\|_{L^2}+\|\nabla P_t\|_{L^2}+\|((\nabla \times H)\times H)_t\|_{L^2}+\|u_t\|_{L^2}) \nonumber \\
&\le C\|\n^{\frac{1}{2}}  u_{tt}\|_{L^2}+C\|\nabla  u_t\|_{L^2} +C\|\nabla  H_t\|_{L^2}+C,
\end{align}
where in the first inequality, we have utilized the $L^p$-estimate for the following elliptic system
\begin{equation}\label{x5b5}
\begin{cases}
  \mu\Delta u_t+(\lambda+\mu)\nabla\div u_t=(\rho\dot{u})_t+\nabla P_t+((\nabla \times H)\times H)_t \,\,\, &\text{in} \,\,\Omega,\\
  u_t\cdot n=0\,\,\,\text{and} \,\,\,\omega_t\times n=0\,\,&\text{on} \,\,\partial\Omega.
\end{cases}	
\end{equation}
Together with \eqref{x5b4} and \eqref{x4bb} yields
\begin{align}\label{x5b6}
\displaystyle  \int_0^T\sigma\|\nabla u_t\|_{H^1}^2dt\leq C.
\end{align}

By Sobolev's inequality, \eqref{udot}, \eqref{x2b1}, \eqref{x4b} and \eqref{x4bb}, we get for any $q\in (3,6)$,
\begin{align}\label{x5b7}
\displaystyle \|\na(\n\dot u)\|_{L^q}
&\le C \|\na \n\|_{L^q}(\|\nabla\dot{u}\|_{L^q}+\|\nabla\dot{u}\|_{L^2}+\|\nabla u\|_{L^2}^2)+C\|\na\dot u \|_{L^q}\nonumber\\
&\le C\sigma^{-\frac{1}{2}}+C\|\nabla u\|_{H^2}+C\sigma^{-\frac{1}{2}}(\sigma\|\nabla u_t\|_{H^1}^2)^{\frac{3(q-2)}{4q}}+C.
\end{align}
Integrating this inequality over $[0,T],$ by \eqref{x1b2} and \eqref{x5b6}, we have
\begin{align}\label{x5b8}
\displaystyle  \int_0^T\|\nabla(\rho\dot{u})\|_{L^q}^{p_0}dt\leq C .
\end{align}

On the other hand, \eqref{x4b} gives
\begin{align}\label{x5b9}
\displaystyle (\|\na^2 P\|_{L^q})_t & \le C \|\na u\|_{L^\infty} \|\na^2 P\|_{L^q}   +C  \|\na^2 u\|_{W^{1,q}}   \nonumber \\
& \le C (1+\|\na u\|_{L^\infty} )\|\na^2 P\|_{L^q}+C(1+ \|\na  u_t\|_{L^2})+ C\| \na(\n \dot u )\|_{L^{q}},
\end{align}
where in the last inequality we have used the  following simple fact that
\begin{align}\label{x5b10}
\displaystyle \|\na^2 u\|_{W^{1,q}}
 &\le C(1 + \|\na  u_t\|_{L^2}+ \| \na(\n\dot u )\|_{L^{q}}+\|\na^2  P\|_{L^{q}}),
\end{align}
due to \eqref{2tdu}, \eqref{3tdu}, \eqref{x1b2} and \eqref{x4b}.

Hence, applying Gronwall's inequality in \eqref{x5b9}, we deduce from \eqref{x2b1}, \eqref{x3b}  and \eqref{x5b8} that
\begin{align}\label{x5b11}
\displaystyle  \sup_{t\in[0,T]}\|\nabla^{2}P\|_{L^q}\leq C ,
\end{align}
which along with \eqref{x3b}, \eqref{x4b}, \eqref{x5b10} and \eqref{x5b8} also gives
\begin{align}\label{x5b12}
\displaystyle  \sup_{t\in[0,T]}\|P-\bar{P}\|_{W^{2,q}}+\int_0^T\|\nabla^{2}u\|_{W^{1,q}}^{p_0}dt\leq C .
\end{align}
Similarly, one has
\begin{align}
\displaystyle \sup\limits_{0\le t\le T}\|
\n-\bar{\rho}\|_{W^{2,q}} \le
 C,
\end{align}
which together with \eqref{x5b12}  gives \eqref{x5b}. The proof of Lemma \ref{lem-x5} is finished.
\end{proof}

\begin{lemma}\label{lem-x6}
There exists a positive constant $C$ such that
\begin{align}\label{x6b}
\displaystyle & \sup_{0\le t\le T}\sigma\left(\|\rho^{1/2} u_{tt}\|_{L^2}+\|H_{tt}\|_{L^2}+\|\nabla u_t\|_{H^1}+\|\nabla H_t\|_{H^1}+\|\nabla^2 H\|_{H^2}+\|\na u\|_{W^{2,q}}\right)\nonumber \\
& +\int_{0}^T\sigma^2(\|\nabla u_{tt}\|_{2}^2+\|\nabla H_{tt}\|_{2}^2)dt\le C ,
\end{align}
for any $q\in (3,6)$.
\end{lemma}

\begin{proof} Differentiating $\eqref{CMHD}_{2,3}$ with respect to $t$ twice,
multiplying them by $2u_{tt}$ and $2H_{tt}$ respectively, and integrating over $\Omega$ lead to
\begin{align}\label{x6b2}
&\quad \frac{d}{dt}\int(\rho |u_{tt}|^2+|H_{tt}|^2)dx \nonumber \\
 &\quad +2(\lambda+2\mu)\int(\div u_{tt})^2dx+2\mu\int|\omega_{tt}|^2dx+2\nu\int|\curl H_{tt}|^2dx \nonumber \\
&=-8\int_{ }  \n u^i_{tt} u\cdot\na
 u^i_{tt} dx-2\int_{ }(\n u)_t\cdot \left[\na (u_t\cdot u_{tt})+2\na
u_t\cdot u_{tt}\right]dx \nonumber \\
&\quad -2\int_{}(\n_{tt}u+2\n_tu_t)\cdot\na u\cdot u_{tt}dx-2\int (\n
u_{tt}\cdot\na u\cdot  u_{tt}-P_{tt}{\rm div}u_{tt})dx \nonumber \\
&\quad -2\int (H \cdot \nabla H-\nabla |H|^2/2)_{tt}u_{tt}dx+2 \int(H \cdot \nabla u- u \cdot \nabla H- H \div u)_{tt}H_{tt}dx \nonumber \\
&\triangleq\sum_{i=1}^6 R_i.
\end{align}
Let us estimate $R_i$ for $i=1,\cdots,6$. H\"{o}lder's inequality and \eqref{x2b1} give
\begin{align}\label{x6r1}
\displaystyle  R_1 &\le
C\|\sqrt{\rho}u_{tt}\|_{L^2}\|\na u_{tt}\|_{L^2}\| u \|_{L^\infty}
\le \de \|\na u_{tt}\|_{L^2}^2+C(\de)\|\sqrt{\rho}u_{tt}\|^2_{L^2} .
\end{align}
By \eqref{x1b2}, \eqref{x3b}, \eqref{x4b} and \eqref{x4bb}, we conclude that
\begin{align}
 R_2
&\le \de \|\na u_{tt}\|_{L^2}^2+C(\de)\|\nabla u_t\|_{L^2}^3+C(\de)\|\nabla u_t\|_{L^2}^2,\label{x6r2}\\
 R_3 
&\le \de \|\na u_{tt}\|_{L^2}^2+C(\de)\|\n_{tt}\|_{L^2}^2+C(\de)\|\nabla u_t\|_{L^2}^2,\label{x6r3}\\
 R_4 
&\le \de \|\na u_{tt}\|_{L^2}^2+C(\de)\|\sqrt{\rho}u_{tt}\|^2_{L^2}
+C(\de)\|P_{tt}\|^2_{L^2},\label{x6r4}\\
 R_5 
&\le \de \|\na u_{tt}\|_{L^2}^2+C(\de)\|H_{tt}\|^2_{L^2}+C(\de)\|\nabla H_{t}\|^3_{L^2},\label{x6r5}\\
 R_6 
 &\le \de (\|\na u_{tt}\|_{L^2}^2+\|\na H_{tt}\|_{L^2}^2)+C(\de)\|H_{tt}\|^2_{L^2} \nonumber \\
 & \quad +C(\de)(\|\nabla u_{t}\|_{L^2}\|\nabla H_{t}\|^2_{L^2}+\|\nabla u_{t}\|^2_{L^2}\|\nabla H_{t}\|^2_{L^2}),\label{x6r6}
\end{align}

Substituting these estimates of $R_i(i=1,\cdots,6)$ into \eqref{x6b2}, utilizing the fact that
\begin{align}\label{x6b3}
\displaystyle  \|\nabla u_{tt}\|_{L^2}\leq C(\|\div u_{tt}\|_{L^2}+\|\omega_{tt}\|_{L^2}), \quad \|\nabla H_{tt}\|_{L^2}\leq C\|\curl H_{tt}\|_{L^2},
\end{align}
due to Lemma \ref{lem-vn} since $u_{tt}\cdot n=0, H_{tt}\cdot n=0,$ on $\partial\Omega,$ and then choosing $\de$ small enough, we can get
\begin{align}\label{x6b4}
&\frac{d}{dt}(\|\sqrt{\rho}u_{tt}\|^2_{L^2}+\|H_{tt}\|^2_{L^2})+\|\na u_{tt}\|_{L^2}^2+\|\na H_{tt}\|_{L^2}^2 \nonumber \\
&\le C (\|\sqrt{\rho}u_{tt}\|^2_{L^2}+\|H_{tt}\|^2_{L^2}+\|\rho_{tt}\|^2_{L^2}+\|P_{tt}\|^2_{L^2}+\|\nabla u_{t}\|^3_{L^2}+\|\nabla H_{t}\|^3_{L^2}) \nonumber \\
& \quad+C(\|\nabla u_{t}\|_{L^2}\|\nabla H_{t}\|^2_{L^2}+\|\nabla u_{t}\|^2_{L^2}\|\nabla H_{t}\|^2_{L^2}),
\end{align}
which together with \eqref{x4b}, \eqref{x4bb}, and by Gronwall's inequality yields that
\begin{align}\label{x6b5}
&\quad\sup_{0\leq t\leq T}\sigma^2(\|\sqrt{\rho}u_{tt}\|^2_{L^2}+\|H_{tt}\|^2_{L^2})+\int_0^T\sigma^2(\|\na u_{tt}\|_{L^2}^2+\|\na H_{tt}\|_{L^2}^2)dt \nonumber \\
&\le C \int_0^{\sigma(T)} \sigma(\|\sqrt{\rho}u_{tt}\|^2_{L^2}+\|H_{tt}\|^2_{L^2})dt+\int_0^T \sigma^2(\|\rho_{tt}\|^2_{L^2}+\|P_{tt}\|^2_{L^2}+\|\nabla u_{t}\|^2_{L^2})dt \nonumber \\
&\quad +\int_0^T \sigma^2(\|\nabla u_{t}\|^3_{L^2}+\|\nabla H_{t}\|^3_{L^2}+\|\nabla u_{t}\|_{L^2}\|\nabla H_{t}\|^2_{L^2}+\|\nabla u_{t}\|^2_{L^2}\|\nabla H_{t}\|^2_{L^2})dt \nonumber \\
 &\leq C.
\end{align}
Furthermore, it follows from \eqref{3tdh}, \eqref{3tdu}, \eqref{x5b4} and \eqref{x4bb} that
\begin{align}\label{x6b6}
\displaystyle &\quad \sup_{0\le t\le T}(\sigma\|\nabla^2 u_t\|_{L^2}+\sigma\|\nabla^2 H_t\|_{L^2}) \nonumber \\
 &\leq C \sigma(1+\|\rho^{1/2}u_{tt}\|_{L^2}+\|H_{tt}\|_{L^2}+\|\nabla u_t\|_{L^2}+\|\nabla H_t\|_{L^2}) \leq C.
\end{align}
Finally, we deduce from \eqref{x4bb}, \eqref{x5b}, \eqref{x5bb}, \eqref{x5h1}, \eqref{x5b7}, \eqref{x5b10}, \eqref{x6b5} and \eqref{x6b6} that
\begin{align}\label{x6b7}
&\displaystyle \quad \sigma\|\na^2 u\|_{W^{1,q}}
 \le C\sigma (1+\|\na  u_t\|_{L^2}+\|\na  H_t\|_{L^2}+\| \na(\n
\dot u )\|_{L^{q}}+\|\na^2  P\|_{L^{q}})\nonumber \\
& \le C(1+ \sigma\|\na u\|_{H^2}+\sigma^{\frac{1}{2}}(\sigma\|\na u_t\|_{H^1}^2)^{\frac{3(q-2)}{4q}})
\le C+C\sigma^{\frac{1}{2}}(\sigma^{-1})^{\frac{3(q-2)}{4q}} \le C ,
\end{align}
and
\begin{align}
\displaystyle  \sigma \|\nabla^2 H\|_{H^2}\leq C \sigma(1+\|\nabla H_t\|_{H^1}+\|\nabla u\|_{H^2}\|\nabla H\|_{H^2})\leq C,
\end{align}
together with \eqref{x6b5} and \eqref{x6b6} yields \eqref{x6b} and this completes the proof of Lemma \ref{lem-x6}.
\end{proof}

\section{Proof of  Theorem  \ref{th1}}\label{se5}

With all the a priori estimates in Section \ref{se3} and Section \ref{se4} at hand, we are going to  prove the main result of the paper in this section.

{\it Proof of Theorem \ref{th1}.}
 By Lemma \ref{lem-local}, there exists a $T_*>0$ such that the  system \eqref{CMHD}-\eqref{boundary} has a unique classical solution $(\rho,u,H)$ on $\Omega\times
(0,T_*]$. One may use the a priori estimates, Proposition \ref{pr1} and Lemmas \ref{lem-x3}-\ref{lem-x6} to extend the classical
solution $(\rho,u,H)$ globally in time.

First, by the definition of \eqref{As1}-\eqref{As5}, the assumption of the initial data \eqref{dt2} and \eqref{ba35}, one immediately checks that
\begin{align}\label{pf1}
\displaystyle  0\leq\rho_0\leq \hat{\rho},\,\, A_1(0)+A_2(0)=0, \,\,  A_3(0)\leq C_0^{\delta_0},\,\,A_4(0)+A_5(0)\leq C_0^{\delta_0}.
\end{align}
Therefore, there exists a $T_1\in(0,T_*]$ such that
\begin{equation}\label{pf2}
\begin{cases}
0\leq\rho_0\leq2\hat{\rho}, \,\, A_1(T)+A_2(T)\leq 2C_0^{\frac{1}{2}}, \\
A_3(T)\leq 2C_0^{\delta_0}, \,\, A_4(\sigma(T))+A_5(\sigma(T))\leq 2C_0^{\delta_0},
\end{cases}	
\end{equation}
hold for $T=T_1.$
Next, we set
\begin{align}\label{pf3}
\displaystyle  T^*=\sup\{T\,|\,{\rm \eqref{pf2} \ holds}\}.
\end{align}
Then $T^*\geq T_1>0$. Hence, for any $0<\tau<T\leq T^*$
with $T$ finite, it follows from Lemmas \ref{lem-x1}-\ref{lem-x6}
that
\begin{equation}\label{pf4}
\begin{cases}
\rho-\bar{\rho} \in C([0,T]; H^2 \cap W^{2,q}), \\
(u, H) \in C([\tau ,T]; H^2),\quad ( \nabla u_t, \nabla H_t) \in C([\tau ,T]; L^q);
\end{cases}	
\end{equation}
where one has taken advantage of the standard embedding
$$L^\infty(\tau ,T;H^1)\cap H^1(\tau ,T;H^{-1})\hookrightarrow
C\left([\tau ,T];L^q\right),\quad\mbox{ for any } q\in [2,6).  $$
Due to \eqref{x3b}, \eqref{x4bb}, \eqref{x6b} and $\eqref{CMHD}_1$,
we obtain
\begin{align*}
&\quad\int_{\tau}^T \left|\left(\int\n|u_t|^2dx\right)_t\right|dt
\le\int_{\tau}^T\left(\|  \n_t  |u_t|^2 \|_{L^1}+2\|  \n  u_t\cdot u_{tt} \|_{L^1}\right)dt\\
&\le C\int_{\tau}^T\left( \| \n^{\frac{1}{2}} |u_t|^2 \|_{L^2}\|\na u\|_{L^\infty}+\|  u\|_{L^6}\|\na\n\|_{L^2} \|u_t  \|^2_{L^6}+  \|\rho^{\frac{1}{2}}u_{tt} \|_{L^2}\right)dt\le C,
\end{align*}
which together with \eqref{pf4} yields
\begin{align}\label{pf5}
\displaystyle  \rho^{1/2}u_t, \quad\rho^{1/2}\dot u \in C([\tau,T];L^2).
\end{align}
Finally, we claim that
\begin{align}\label{pf6}
 \displaystyle  T^*=\infty.
 \end{align}
Otherwise, $T^*<\infty$. Then by Proposition \ref{pr1}, it holds that
\begin{equation}\label{pf7}
\begin{cases}
0\leq\rho\leq\frac{7}{4}\hat{\rho},\,\,A_1(T^*)+A_2(T^*)\leq C_0^{\frac{1}{2}},\\
A_3(T^*)\leq C_0^{\delta_0}, \,\, A_4(\sigma(T^*))+A_5(\sigma(T^*))\leq C_0^{\delta_0},	
\end{cases}	
\end{equation}
It follows from Lemmas \ref{lem-x5}, \ref{lem-x6} and \eqref{pf5} that $(\rho(x,T^*),u(x,T^*), H(x,T^*))$ satisfies the initial data condition \eqref{dt1}-\eqref{dt2}, \eqref{dt3}, where  $g(x)\triangleq\sqrt{\rho}\dot u(x, T^*),\,\,x\in \Omega.$
Thus, Lemma \ref{lem-local} implies that there exists some $T^{**}>T^*$ such that \eqref{pf2} holds for $T=T^{**}$, which contradicts the definition of $ T^*.$ As a result, \eqref{pf6} holds.
By Lemmas \ref{lem-local} and \ref{lem-x1}-\ref{lem-x6}, it indicates that $(\rho,u,H)$ is in fact the unique classical solution defined on $\Omega\times(0,T]$ for any  $0<T<T^*=\infty.$

Finally, with \eqref{m2}, \eqref{tdu1}, \eqref{tdh-2}, \eqref{grho}, \eqref{p-time}, \eqref{tdh-3} and \eqref{I01} at hand, \eqref{esti-t}  can be obtained in similar arguments as used in \cite{cl2019}, and we omit the details.
The proof of Theorem \ref{th1} is finished.   \endproof





\appendix
\section{Some basic theories and lemmas}\label{appendix-a}
In this appendix, we review some elementary inequalities and important lemmas that are used extensively in this paper.

First, we recall the well-known Gagliardo-Nirenberg inequality (see \cite{nir}).
\begin{lemma}[Gagliardo-Nirenberg]\label{lem-gn}
Assume that $\Omega$ is a bounded Lipschitz domain in $\r^3$. For  $p\in [2,6],\,q\in(1,\infty), $ and
$ r\in  (3,\infty),$ there exist two generic
 constants
$C_1,\,\,C_2>0$ which may depend  on $p$, $q$ and $r$ such that for any  $f\in H^1({\O }) $
and $g\in  L^q(\O )\cap D^{1,r}(\O), $
\be\label{g1}\|f\|_{L^p(\O)}\le C_1 \|f\|_{L^2}^{\frac{6-p}{2p}}\|\na
f\|_{L^2}^{\frac{3p-6}{2p}}+C_2\|f\|_{L^2} ,\ee
\be\label{g2}\|g\|_{C\left(\ol{\O }\right)} \le C_1
\|g\|_{L^q}^{q(r-3)/(3r+q(r-3))}\|\na g\|_{L^r}^{3r/(3r+q(r-3))} + C_2\|g\|_{L^2}.
\ee
Moreover, if $f\cdot n|_{\partial\Omega}=0$ ,$g\cdot n|_{\partial\Omega}=0$ , then the constant $C_2=0.$
\end{lemma}

In order to get the uniform (in time) upper bound of the density $\n,$ we need the following Zlotnik  inequality in \cite{z2000}.
\begin{lemma}\label{lem-z}
Suppose the function $y$ satisfy
\bnn y'(t)= g(y)+b'(t) \mbox{  on  } [0,T] ,\quad y(0)=y^0, \enn
with $ g\in C(R)$ and $y, b\in W^{1,1}(0,T).$ If $g(\infty)=-\infty$
and \be\label{a100} b(t_2) -b(t_1) \le N_0 +N_1(t_2-t_1)\ee for all
$0\le t_1<t_2\le T$
  with some $N_0\ge 0$ and $N_1\ge 0,$ then
\bnn y(t)\le \max\left\{y^0,\overline{\zeta} \right\}+N_0<\infty
\mbox{ on
 } [0,T],
\enn
where $\overline{\zeta} $ is a constant such
that \be\label{a101} g(\zeta)\le -N_1 \quad\mbox{ for }\quad \zeta\ge \overline{\zeta}.\ee
\end{lemma}

Consider the Lam\'{e}'s system
\be\label{lame1}\begin{cases}
-\mu\Delta u-(\lambda+\mu)\nabla\div u=f \,\, &in~ \Omega, \\
u\cdot n=0\,\,\text{and}\,\,\curl u\times n=0\,\,&on\,\,\partial\Omega,
\end{cases} \ee
Then, the following estimate is standard (see \cite{adn}).
\begin{lemma}  \label{lem-lame}
For the Lam\'{e}'s equation \eqref{lame1}, one has

(1) If $f\in W^{k,q}$ for some $q\in(1,\infty),\,\, k\geq0,$ then there exists a unique solution $u\in W^{k+2,q},$ such that
$$\|u\|_{W^{k+2,q}}\leq C(\|f\|_{W^{k,q}}+\|u\|_{L^q});$$

(2) If $f=\nabla g$ and $g\in W^{k,q}$ for some $q\geq1,\,\,k\geq0,$ then there exists a unique weak solution $u\in W^{k+1,q},$ such that
$$\|u\|_{W^{k+1,q}}\leq C(\|g\|_{W^{k,q}}+\|u\|_{L^q}).$$
\end{lemma}

The following two lemmas are given in Theorem 3.2 in \cite{vww} and Propositions 2.6-2.9 in \cite{ar2014}.
\begin{lemma}   \label{lem-vn}
Let $k\geq0$ be a integer, $1<q<+\infty$, and assume that $\Omega$ is a simply connected bounded domain in $\r^3$ with $C^{k+1,1}$ boundary $\partial\Omega$. Then for $v\in W^{k+1,q}$ with $v\cdot n=0$ on $\partial\Omega$, it holds that
$$\|v\|_{W^{k+1,q}}\leq C(\|\div v\|_{W^{k,q}}+\|\curl v\|_{W^{k,q}}).$$
In particular, for $k=0$, we have
$$\|\nabla v\|_{L^q}\leq C(\|\div v\|_{L^q}+\|\curl v\|_{L^q}).$$
\end{lemma}
\begin{lemma}   \label{lem-curl}
Let $k\geq0$ be a integer, $1<q<+\infty$. Suppose that $\Omega$ is a bounded domain in $\r^3$ and its $C^{k+1,1}$ boundary $\partial\Omega$ only has a finite number of 2-dimensional connected components. Then for $v\in W^{k+1,q}$ with $v\times n=0$ on $\partial\Omega$, we have
$$\|v\|_{W^{k+1,q}}\leq C(\|\div v\|_{W^{k,q}}+\|\curl v\|_{W^{k,q}}+\|v\|_{L^q}).$$
In particular, if  $\Omega$ has no holes, then
$$\|v\|_{W^{k+1,q}}\leq C(\|\div v\|_{W^{k,q}}+\|\curl v\|_{W^{k,q}}).$$
\end{lemma}

Next, similar to \cite{bkm,hlx,hl1}, we need a Beale-Kato-Majda type inequality with respect to the slip boundary condition \eqref{navier-b} which is given in \cite{cl2019}.
\begin{lemma}\label{lem-bkm}
For $3<q<\infty$, assume that $u\cdot n=0$ and $\curl u\times n=0$ on $\partial\Omega$, $ u\in W^{2,q}$, then there is a constant  $C=C(q,\Omega)$ such that  the following estimate holds
\bnn\ba
\|\na u\|_{L^\infty}\le C\left(\|{\rm div}u\|_{L^\infty}+\|\curl u\|_{L^\infty} \right)\ln(e+\|\na^2u\|_{L^q})+C\|\na u\|_{L^2} +C .
\ea\enn
\end{lemma}

Finally, we consider the problem
\begin{equation}\label{divf}
\begin{cases}
{\rm div}v=f \,\,\,\,  in \,\,\Omega, \\
v=0\,\,\,\text{on}\,\,\,{\partial\Omega}.
\end{cases}
\end{equation}
One has the following conclusion (see \cite{GPG}, Theorem III.3.1).
\begin{lemma} \label{lem-divf}
There exists a linear operator operator $\mathcal{B} = [\mathcal{B}_1 , \mathcal{B}_2 , \mathcal{B}_3 ]$ enjoying
the properties:

1) $$\mathcal{B}:\{f\in L^p(\O)|\int_\O fdx=0\}\mapsto (W^{1,p}_0(\O))^3$$ is a bounded linear operator, that is,
\be \|\mathcal{B}[f]\|_{W^{1,p}_0(\O)}\le C(p)\|f\|_{L^p(\O)}, \mbox{ for any }p\in (1,\infty),\ee

2) The function $v = \mathcal{B}[f]$ solve the problem \eqref{divf}.

3) if $f$ can be written in the form $f = \div  g$ for a certain $g\in L^r(\O), g\cdot n|_{\pa\O}=0,$  then
\be \|\mathcal{B}[f]\|_{L^{r}(\O)}\le C(r)\|g\|_{L^r(\O)}, \mbox{ for any }r  \in (1,\infty).\ee
\end{lemma}


\section*{Acknowledgements} This research was partially supported by National Natural Sciences Foundation of China No. 11671027, 11901025, 11971020, 11971217.








\begin{thebibliography}{99}
\bibitem{adn}
S. Agmon, A. Dougllis, and L. Nirenberg, Estimates near the boundary for solutions of elliptic partial differential equations satisfying general boundary conditions, II., \textit{Commun. Pure Appl. Math.}, \textbf{17}, 35-92, (1964).
\bibitem{ar2014}
J. Aramaki, $L^p$ theory of the div-curl system, \textit{Int. J. Math. Anal.}, \textbf{8}(6), 259-271, (2014).
\bibitem{bkm}
J.T. Beale, T. Kato, and A. Majda, Remarks on the breakdown of smooth solutions for the 3-D Euler equations, \textit{Comm. Math. Phys.}, \textbf{94}, 61-66, (1984).
\bibitem{hbdv3}
H. Beir\~{a}o da Veiga, On the regularity of flows with Ladyzhenskaya shear-dependent viscosity and slip or nonslip boundary conditions, \textit{Comm. Pure Appl. Math.}, \textbf{58}(4), 552-577, (2005).
\bibitem{bl}
J. Bergh, J. L\"{o}fstr\"{o}m, Interpolation spaces: An introduction, Springer-Verlag, Berlin-Heidelberg-New York, 1976.
\bibitem{cl2019}
G.C. Cai and J. Li, Existence and exponential growth of  global classical solutions to the  compressible Navier-Stokes equations with slip boundary conditions in 3D bounded domains. arXiv:2102.06348.
\bibitem{cw2002}
G.Q. Chen and D. Wang, Global solutions of nonlinear magnetohydrodynamics with large initial data, \textit{J. Differ. Eqs.}, \textbf{182}, 344-76, (2002).
\bibitem{cw2003}
G.Q. Chen and D. Wang, Existence and continuous dependence of large solutions for the magnetohydrodynamic equations, \textit{Z. Angew. Math. Phys.}, \textbf{54}, 608-632, (2003).
\bibitem{cf1988}
P. Constantin,  C. Foias, Navier-Stokes equations, Chicago Lectures in Mathematics, University of Chicago Press, Chicago, 1988.
\bibitem{djj2013}
C. Dou, S. Jiang, and Q. Ju, Global existence and the low Mach number limit for the compressible magnetohydrodynamic equations in a bounded domain with perfectly conducting boundary, \textit{Z. Angew. Math. Phys.}, \textbf{64}(6), 1661-1678, (2013).
\bibitem{df2006}
B. Ducomet and E. Feireisl, The equations of magnetohydrodynamics: On the interaction between matter and radiation in the evolution of gaseous stars, \textit{Comm. Math. Phys.}, \textbf{266}, 595-629, (2006).
\bibitem{fjn2007}
J. Fan, S. Jiang, and G. Nakamura, Vanishing shear viscosity limit in the magnetohydrodynamic equations, \textit{Comm. Math. Phys.}, \textbf{270}, 691-708, (2007).
\bibitem{fy2008}
J. Fan, W. Yu, Global variational solutions to the compressible magnetohydrodynamic equations, \textit{Nonlinear Anal.}, \textbf{69}, 3637-3660, (2008).
\bibitem{fy2009}
J. Fan, W. Yu, Strong solution to the compressible MHD equations with vacuum, \textit{Nonlinear Anal. Real World Appl.}, \textbf{10} , 392-409, (2009).
\bibitem{GPG}
G.P. Galdi, An introduction to the mathematical theory of the Navier-Stokes equations, Springer, New-York, 1994.
\bibitem{Ho3}
D. Hoff, Compressible flow in a half-space with Navier boundary conditions, \textit{J. Math. Fluid Mech.}, \textbf{7}(3), 315-338, (2005).
\bibitem{hhpz}
G. Hong, X. Hou, H. Peng, and  C. Zhu, Global existence for a class of large solutions to three-dimensional compressible magnetohydrodynamic equations with vacuum, \textit{SIAM  J. Math. Anal.}, \textbf{49}(4), 2409-2441, (2017).
\bibitem{hw2008}
X. Hu and D. Wang, Global solutions to the three-dimensional full compressible magnetohydrodynamic flows, \textit{Commun. Math. Phys.}, \textbf{283}, 255-284, (2008).
\bibitem{hw2008-1}
X. Hu and D. Wang, Compactness of weak solutions to the three-dimensional compressible magnetohydrodynamic equations, \textit{J. Differ. Eqs.}, \textbf{245}, 2176-2198, (2008).
\bibitem{hw2009}
X. Hu and D. Wang, Low mach number limit of viscous compressible magnetohydrodynamic fows, \textit{SIAM J. Math. Anal.}, \textbf{41}, 1272-1294, (2009).
\bibitem{hw2010}
X. Hu and D. Wang, Global existence and large-time behavior of solutions to the three dimensional equations of compressible magnetohydrodynamic flows, \textit{Arch. Ration. Mech. Anal.}, \textbf{197}, 203-238, (2010).
\bibitem{hl1}
X.D. Huang, J. Li, Serrin-type blowup criterion for viscous, compressible, and heat conducting Navier-Stokes and magnetohydrodynamic flows, \textit{Comm. Math. Phys.}, \textbf{324}, 147-171, (2013).
\bibitem{hlx}
X.D. Huang, J. Li, and Z.P. Xin, Serrin type criterion for the three-dimensional compressible flows, \textit{SIAM J. Math. Anal.}, \textbf{43}(4), 1872-1886, (2011).
\bibitem{hlx1}
X.D. Huang, J. Li, and Z.P. Xin, Global well-posedness of classical solutions with large oscillations and vacuum to the three-dimensional isentropic compressible Navier-Stokes equations, \textit{Comm. Pure Appl. Math.}, \textbf{65}, 549-585, (2012).
\bibitem{Itt2}
S. Itoh, N. Tanaka, and A. Tani, The initial value problem for the Navier-Stokes equations with general slip boundary condition in H\"{o}lder spaces, \textit{J. Math. Fluid Mech.}, \textbf{5}(3), 275-301, (2003).
\bibitem{k1984}
S. Kawashima, Smooth global solutions for two-dimensional equations of electromagneto-fluid dynamics, \textit{Japan J. Appl. Math.}, \textbf{1}, 207-222, (1984).
\bibitem{lxz2013}
H. Li, X. Xu, and J. Zhang, Global classical solutions to the 3D compressible magnetohydrodynamic equations with large oscillations and vacuum, \textit{SIAM J. Math. Anal.}, \textbf{43}, 1356-1387, (2013).
\bibitem{lx}
J. Li and Z. Xin, Some uniform estimates and blowup behavior of global strong solutions to the Stokes approximation equations for two-dimensional compressible flows, \textit{J. Differ. Eqs.}, \textbf{221}(2), 275-308, (2006).
\bibitem{jx01}
J. Li and Z. Xin, Global Existence of Regular Solutions with Large Oscillations and Vacuum. In: Giga Y., Novotny A. (eds) Handbook of Mathematical Analysis in Mechanics of Viscous Fluids, Springer, 2016.
\bibitem{liu2015}
Y. Liu, Global classical solutions of 3D isentropic compressible MHD with general initial data, \textit{Z. Angew. Math. Phys.}, \textbf{66}(4), 1777-1797, (2015).
\bibitem{lvh2015}
B. Lv and  B. Huang, On strong solutions to the cauchy problem of the two-dimensional compressible magnetohydrodynamic equations with vacuum, \textit{Nonlinearity}, \textbf{28}(2), 509-530, (2015).
\bibitem{lsx2016}
B. Lv, X. Shi, and X. Xu, Global existence and large time asymptotic behavior of strong solutions to the 2-D compressible magnetohydrodynamic equations with vacuum, \textit{Indiana Univ. Math.J.}, \textbf{65}, 925-975, (2016).
\bibitem{Nclm1}
C.L.M.H. Navier, M\'{e}moire sur les lois du mouvement des fluides. \textit{M\'{e}m. Acad. Re. Sci.}, Paris 6, 389-416, (1823).
\bibitem{nir}
L. Nirenberg, On elliptic partial differential equations. \textit{Ann. Scuola Norm. Sup. Pisa}, \textbf{13}, 115-162, (1959).
\bibitem{ns2004}
A. Novotny, I. Straskraba, Introduction to the Mathematical Theory of Compressible Flow, Oxford Lecture Ser. Math. Appl., Oxford Univ. Press, Oxford, 2004.
\bibitem{tg2016}
T. Tang and H. Gao, Strong solutions to 3d compressible magnetohydrodynamic equations with navier‐slip condition, \textit{Math. Methods Appl. Sci.}, \textbf{39}(10), 2768-2782, (2016).
\bibitem{vk1972}
A.I. Volpert and S.I. Hudjaev, On the Cauchy problem for composite systems of nonlinear equations, \textit{Mat. Sb.}, \textbf{87}, 504-528, (1972).
\bibitem{vww}
W. von Wahl, Estimating $\nabla u$ by $\div u$ and $\curl u$, \textit{Math. Methods Appl. Sci.}, \textbf{15}, 123-143, (1992).
\bibitem{wdh2003}
D. Wang, Large solutions to the initial-boundary value problem for planar magnetohydrodynamics, \textit{SIAM J. Appl. Math.}, \textbf{63}, 1424-1441, (2003).
\bibitem{xh2017}
S. Xi, X. Hao, Existence for the compressible magnetohydrodynamic equations with vacuum, \textit{J. Math Anal Appl.}, \textbf{453}, 410-433, (2017).
\bibitem{xx2007}
Y. Xiao and Z. Xin, On the vanishing viscosity limit for the 3D Navier-Stokes equations with a slip boundary condition. \textit{Comm. Pure Appl. Math.}, \textbf{60}(7), 1027-1055, (2007).
\bibitem{xxw2009}
Y. Xiao,  Z. Xin, and J. Wu, Vanishing viscosity limit for the 3D
magnetohydrodynamic system with a slip boundary condition, \textit{J. Funct. Anal.}, \textbf{257}, 3375-3394, (2009).
\bibitem{z1998}
W.M. Zajaczkowski, On nonstationary motion of a compressible baratropic viscous fluids with boundary slip condition, \textit{J. Appl. Anal.}, \textbf{4}, 167-204,  (1998).
\bibitem{zjx2009}
J.W. Zhang, S. Jiang, and F. Xie, Global weak solutions of an initial boundary value problem for screw pinches in plasma physics, \textit{Math. Models Methods Appl. Sci.}, \textbf{19}, 833-875, (2009).
\bibitem{zz2010}
J.W. Zhang and J.N. Zhao, Some decay estimates of solutions for the 3-D compressible isentropic magnetohydrodynamics, \textit{Commun. Math. Sci.}, \textbf{8}, 835-850, (2010).
\bibitem{zhu2015}
S. Zhu, On classical solutions of the compressible magnetohydrodynamic equations with vacuum, \textit{SIAM J. Math. Anal.}, \textbf{47}(4), 2722-2753, (2015).
\bibitem{zg2001}
Y.X. Zhu, S. Granick, Limits of the hydrodynamic no-slip boundary condition, \textit{Phys. Rev. Lett.}, \textbf{88}, 054504, (2001).
\bibitem{z2000}
A.A. Zlotnik, Uniform estimates and stabilization of symmetric solutions of a system of quasilinear equations, \textit{Diff. Eqs.}, \textbf{36}, 701-716, (2000).

\end{thebibliography}
\end{document}